\documentclass[review,onefignum,onetabnum]{siamonline220329}

\usepackage[utf8]{inputenc}
\usepackage{geometry}[margins=1in]
\usepackage{amsmath,amssymb,amsfonts}
\usepackage{algorithmic}
\usepackage{algorithm}

\usepackage{subcaption}
\usepackage{hyperref}
\usepackage{multirow}
\usepackage{textcomp}
\usepackage{float}
\usepackage{graphicx}

\usepackage{soul}

\usepackage{bm}
\usepackage{epstopdf}
\usepackage{enumerate}
\usepackage{lipsum}
\usepackage{placeins}

\usepackage{xfrac}
\usepackage{lineno}
\nolinenumbers
\usepackage[normalem]{ulem}
\usepackage{cite}

\newtheorem*{notation}{Notation}
\newtheorem*{prop:reverse}{Proposition \ref{prop:generalized_reverse_sobolev}}
\newtheorem*{def:MGram}{Definition \ref{def:MGram}}
\definecolor{MYORANGE}{rgb}{1,0.5,0}

\newsiamthm{remark}{Remark}
\newsiamremark{hypothesis}{Hypothesis}
\crefname{hypothesis}{Hypothesis}{Hypotheses}
\newsiamthm{claim}{Claim}

\usepackage{amsopn}

\newcommand{\ip}[2]{\left \langle #1,\ #2\right \rangle} 

\DeclareMathOperator*{\argmin}{argmin}

\def\projK{\mathcal{P}_K}
\def\projJ{\mathcal{P}_J}
\def\op{\textup{op}}

\headers{weak-SINDy Surrogate Models}{Russo, Laiu}

\title{Convergence of weak-SINDy Surrogate Models\thanks{This manuscript has been authored by UT-Battelle, LLC, under contract DE-AC05-00OR22725 with the US Department of Energy (DOE). The US government retains and the publisher, by accepting the work for publication, acknowledges that the US government retains a non-exclusive, paid-up, irrevocable, world-wide license to publish or reproduce the submitted manuscript version of this work, or allow others to do so, for US government purposes. DOE will provide public access to these results of federally sponsored research in accordance with the DOE Public Access Plan (http://energy.gov/downloads/doe-public-access-plan).
\funding{The work of B.~Russo was sponsored by the Laboratory Directed Research and Development Program of Oak Ridge National Laboratory, managed by UT-Battelle, LLC for the US Department of Energy under contract DE-AC05-00OR22725. The work of M.~P.~Laiu was supported by the Office of Advanced Scientific Computing Research and performed at the Oak Ridge National Laboratory, which is managed by UT-Battelle, LLC under Contract No. DE-AC05-00OR22725.
}}}

\author{Benjamin P. Russo\thanks{Mathematics in Computation Section, Computer Science and Mathematics Division, Oak Ridge National Laboratory, Oak Ridge, TN 37831, USA.
  (\email{russobp@ornl.gov}, \email{laiump@ornl.gov}).}
\and M. Paul Laiu\footnotemark[2]}

\begin{document}

\maketitle

\begin{abstract}
In this paper, we give an in-depth error analysis for surrogate models generated by a variant of the Sparse Identification of Nonlinear Dynamics (SINDy) method. 
We start with an overview of a variety of nonlinear system identification techniques, namely, SINDy, weak-SINDy, and the occupation kernel method.  Under the assumption that the dynamics are a finite linear combination of a set of basis functions, these methods establish a linear system to recover coefficients. We illuminate the structural similarities between these techniques and establish a projection property for the weak-SINDy technique.  Following the overview, we analyze the error of surrogate models generated by a simplified version of weak-SINDy. In particular, under the assumption of boundedness of a composition operator given by the solution, we show that (i) the surrogate dynamics converges towards the true dynamics and (ii) the solution of the surrogate model is reasonably close to the true solution. Finally, as an application, we discuss the use of a combination of weak-SINDy surrogate modeling and proper orthogonal decomposition (POD) to build a surrogate model for partial differential equations (PDEs).
\end{abstract}



\section{Introduction}
Dynamical systems have an important position in science and engineering as a way to describe the evolution of a system in a quantifiable way. Occasionally, a system may have some unknown parameters or be completely unknown. In this setting, system identification techniques are leveraged to identify the system. If the dynamics are suspected to be linear, then a variety of techniques exist to identify the dynamics via the Fourier transform or Laplace transform \cite{ljung1998system}. However, when the dynamics are expected to be nonlinear, these techniques do not apply. To combat this, a variety of nonlinear identification techniques have been developed, including the SINDy technique \cite{brunton2016discovering}. 

SINDy (Sparse Identification of Nonlinear Dynamics) essentially makes two fundamental assumptions. One is that the dynamics can be represented as a finite linear combination of functions and the other is that this representation is sparse. Under these assumptions, the problem of system identification is reduced to a parameter estimation/identification problem. As a data-driven technique, SINDy uses pointwise data from the system to create a linear system which describes the dynamics at the chosen points. This linear system is then solved using sparsity enforcing techniques, such as LASSO \cite{tibshirani1996regression, hastie2009elements} and sequentially thresholded least-squares \cite{brunton2016discovering, convergence_of_SINDy}. 

One drawback to the SINDy technique is that the time derivative of system states is needed to build the linear system. These time derivatives often need to be estimated from the system states, which is untenable in the presence of noise. A common approach to address this issue is to apply low-pass filters to reduce the noise. However, it is known \cite{filters} that low-pass filtering does not apply in certain situations such as handling EEG data \cite{EEG} in neuroscience or turbulence data \cite{smagorinsky1963general} in computational fluid dynamics, where high-frequency data is present. 
Hence, a technique which is robust to noise is preferable in these situations. Integral formulations of SINDy, such as the weak-SINDy technique developed by Messenger and Bortz \cite{messenger2021weak, messenger2021weakpde} and the Occupation Kernel techniques developed by Rosenfeld et~al.~\cite{rosenfeld2019occupation, SCC.Rosenfeld.Kamalapurkar.ea2019a}, are naturally robust to noise as the derivative data is accessed in a circuitous fashion via integral equations. Weak-SINDy utilizes test functions to access the derivative data via integration by parts. Similarly, the Occupation Kernel technique uses test functions to access the derivative data via the fundamental theorem of calculus. 

As stated above, a fundamental assumption of these nonlinear system identification techniques is that the dynamics are describable by a \emph{chosen} finite basis of functions. However, if the dynamics are truly unknown, a practitioner is left to guess or use domain knowledge to choose a reasonable basis. Additionally, one may not be interested in identifying a system but in reducing a known but computationally expensive system to an inexpensive model which retains some properties of the original system. As the only required inputs for these techniques are choices of basis and system data, it is clear that these techniques will return \emph{some} model, but the relevancy of this model to the underlying system is unclear.

This paper asserts in the case of weak-SINDy that if a reasonable basis, composed of functions known to be dense in an underlying Hilbert space, is chosen, then the returned system is a good approximation of the original and has solutions reasonably close to the original. Section~\ref{sec:A general formulation of SINDy-type techniques} gives an overview of SINDy-type techniques, illuminates a projection property of weak-SINDy, and exemplifies the structural similarities between these SINDy-type techniques giving possible avenues for extensions of this work. In Section~\ref{sec:ODE case}, we give an in-depth error analysis of the surrogate models generated by a simplified version of weak-SINDy. In particular, in the scalar ordinary differential equation (ODE) case, we show that (i) the surrogate dynamics converges towards the true dynamics and (ii) the solution of the surrogate model is reasonably close to the true solution, under the assumption of a bounded composition operator. In Section~\ref{sec:system of odes}, this error analysis is extended to systems of ODEs, and the results are applied to identify a surrogate ODE system over coordinates defined by proper orthogonal decomposition (POD) for solutions to partial differential equations (PDEs), similar to the approach taken in \cite{qian2022reduced}. Section~\ref{sec:Numerical Experiments} contains a variety of numerical examples exemplifying the main results of this paper.

\section{A general formulation of SINDy-type techniques}\label{sec:A general formulation of SINDy-type techniques}
In this section, we review three SINDy-type techniques (Section~\ref{sec:SINDy_review}), provide a unified formulation for these techniques (Section~\ref{sec:general_formulation}), and analyze the projection property of the weak-SINDy method (Section~\ref{sec:projection}).

\subsection{Review of three SINDy-type techniques}
\label{sec:SINDy_review}
We begin with an overview of three SINDy-type techniques, including the original SINDy method \cite{brunton2016discovering}. These system identification techniques aim to identify an underlying dynamical system
$\dot{x}= f(x(t))$ from measured or simulated solution data $x$ in a time interval $[a,b]$. Here $\dot{x}$ denotes the time derivative of $x$, which will be used throughout this paper.
An overarching assumption in these techniques is that $f$ is assumed to be a linear combination of a finite set of functions $\{\varphi_j\}_{j=0}^J$, i.e.,
\[\textstyle f(\cdot) = \sum_{j=0}^J w_j\varphi_j(\cdot), \quad w_j \in \mathbb{R}, \quad j=0,\dots,J.\]
The SINDy-type technique then estimate the weights $\{w_j\}_{j=0}^J$ from data that contains (part of) the solution $x$ to the underlying system.

\paragraph{SINDy}
Given $\{x(t_k)\}_{k=0}^K$ at time steps $\{t_k\}_{k=0}^K$, the original SINDy method \cite{brunton2016discovering} seeks weights $\{w_j\}_{j=0}^J$ that satisfy
\[\textstyle\dot{x}(t_k) = \sum_{j=0}^J w_j \varphi_j(x(t_k)),\quad k=0,\dots,K,\]
which is then formulated as a linear system
\begin{equation}\label{eq:SINDy}
    \bm{b} = \bm{G}\bm{w} \quad\text{with entries}\quad
 \bm{b}_k = \dot{x}(t_k),\quad
 \bm{G}_{k,j} = \varphi_j(x(t_k)),\quad\text{and}\quad
 \bm{w}_j = w_j.
\end{equation}

Additionally, the SINDy method makes a second fundamental assumption that the underlying dynamics are sparsely represented in the correct choice of basis. Hence the above linear system can be solved using a sparsity enforcing solver such as LASSO. 

Astute readers will realize that the vector $\bm{b}$ is populated with the derivative data $\dot{x}(t)$ at the prescribed time-steps. Ideally, the derivative data is measured, but in many cases it must be estimated using a variety of techniques. In the presence of measurement noise, this becomes increasingly harder to do. To circumvent this obstruction, integral formulations of SINDy have been created. Two techniques to the author's knowledge are Messenger and Bortz's weak formulation \cite{messenger2021weak, messenger2021weakpde} and Rosenfeld et~al.'s operator theoretic technique \cite{SCC.Rosenfeld.Kamalapurkar.ea2019a, rosenfeld2019occupation}, which we introduce below.  

\paragraph{weak-SINDy}
The weak-SINDy method proposed in \cite{messenger2021weak, messenger2021weakpde} introduces a class of \emph{test} functions $\{\psi_k\}_{k=0}^K$ in addition the basis functions $\{\varphi_j\}_{j=0}^J$ in the SINDy method. Throughout this paper, we refer to $\{\varphi_j\}_{j=0}^J$ as a \emph{projection basis} and $\{\psi_k\}_{k=0}^K$ as a \emph{test function basis}. 

Given $x(t)$, $\forall t\in[a,b]$, the weak-SINDy method finds the weights $\{w_j\}_{j=0}^J$ by solving
\begin{equation}\label{eq:wSINDy}
\textstyle
    \ip{\dot{x}}{\psi_k} = \ip{\sum_{j=0}^J w_j\varphi_j(x)}{\psi_k},\quad k=0,\dots,K,
\end{equation}   
where $\ip{\cdot}{\cdot}$ denotes the $L^2$ inner product on $[a,b]$.
In matrix form, Eq.~\eqref{eq:wSINDy} is written as 
\begin{equation}
\label{weak-SINDy_form}
\bm{b} = \bm{G} \bm{w}\quad\text{with entries}\quad
 \bm{b}_k = \ip{\dot{x}}{\psi_k}\quad\text{and}\quad
 \bm{G}_{k,j} = \ip{\varphi_j(x)}{\psi_k},
\end{equation}
where the entries of $\bm{b}$ can be computed using integration by parts, i.e.,
\begin{equation}\label{eq:integration_by_parts}
\textstyle
\ip{\dot{x}}{\psi_k} = -\ip{x}{\dot{\psi_k}},
\end{equation}
under the assumption that $\{\psi_k\}_{k=0}^K$ have compact support in $[a,b]$.
Therefore, the weak-SINDy method avoids direct computations of $\dot{x}$ by invoking $\dot{\psi_k}$, the derivative of the test function, which is often known a priori.

\begin{remark} While not the main focus of the paper, it should be noted that inaccuracies in the numerical approximations of the inner products will perturb both $\bm{G}$ and $\bm{b}$ and affect the final weights $\bm{w}$. For a reference on peturbed least squares problems see \cite{WEI1990177}.
\end{remark}

\paragraph{Occupation kernel method}
The basic form of the operator-theoretic technique developed in \cite{rosenfeld2019occupation, SCC.Rosenfeld.Kamalapurkar.ea2019a} is similar to weak-SINDy. 
Given $x(t)$, $\forall t\in[a,b]$ and a set of test functions $\{\psi_k\}_{k=0}^K$, the occupation kernel method solves
\begin{equation}\label{eq:occ_kernel}
\textstyle
  \ip{\dot{\psi_k}(x)}{\dot{x}} = \ip{\dot{\psi_k}(x)}{\sum_{j=0}^J w_j\varphi_j(x)}\:,\quad k=0,\dots,K,
\end{equation}
where the left-hand side can be reformulated by
\begin{equation}
\textstyle
  \ip{\dot{\psi_k}(x)}{\dot{x}} = \int_a^b \dot{\psi}_k(x(t)) \dot{x} \,dt = \int^b_a \frac{d}{dt}\left(\psi_k(x(t)\right) dt  = \psi_k(x(b))-\psi_k(x(a))\:.
\end{equation}
Thus, the matrix formulation of Eq.~\eqref{eq:occ_kernel} becomes
\begin{equation}
\label{occkernel_form}
\bm{b} = \bm{G} \bm{w}\quad\text{with entries}\quad
\bm{b}_k = \psi_k(x(b))-\psi_k(x(a))\quad\text{and}\quad
\bm{G}_{k,j} =\ip{\varphi_j(x)}{\dot{\psi_k}(x)}, 
\end{equation}
which also avoid direct measurements or evaluations of $\dot{x}$.

It should be noted that in \cite{rosenfeld2019occupation} the test basis is chosen from a particular class of functions called reproducing kernels and the method is framed in the language of operators (see Section~\ref{sec:general_formulation}).

\subsection{A general class of techniques}
\label{sec:general_formulation}
In this section, we point out the structural similarities between the three techniques discussed in Section~\ref{sec:SINDy_review} by taking a Hilbert-space theory approach. Although this paper primarily focuses on the weak-SINDy method, the formulation presented in this section is expected to illuminate ways in which the primary analysis can be extended. In general, the SINDy-type techniques resolve the dynamics over a finite dimensional subspace of a Hilbert function space generated by the test functions, and system identification is achieved via enforcing a rule over this finite dimensional space. The weak-SINDy technique most obviously exemplifies this point. The original SINDy method and the occupation kernel method can be cast in this light as well. 

\begin{definition}
A (real) Hilbert function space over a domain $X$ is a Hilbert space of functions $f:X\rightarrow \mathbb{R}^n$. A reproducing kernel Hilbert space is a Hilbert function space $H$ in which the evaluation functional, $E_x(f) :=f(x)$, is continuous. By the Riesz representation theorem, there exists functions $K_x$, called reproducing kernels, such that $f(x)=\ip{f}{K_x}_H$ for all $f\in H$. 
\end{definition}
For a good reference on reproducing kernel Hilbert spaces consult \cite{paulsen2016introduction}.

Original SINDy has the same form as weak-SINDy (Eq.~\eqref{weak-SINDy_form}), under the assumptions that $x$, $\dot{x}$, and $\varphi_j(x)$ lie in the same reproducing kernel Hilbert space and that the test functions are chosen to be the reproducing kernels at the time-steps $t_k$. Thus,
\[\ip{\varphi_j(x(t))}{\psi_k}_H = \ip{\varphi\circ x }{K_{t_k}}_H = \varphi_j(x(t_k)).\] 
It should be pointed out that the test functions could have also been chosen to be Dirac delta distributions. However, this paper takes the view point that the test functions will form a finite dimensional subspace of the underlying Hilbert function space such that Hilbert space theory applies to the analysis. The theoretical sacrifice made here is the assumption that the solution and the basis functions lie in a reproducing kernel Hilbert space which equates to the assumption of some increased regularity. 

Rosenfeld et al.~also introduce theory along side their technique using the language of reproducing kernel Hilbert spaces. We include part of the theory here to further illustrate the similarities between these techniques. In \cite{rosenfeld2019occupation, SCC.Rosenfeld.Kamalapurkar.ea2019a}, solutions (considered as trajectories) are encoded as occupation kernels. 

\begin{definition}
Let $H$ be a reproducing kernel Hilbert space in which the linear functional 
\[g\mapsto \int_{a}^b g(x(t)) dt \] is bounded. By Riesz representation theorem, there exists a function $\Gamma_x\in H$ such that 
\[\int_{a}^b g(x(t)) dt = \ip{g}{\Gamma_x}.\]
We will call $\Gamma_x$ the occupation kernel for trajectory $x$. 
\end{definition}
Intuitively, we can think of occupation kernels as the output of a feature map which takes trajectories to a function in a Hilbert space. The action of the dynamics are encoded as an operator called the Liouville operator in a similar fashion to which dynamics are encoded as Koopman/composition operators in Dynamic Mode Decomposition (DMD) \cite{kutz2016dynamic, rosenfeld2022dynamic}. 
\begin{definition}
The Liouville operator on $H$ with symbol $f$ is given by 
$A_f(g)= \dot{g} \cdot f$.
\end{definition}
Under this viewpoint, $\int_a^b \dot{g}(x(t)) f(x(t)) dt = \ip{A_f(g)}{\Gamma_x}_H$. Therefore, the entries in the linear system Eq.~\eqref{occkernel_form} can be written as
\begin{equation}
\bm{b}_k = \ip{\psi_k}{A^*_f(\Gamma_x)},\quad\text{and}\quad
\bm{G}_{k,j}=\ip{\psi_k}{A_{\varphi_j}^*(\Gamma_x)}.
\end{equation}

A few points should be made at this stage. Generally, $A_f$ is a densely defined operator and the properties of the symbol influence the properties of the operator, see \cite{russo2022liouville}. Secondly, under this view point, the success of Rosenfeld et al.~method is tied to how similar the functions $A_f^*(\Gamma_x)$ and $A^*_{\sum_j w_j\varphi_j}(\Gamma_x)$ on the finite dimensional space generated by the test functions $\psi_k$.  Finally, the entries in Eq.~\eqref{weak-SINDy_form} can also be written in an operator form by invoking a composition operator $C_x(f) := f\circ x$, i.e., 
\begin{equation}
 \bm{b}_k = \ip{C_x(f)}{\psi_k},\quad\text{and}\quad
 \bm{G}_{j,k} = \ip{C_{x}(\varphi_j)}{\psi_k}.
\end{equation}

Even though this paper does not take full advantage of this operator formulation, we expect the operator theoretic viewpoints of the SINDy-type techniques to open avenues for further analysis.

\subsection{The projection property of the weak-SINDy technique}
\label{sec:projection}
In the context of system identification, the techniques presented in Section~\ref{sec:SINDy_review} are shown to be successful when the dynamics can be written as a linear combination of the choice of basis functions. However, this assumption does not hold in general when generating a surrogate model for complex systems with these techniques. In this section, we show that the weak-SINDy technique acts as a projection operator that projects the underlying dynamics to the space spanned by the chosen basis. 

As shown above, the weak-SINDy technique solves functional equations represented by a linear system Eq.~\eqref{eq:wSINDy} and then returns a linear combination of basis functions. 
Here we separate the weak-SINDy technique into three steps, referred to as the encoding, solving, and decoding steps, and formalize these steps as three mappings which we will call the \emph{encoder}, \emph{solver}, and \emph{decoder} mappings, respectively.
In the following definition, we will give precise descriptions of the space of the basis and test functions and define each mapping under the choice of basis and test functions.

\begin{definition}\label{def:projection_operators}
Let $H$ be a Hilbert function space on a set $X$ and $\{\varphi_j\}_{j=0}^J$ be a finite collection of functions in $H$. Let $\mathcal{H}$ be a Hilbert function space on a set $[a,b]$ and
$\mathcal{F} = \{\psi_k\}_{k=0}^K$ be a finite subset of basis functions in $\mathcal{H}$. Suppose that $x:[a,b]\rightarrow X$ defines a bounded composition operator $C_x:H\rightarrow \mathcal{H}$. We then define the following maps:
\begin{equation}\label{eq:operator_def}
    \begin{alignedat}{3}
&E:H \rightarrow \mathbb{R}^{K+1},  \quad 
&&E(g)= \left[ \ip{g(x)}{\psi_0}_\mathcal{H}\ldots \ip{g(x)}{\psi_K}_\mathcal{H}\right]^\top,\\
&S:\mathbb{R}^{K+1}\rightarrow \mathbb{R}^{J+1}, \quad 
&&S(\bm{b})=\argmin_{\bm{w}\in \mathbb{R}^{J+1}}\|\bm{G}\bm{w}-\bm{b}\| + \lambda\|\bm{w}\|,\\
&D:\mathbb{R}^{J+1} \rightarrow H,  \quad 
&&D(\bm{w})=\textstyle\sum_{j=0}^J w_j\varphi_j(\cdot), 
    \end{alignedat}
\end{equation}
where the entries of the matrix $\bm{G}\in \mathbb{R}^{K+1\times J+1}$ are given by $G_{k,j} = \ip{\varphi_j(x)}{\psi_k}_\mathcal{H}$ and $\lambda\geq 0$ is a regularization parameter.
We will refer to the maps $E$, $S$, and $D$ as the \emph{encoder}, \emph{solver}, and \emph{decoder} maps, respectively. 
\end{definition}
We note that the solver map $S$ is well-defined only if the problem $\min_{\bm{w}\in \mathbb{R}^{J+1}}\|\bm{G}\bm{w}-\bm{b}\| + \lambda\|\bm{w}\|$ has a unique minimizer. This requirement is usually satisfied in practice as $K$ is often chosen to be much greater than $J$. We assume that $S$ is well-defined throughout this paper.

\begin{lemma}\label{lem:solver_projection}
Let $S$ be the solver map defined in Definition~\ref{def:projection_operators}.
Let $\hat{\bm{w}} = S(\bm{b})$ for some $\bm{b}\in\mathbb{R}^{K+1}$, then $S(\bm{G}\hat{\bm{w}}) = \hat{\bm{w}} $.
\end{lemma}

\begin{proof}
By contradiction, suppose $S(\bm{G}\hat{\bm{w}}) = \bm{w}^*\neq \hat{\bm{w}}$, then, since $S$ is well-defined,
\begin{equation}
\begin{alignedat}{2}
    \|\bm{G}\bm{w}^* - \bm{G}\hat{\bm{w}}\| + \lambda\|\bm{w}^*\| 
    &< 
    \|\bm{G}\hat{\bm{w}} - \bm{G}\hat{\bm{w}}\| + \lambda\|\hat{\bm{w}}\|\\
    \Rightarrow\quad
    \|\bm{G} (\bm{w}^* - \hat{\bm{w}})\| + \lambda\|\bm{w}^*\| 
    &< 
    \lambda\|\hat{\bm{w}}\|
\end{alignedat}
\end{equation}
Thus, by the triangle inequality,
\begin{equation}
\begin{alignedat}{2}
    \|\bm{G}\bm{w}^* - {\bm{b}}\| + \lambda\|\bm{w}^*\| &\leq 
    \|\bm{G}\hat{\bm{w}} - {\bm{b}}\| + \|\bm{G} (\bm{w}^* - \hat{\bm{w}})\| + \lambda\|\bm{w}^*\| \\
    &<\|\bm{G}\hat{\bm{w}} - {\bm{b}}\| + \lambda\|\hat{\bm{w}}\|,
\end{alignedat}
\end{equation}
which contradicts to the assumption that $\hat{\bm{w}}=S(\bm{b})$.
\end{proof}
Here we do not specify the choice of the norms in the fidelity term $\|\bm{G}\bm{w}-\bm{b}\|$ and the regularization term $\lambda\|\bm{w}\|$ in the solver. In fact, the result in Lemma~\ref{lem:solver_projection} holds for arbitrary norms in $\mathbb{R}^{K+1}$ and $\mathbb{R}^{J+1}$ as shown in the above proof. 

With Lemma~\ref{lem:solver_projection}, we next show in Proposition~\ref{prop:projection} that the weak-SINDy technique is indeed a projection of the underlying dynamics to the space spanned by the basis function.
\begin{proposition}\label{prop:projection}
Let $V = \operatorname{span}\{\varphi_j\}_{j=0}^J\subset H$. Using the notation in Definition~\ref{def:projection_operators}, the mapping defined by
\[P:H\rightarrow V\subset H \quad \quad  P = D\circ S\circ E\]
is a well-defined projection map. 
\end{proposition}
\begin{proof}
We show that $P$ is a projection map by proving that $P^2 = P$. Let $F\in H$ and suppose $P(F) = (D \circ S \circ E)(F)=\sum_{j=0}^J \hat{w}_j \varphi_j(\cdot)$ for some $\hat{\bm{w}}\in\mathbb{R}^{J+1}$. From the definition of the solver operator $S$ in Eq.~\eqref{eq:operator_def}, $\hat{\bm{w}}=S(E(F))$.
Applying $E$ to $P(F)$ gives
\[\textstyle
(E\circ P)(F) = \left[\ip{\sum_{j=0}^J \hat{w}_j \varphi_j(x)}{\psi_0}_\mathcal{H}, \ldots, \ip{\sum_{j=0}^J \hat{w}_j\varphi_j(x)}{\psi_K}_\mathcal{H}\right]^\top=\bm{G}\hat{\bm{w}}.\]
It then follows from Lemma~\ref{lem:solver_projection} that
\begin{equation}
    (S\circ E \circ P)(F) = S(\bm{G}\hat{\bm{w}}) = \hat{\bm{w}}\:.
\end{equation}
Thus, we have
\begin{equation}
    P^2(F) = (D\circ S\circ E \circ P)(F) = D(\hat{\bm{w}}) = \sum_{j=0}^J \hat{w}_j \varphi_j(\cdot) = P(F)\:.
\end{equation}
\end{proof}

\subsection{weak-SINDy in the context of this paper}
\label{subsec:assumption}
The SINDy-type methods have been shown to be very efficient at system identification under the assumption that the underlying dynamics is a linear combination of the chosen finite set of basis functions. In this paper, we will abandon this assumption and rephrase the SINDy-type methods in the context of surrogate modeling. For system $\dot{x}=f(x)$, if $f$ is \emph{not} in the span of the chosen finite set of basis functions $\varphi_j$, the SINDy-type methods will return a model, which is shown in the preceding section to be a projection of the underlying dynamics in the case of weak-SINDy. In this section, we study the viability of the surrogate model and investigate the difference between the solution to the projected surrogate model and the original solution. 

For the remainder of the paper, we focus on the weak-SINDy method, which is the most amenable to Hilbert space analysis techniques. We also make several simplifying assumptions. Namely, we assume: 
\begin{enumerate}[(i)]
    \item no regularization is used in solving the least squares problem 
    Eq.~\eqref{weak-SINDy_form}, 
    \item the chosen set of basis functions $\{\varphi_j\}_{j=0}^J$ spans $\mathbb{P}_J$, where $\mathbb{P}_J$ is the space of polynomials up to degree $J$ over a domain,
    \item the test functions form an orthonormal basis,
    \item and the matrix $\bm{G}$ in Eq.~\eqref{weak-SINDy_form} is full column rank. Together with assumptions (ii) and (iii), this assumption is effectively an assumption on the trajectory $x(t)$.
\end{enumerate}
We note that, unlike \cite{messenger2021weak}, we do not assume in the analysis that the test functions have compact support in the time interval. In this case, data at the endpoints are needed for integration by parts. In addition, the choice of polynomial basis functions allows us to leverage well-known polynomial approximation properties in the analysis. Finally, as portions of this paper contain multivariable results, we define the notion of ``degree" used for multi-variable polynomials in this paper. 

\begin{definition}[Degree of a multi-variable polynomial]
Let $X\subset \mathbb{R}^N$. We say a multi-variable polynomial is of \emph{max degree} $J$ if all variables are of degree $J$ or less. We define
\begin{equation}\textstyle
    \mathbb{P}_J(X) =\left\{ p(\mathbf{x})= \sum_{\bm{j}} \alpha_{\bm{j}} \mathbf{x}^{\bm{j}} \ :\  \mathbf{x}\in X\:,\,\bm{j} = (j_1,\ldots, j_N)\:,\,\max_\ell j_\ell \leq J \right\}\:,
\end{equation}
where $J$ is the max degree. For convenience, $|\bm{j}|_\infty:=\max_\ell j_\ell$.
\end{definition}

\section{Analysis of weak-SINDy approximation to scalar ODEs}\label{sec:ODE case}
In this section, we analyze the weak-SINDy method in the case of approximating a scalar ODE
\begin{equation}\label{eq:ode}
\dot{x}(t) = f(x(t)),\quad x(a) = \eta\:,\quad \text{for } t\in[a,b]\:,
\end{equation}
where $f:X\to \mathbb{R}$ is the dynamics with $X$ the range of $x$, which for this paper we assume is bounded.
We start from the following proposition, which bounds the difference between the approximate and original solutions by the error in the approximate system. 
\begin{proposition}\label{prop:ODE_sol_bound}
Suppose $x:[a,b]\rightarrow X\subseteq\mathbb{R}$ on some compact region $[a,b]$ satisfies the ODE Eq.~\eqref{eq:ode}. Let $\hat{x}:[a,b]\rightarrow X$ satisfy
\begin{equation}\label{eq:approx_ODE}
    \dot{\hat{x}}(t) = p(\hat{x}(t)),\quad \quad \hat{x}(a) = \eta
\end{equation}
for some $p:X\to\mathbb{R}$. Then,
$\|x - \hat{x}\|_{L^\infty([a,b])} 
\leq (b-a)^{1/2}\|f\circ x -p\circ \hat{x}\|_{L^2([a,b])}$.
\end{proposition}

 \begin{proof}
 For each $t\in[a,b]$, we have
  \begin{equation}
  \begin{alignedat}{2}
  |x(t) - \hat{x}(t)| 
  &=\left|\int_{a}^t f(x(s)) - p(\hat{x}(s)) \, ds\right |\\
  &\leq \int_{a}^b |f(x(s))-p(\hat{x}(s))|ds 
  \leq (b-a)^{1/2}  \|f\circ x -p\circ \hat{x}\|_{L^2([a,b])}\:
  \end{alignedat}
  \end{equation}
  by H\"{o}lder's inequality. This concludes the proof since the bound holds for all $t\in[a,b]$. 
\end{proof}
Motivated by Proposition~\ref{prop:ODE_sol_bound}, the goal of this section is to estimate
\begin{equation}\label{eq:thingtoest}
    \|f\circ x -p\circ \hat{x}\|_{L^2([a,b])},
\end{equation} 
where $p$ is the polynomial projection of the dynamics $f$ given by the weak-SINDy method.
 
Applying the triangle inequality to Eq.~\eqref{eq:thingtoest} gives
\begin{equation}
  \|f \circ x - p \circ \hat{x}\|_{L^2([a,b])}\leq \|f\circ x - p\circ x\|_{L^2([a,b])} + \|p\circ x - p \circ \hat{x}\|_{L^2([a,b])}.  
\end{equation}
We then provide an estimate on the first term on the right-hand side in Section~\ref{sec:scalar_ODE_analysis} and take the second term into account in Section~\ref{sec:Lipschitz argument}.

Boundedness of composition operators will play an important role in this paper in establishing several inequalities using the standard operator bound.
\begin{equation}\label{opnorm}
\|A x\|_\mathcal{H} \leq \|A\|_{\op}\|x\|_H\, 
\,\, \text{where} \,\,
\|A\|_{\op} = \sup\{ \|Ax\|_\mathcal{H}\colon \|x\|_H \leq 1\}
\,\,\text{for}\,\, A:H\rightarrow \mathcal{H}.
\end{equation}
Additionally, boundedness of a linear operator is equivalent to continuity (in the norm topology). 

A well-known theorem in \cite{dunford1988linear} proves boundedness of $L^2$ composition operators under reasonable assumptions in the setting that the flowmap $\phi$ is a self-mapping of a set $S$ into itself. 
In the following Proposition~\ref{prop:compmapboundcondition}, we give a slightly generalized version of this well-known theorem that is not restricted to self-mappings.
\begin{remark}
Here Proposition~\ref{prop:compmapboundcondition} is stated in the multi-dimensional setting. While the analysis in Section~\ref{sec:ODE case} focuses only on the one-dimensional case, the multi-dimensional statement is applicable in Section~\ref{sec:system of odes}.
\end{remark}
\begin{notation}
Let $\mu$ be a measure on $\mathbb{R}^n$  with $\sigma$-algebra $\Sigma$ and $A$ a finite measurable subset of $\mathbb{R}^n$.  Let $\Sigma_A$ and $\mu_A$ denote the restriction of the $\Sigma$ and $\mu$ to subsets of $A$.
\end{notation}

\begin{proposition}\label{prop:compmapboundcondition}
Let $\phi:B\rightarrow A:=\phi(B)$ where $A$ and $B$ are finite measurable subsets of $\mathbb{R}^n$. The composition operator $C_\phi:L^p(A, \Sigma_A, \mu_A)\rightarrow L^p(B, \Sigma_B, \mu_B)$
given by 
\[L^p(A, \Sigma_A, \mu_A)\ni f\mapsto f\circ\phi \in L^p(B, \Sigma_B, \mu_B) \]
is well defined and bounded if $\mu_B(\phi^{-1}(E))=0$ whenever $\mu_A(E)=0$ and 
\[\sup_{E\in \Sigma_A} \frac{\mu_B(\phi^{-1}(E))}{\mu_A(E)}=M<\infty.\]
Here, $\|C_\phi\| = M^{\frac{1}{p}}$.
\end{proposition}

\begin{proof}
Note, if $f\circ \phi(s)\neq g\circ \phi(s)$ then $f(r) \neq g(r)$ for $r=\phi(s)$. Let \[S= \{s\in B \mid f\circ \phi(s)\neq g\circ \phi(s)\}\]
and
\[R = \phi(S) \subset A.\]
Note, if $f=g$ almost everywhere in $\mu_A$, then the condition $\mu_B(\phi^{-1}(E))=0$ whenever $\mu_A(E) = 0$ gives us that 
\[0\leq \mu_B(S)\leq \mu_B(\phi^{-1}(R))= 0 \]
since $\mu_A(R)=0$. Therefore the operator is well-defined. For simple functions $f= \sum_{i}\alpha_i \chi_{E_i}$, the inequality
\[\|C_\phi f\|^p =\sum_{i}|\alpha_i|^p\mu_B(\phi^{-1}(E_i))\leq \sum_{i}|\alpha_i|^pM\mu_A(E_i)=M\|f\|^p \]
 shows that $\|C_\phi\|$ is bounded by $M^{\frac{1}{p}}$, where the first equality follows from $ \|C_\phi \chi_E\|^p = \mu_B(\phi^{-1}(E))$. With the boundedness of $C_\phi$,
\[\mu_B(\phi^{-1}(E)) = \|C_\phi \chi_E\|^p\leq \|C_\phi\|^p\|\chi_E\|^p=\|C_\phi\|^p\mu_A(E)\]
shows that $\|C_\phi\|\geq M^{\frac{1}{p}}$. Hence, $\|C_\phi\| = M^{\frac{1}{p}}$.

\end{proof}

\subsection{Convergence of weak-SINDy approximation} 
\label{sec:scalar_ODE_analysis}
In this section, we provide estimates of the approximation error $\|f\circ x-p\circ x\|_{L^2([a,b])}$ as the degrees of projection basis functions and test functions increase. Recall that the projection basis functions are denoted as $\varphi_j:X\to \mathbb{R}$, $j=0,\dots,J$ and the test functions are denoted as $\psi_k:[a,b]\to\mathbb{R}$, $k=0,\dots,K$. 
To simplify the presentation, here we focus on the case that both $\{\varphi_j\}$ and $\{\psi_k\}$ are chosen to be polynomials. The analysis extends to other choices of test functions, e.g., Fourier basis, as long as similar convergence properties are available.

The polynomial space spanned by the basis and test functions are denoted as $\mathbb{P}_J(X)$ and $\mathbb{P}_K([a,b])$, respectively.
We define the operator that projects functions in $\mathcal{H}$ onto $\mathbb{P}_K([a,b])$ in the $L^2$ sense as $\projK$. 
It is then straightforward to verify that the weak-SINDy approximation obtained by solving the least squares problem 
$\min_{\bm{w}\in \mathbb{R}^{J+1}}\|\bm{G}\bm{w}-\bm{b}\|_2$
with $\bm{G}$ and $\bm{b}$ defined in Eq.~\eqref{weak-SINDy_form} is equivalent to the minimizing polynomial
\begin{equation}\label{eq:p_JK_def}
    p_{J,K}^* := \argmin_{p\in\mathbb{P}_J(X)}\|\projK (f\circ x - p \circ x) \|_{L^2([a,b])}\:, 
\end{equation}
which is used throughout this section. Here the uniqueness of $p_{J,K}^*$ is guaranteed by the full column rank assumption of $\bm{G}$ stated in Section~\ref{subsec:assumption}, which is assumed throughout the paper along with other assumptions given therein.

To start the analysis, we first introduce the following notation.

\begin{notation}
Let $H^m(\Omega)$ denote the Hilbert Sobolev space of $L^2(\Omega)$ functions with $L^2$ distributional derivatives up to and including order $m$. Let $C^s(\Omega)$ be the space of $s$ times continuously differentiable functions on a set $\Omega$.
\end{notation}

The following two estimates in Lemmas~\ref{lem:polyapproxindeg} and \ref{lem:SobolevPolyest} for polynomial approximation errors are essential tools in the analysis in this section.
In these two lemmas, the operator $\projJ$ is defined as the $L_2$ projection operator from $H^m(\Omega)$ onto $\mathbb{P}_J(\Omega)$.
\begin{remark} In this section, we only need the single variable ($n=1$) version of Lemmas~\ref{lem:polyapproxindeg} and \ref{lem:SobolevPolyest}. However, we state these two lemmas in the general multi-variable form since they are needed in the analysis of ODE system approximations considered in Section~\ref{sec:system of odes} (specifically for Theorem~\ref{multivar convergence theorem}).
In these lemmas, $J$ refers to the max degree when multi-variable polynomials are considered.
\end{remark}

\begin{lemma}\label{lem:polyapproxindeg}
Let $\Omega$ a bounded cube in $\mathbb{R}^{n}$. Suppose $u\in H^m(\Omega)$, for any real $m\geq 0$ then 
\begin{equation}
    \|u-\projJ u\|_{L^2(\Omega)}
    \leq C\, J^{-m} \, \|u\|_{H^m(\Omega)},
\end{equation}
where the constant $C$ is independent of $J$ and $u$. 
\end{lemma}
\begin{proof}
See \cite[Theorem~2.3]{canuto1982approximation}.
\end{proof}

\begin{lemma}\label{lem:SobolevPolyest}
Let $\Omega$ be a bounded cube of $\mathbb{R}^{n}$, Suppose $u\in H^m(\Omega)$, then for any real $0\leq s\leq m$ we have 
\[\|u-\projJ u\|_{H^s(\Omega)} \leq C J^{e(s,m)}\|u\|_{H^{m}(\Omega)}
\quad\text{with}\quad
e(s,m) =\left\{ \begin{array}{lc} 2s - m - \frac{1}{2} & s\geq 1\\
\frac{3s}{2} - m & 0\leq s \leq 1\end{array}\right.,\]
where the constant $C$ is independent of $J$ and $u$.
\end{lemma}
\begin{proof}
See \cite[Theorem~2.4]{canuto1982approximation}. 
\end{proof}

We are now ready to estimate the approximation error $\|f\circ x - p_{J,K}^*\circ x\|_{L^2([a,b])}$. To start, we apply the triangle inequality
\begin{align}
\|f\circ x - p_{J,K}^*\circ x\|_{L^2([a,b])}\leq 
&\|f\circ x - \projK(f\circ x)\|_{L^2([a,b])}
+ \|\projK(f\circ x-p^*_{J,K}\circ x)\|_{L^2([a,b])}\nonumber\\
\label{eq:triangle_ineq}
&+\|p^*_{J,K}\circ x - \projK(p^*_{J,K}\circ x)\|_{L^2([a,b])},
\end{align}
and provide upper bounds for each of the three terms on the right-hand side independently. The upper bound for the first term follows from a direct application of Lemma~\ref{lem:polyapproxindeg}, which requires the regularity information of $f\circ x$. Here, we derive the regularity of $f\circ x$ from regularity assumptions on the solution $x$ and the composition operator $C_x$.

\begin{remark}\label{fixed_reg}
Let $x:[a,b]\rightarrow X$ be a function in $C^s([a,b])$. Since all polynomials $p$ are $C^\infty$, it follows that $p\circ x\in C^s([a,b])$ (and thus in $H^s([a,b])$).
Further, suppose $f\in H^m(X)$ and $C_x:H^s(X)\rightarrow H^s([a,b])$ is a bounded operator for some $s\in (0,m]$, then the boundedness of $C_x$ implies that $f\circ x\in H^s([a,b])$.  
\end{remark}
\begin{corollary}\label{cor:subspacenormest}
Suppose $f\in H^m(X)$ and $x\in C^s([a,b])$ for some $s\leq m$, then 
\begin{equation}\label{eq:subspacenormest}
\|f\circ x -\projK(f\circ x)\|_{L^2([a,b])}\leq C K^{-s}\|(f\circ x)\|_{H^s([a,b])}\:.
\end{equation}
provided $C_x:H^s(X)\rightarrow H^s([a,b])$ is bounded. 
\end{corollary}
Next, we give an estimate of the second term in Eq.~\eqref{eq:triangle_ineq} in the following proposition.
\begin{proposition}\label{prop:minproblem} 
Let $f\in H^m(X)$  for some $m>0$ and $x:[a,b]\to X$, then
\begin{equation}
    \|\projK(f\circ x-p_{J,K}^*\circ x )\|_{L^2([a,b])}\leq 
    C \, \|C_x\|_{\op}\, J^{-m} \,  \|f\|_{H^m(X)}\:,
\end{equation}
provided that the composition operator $C_x$ is bounded in the $L^2$ sense.
\end{proposition}

\begin{proof}
    By definition of $p_{J,K}^*$ in Eq.~\eqref{eq:p_JK_def}, we have that
    \begin{equation}\label{eq:prop_5_1}
        \|\projK (f\circ x - p_{J,K}^* \circ x) \|_{L^2([a,b])}\leq
        \|\projK (f\circ x - p \circ x) \|_{L^2([a,b])},\quad 
        \forall p\in \mathbb{P}_J(X)\:.
    \end{equation}
    Since $\projK$ is the minimum $L^2$ projection from $\mathcal{H}$ to $\mathbb{P}_K([a,b])$, we have, $\forall p\in \mathbb{P}_J(X)$,
    \begin{equation}\label{eq:prop_5_2}
        \|\projK(f\circ x-p\circ x )\|_{L^2([a,b])} \leq\|f\circ x-p\circ x \|_{L^2([a,b])}\leq  \|C_x\|_{\op}\|f-p\|_{L^2(X)},
    \end{equation}
    where $\|C_x\|_{\op}$ denotes the operator norm of $C_x$ (see Eq.~\eqref{opnorm}).
    Now, define $p_J:=\projJ f \in \mathbb{P}_J(X)$, it then follows from Eq.~\eqref{eq:prop_5_2} and Lemma~\ref{lem:polyapproxindeg} that
    \begin{equation}\label{eq:prop_5_3}
        \|\projK(f\circ x-p_J\circ x )\|_{L^2([a,b])} \leq  C\, \|C_x\|_{\op} \, J^{-m}\, \|f\|_{H^m(X)}\:.
    \end{equation}
    Plugging Eq.~\eqref{eq:prop_5_3} into Eq.~\eqref{eq:prop_5_1} leads to the desired bound.
\end{proof}
 
Several of the results to follow require $p\circ x$ to be of a fixed regularity for all polynomials $p$. Based on Remark~\ref{fixed_reg}, $p \circ x$ is of the same regularity for all polynomials $p$.
With Remark~\ref{fixed_reg}, we apply Lemma~\ref{lem:polyapproxindeg} and provide an upper bound for the third term in Eq.~\eqref{eq:triangle_ineq}, $\|p_{J,K}^*\circ x -\projK(p_{J,K}^*\circ x)\|_{L^2([a,b])}$, in the following corollary.

\begin{corollary}\label{cor:polyprojest}
Suppose $x\in C^s([a,b])$, then 
\begin{equation}\label{eq:polyprojest}
    \|p_{J,K}^*\circ x -\projK(p_{J,K}^*\circ x)\|_{L^2([a,b])}
    \leq C K^{-s}\|p_{J,K}^*\circ x\|_{H^s([a,b])},\quad
    J=1,2,\dots\:.
\end{equation}
\end{corollary}

We note that since the constant $C$ in Eq.~\eqref{eq:polyprojest} comes from Lemma~\ref{lem:polyapproxindeg}, $C$ is independent of $K$ and $p_{J,K}^*$. Therefore, we next show that $\|p_{J,K}^*\circ x\|_{H^s([a,b])}$ is bounded from above by a constant independent of $J$ and $K$. To bound  $\|p_{J,K}^*\circ x\|_{H^s([a,b])}$, we will make use of two intermediate polynomials,
\begin{equation}
  p_J := \argmin_{p\in \mathbb{P}_J(X)}\|f-p\|_{L^2(X)} 
  \quad\text{and}\quad
  q_J :=\argmin_{p\in \mathbb{P}_J(X)} \|f\circ x - p\circ x \|_{L^2([a,b])},
\end{equation} 
where $p_J$ is the best $L^2$ polynomial approximation on $X$ and $q_J$ is the best polynomial approximation under composition with $x$.
Proposition \ref{prop:pjk_converges_to_qj} states that, $p^*_{J,K}$ converges to $q_J$ as $K$ increases, and the proof is given in Appendix \ref{sec:appendix}. We will also show in Lemma \ref{lem:q_JfollowsP_J} that $q_J\circ x$ converges to $p_J\circ x$ as $J$ increases. Since $p_J$ converges to $f$ as $J$ increases by Lemma \ref{lem:polyapproxindeg}, we will use these facts to bound $\|p_{J,K}^*\circ x\|_{H^s([a,b])}$ in Proposition \ref{prop:R3_bound}. 
\begin{proposition}\label{prop:pjk_converges_to_qj}
Under the assumption that the polynomial $q_J$ is uniquely defined, $p_{J,K}^*\in H^s(X)$ converges to $q_J\in H^s(X)$ in Sobolev norm as $K$ increases for any $s\geq 0$ and bounded domain $X$. 
\end{proposition}
\begin{proof}
This is a restatement of Proposition \ref{prop:coefficientconvergence}, the proof of which is in Appendix \ref{sec:appendix}.
\end{proof}
By expressing in terms of the test function basis $\{\psi_k\}$, we note that $q_J$ solves the minimization problem:
\begin{equation}\label{eq:qj_min}
\min_{p\in \mathbb{P}_J} \sum_{k=0}^\infty |\langle f\circ x - p \circ x,\psi_k\rangle|^2
\end{equation}
and $p^*_{J,K}$ solves a truncated version:
\begin{equation}\label{eq:pjk_min}
\min_{p\in \mathbb{P}_J} \sum_{k=0}^K |\langle f\circ x - p \circ x,\psi_k\rangle|^2
\end{equation}
It is clear that the minimal value of Eq.~\eqref{eq:pjk_min} converges to the one of Eq.~\eqref{eq:qj_min} as $K\to\infty$. However, it is necessary to show that the minimizers converge as well. The convergence proof of the minimizers, i.e., the polynomials, is given in Appendix~\ref{sec:appendix}, which makes use of strong convexity arguments of the objective functions in Eqs.~\eqref{eq:qj_min} and \eqref{eq:pjk_min}.

Next, to show that $q_J\circ x$ converges to $p_J\circ x$, we need a generalized reverse Sobolev inequality as stated in the following Proposition~\ref{prop:generalized_reverse_sobolev}, which is a single variable simplification of Proposition~\ref{prop:mv-generalized_reverse_sobolev} stated in Appendix~\ref{appendix:recurrence_relations} as  and proved therein. 

\begin{proposition}[Generalized reverse Sobolev inequality]
\label{prop:generalized_reverse_sobolev}
Suppose that $x\in C^s([a,b])$, $u$ is a $J$-th degree polynomial, then there exists a constant $C$ such that
\begin{equation}
\|u\circ {x}\|_{H^{s}([a,b])}\leq {C} J^{2s} \|u\circ {x}\|_{L^2([a,b])}.
\end{equation}
\end{proposition}
\begin{proof}
    See the proof of Proposition~\ref{prop:mv-generalized_reverse_sobolev} in Appendix~\ref{appendix:recurrence_relations}.
\end{proof}

\begin{lemma}
\label{lem:q_JfollowsP_J} Define $r_J = q_J-p_J\in \mathbb{P}_J(X)$. Let $f\in H^m(X)$, $x\in C^s([a,b])$, and $s$ be such that $s<\frac{m}{2}$.
Suppose that the composition operator $C_x$ is bounded in the $L^2$ sense, then there exists a constant $C$ such that 
\begin{equation}
    \|r_J\circ x\|_{H^s([a,b])}\leq C \|C_x\|_{\textup{op}}  J^{2s-m} \|f\|_{H^m(X)}\:.
\end{equation}
\end{lemma}
\begin{proof}From the definition of $r_J$ and the triangle inequality, we have
\begin{equation}
    \begin{alignedat}{2}
    \|r_J\circ x\|_{L^2([a,b])} &\leq \|f\circ x - q_J \circ x \|_{L^2([a,b])} + \|f\circ x - p_J \circ x \|_{L^2([a,b])}\\
    &\leq \|f\circ x - p_J \circ x \|_{L^2([a,b])} + \|f\circ x - p_J \circ x \|_{L^2([a,b])}\\
    & \leq 2\|C_x\|_{\text{op}}\|p_J - f\|_{L^2(X)}
    \leq 2C\|C_x\|_{\text{op}}J^{-m}\|f\|_{H^m(X)},
    \end{alignedat}
\end{equation}
where the definition of $p_J$, the boundedness of $C_x$, and Lemma~\ref{lem:polyapproxindeg} are used. The result then follows from Proposition~\ref{prop:generalized_reverse_sobolev}. 
\end{proof}

\begin{proposition}
\label{prop:q_jcircx-bound}
Let $f\in H^m(X)$, $x\in  C^s([a,b])$, and $s$ be such that $s<\frac{m}{2}$. Suppose that the composition operator $C_x$ is bounded in both the $L^2$ and $H^s$ sense, then we have 
\begin{equation}
   \|q_J\circ x\|_{H^s([a,b])}\leq {C}J^{2s-m}\|f\|_{H^m(X)} + {C}J^{e(s,m)}\|f\|_{H^m(X)} + \|f\circ x\|_{H^s([a,b])}.
\end{equation}
\end{proposition}

\begin{proof}
By the triangle inequality,
\begin{align}
 \|q_J\circ x\|_{H^s([a,b])}&\leq  \|q_J \circ x- p_J\circ x\|_{H^s([a,b])} + \|p_J\circ x - f\circ x \|_{H^s([a,b])} + \|f\circ x\|_{H^s([a,b])}\nonumber\\
&\leq \|r_J\circ x\|_{H^s([a,b])} +  \|C_x\|_{\text{op}}\|p_J - f \|_{H^s(X)} + \|f\circ x\|_{H^s([a,b])}.
 \end{align}
 We apply Lemma \ref{lem:q_JfollowsP_J} to the first term, and the second term is bounded by Lemma \ref{lem:SobolevPolyest}.
\end{proof}

\begin{proposition}\label{prop:R3_bound}
Let $f\in H^m(X)$, $x\in C^s([a,b])$, and $s$ be such that $s<\frac{m}{2}$. Suppose that the composition operator $C_x$ is bounded in the $L^2$ and $H^s$ sense. Then, for some $\tilde{C} >0$, 
\begin{equation}
    \|p^*_{J,K}\circ x\|_{H^s([a,b])}\leq \|f\circ x\|_{H^s([a,b])}+\tilde{C}\:,
\end{equation}
and thus
\begin{equation}
\|p_{J,K}^*\circ x -\projK(p_{J,K}^*\circ x)\|_{L^2([a,b])}\leq C K^{-s}(\|f\circ x\|_{H^s([a,b])}+\tilde{C})\:.
\end{equation}
\end{proposition}

\begin{proof}

By the triangle inequality,
\begin{equation}
\label{eq:newthingtobound}
\begin{alignedat}{2}
    \|p^*_{J,K}\circ x\|_{H^s([a,b])}
    &\leq 
    \|(q_J \circ x - p^*_{J,K}\circ x)\|_{H^s([a,b])}
     + \|q_J \circ x\|_{H^s([a,b])}
\end{alignedat}
\end{equation}
where $\|(q_J \circ x - p^*_{J,K}\circ x)\|_{H^s([a,b])}\rightarrow 0$ as $K\rightarrow \infty$ by Proposition \ref{prop:pjk_converges_to_qj} and continuity of the composition operator. We bound $\|q_J \circ x\|_{H^s([a,b])}$ by applying Proposition \ref{prop:q_jcircx-bound}.

Since $s<\frac{m}{2}$, $J^{2s-m}$ and $J^{e(s,m)}$ decrease as $J$ increases, hence there exists some constant $\tilde{C}>0$, independent of $J$ and $K$, such that $\|p^*_{J,K}\circ x\|_{H^s([a,b])} \leq \|f\circ x\|_{H^s([a,b])}+\tilde{C}$, proving the first statement. Plugging this into Eq.~\eqref{eq:polyprojest} in Corollary~\ref{cor:polyprojest} then leads to the second statement.
\end{proof}

With the three terms on the right-hand side of Eq.~\eqref{eq:triangle_ineq} bounded, the following theorem is then a direct consequence of Corollary~\ref{cor:subspacenormest}, Proposition~\ref{prop:minproblem}, and Proposition~\ref{prop:R3_bound}.

\begin{theorem}\label{thm:convergence theorem}
Let $f\in H^m(X)$, $x\in C^s([a,b])$, and $s$ be such that $s<\frac{m}{2}$. Suppose that the composition operator $C_x$ is bounded in both the $L^2$ and $H^s$ sense.
The approximation error 
\begin{equation}
\begin{alignedat}{2}
    \|f\circ x - p_{J,K}^*\circ x\|_{L^2([a,b])} \leq&  
    C K^{-{s}}\|f\circ x\|_{ H^s([a,b])}
    + C \, \|C_x\|_{\op}\, J^{-m} \,  \|f\|_{H^m(X)}\\
    &+ C K^{-s}(\|f\circ x\|_{H^s([a,b])}+\tilde{C})\:.
\end{alignedat}
\end{equation}
\end{theorem}

\subsection{Error estimates on the approximate solution from weak-SINDy}
\label{sec:Lipschitz argument}
With the convergence of the consistency error shown in Theorem~\ref{thm:convergence theorem},
we show in the following proposition that the difference between the original and approximate solutions is governed by the consistency error in a sufficiently small time interval.

\begin{proposition}\label{prop:Lipschitz} Consider the ODE given in Eq.~\eqref{eq:ode} on time interval $[a,b]$.
Suppose that the polynomial $p=p_{J,K}^*$ defined in Eq.~\eqref{eq:p_JK_def} satisfies
\begin{equation}\label{eq:small_consistency_error}
    \|f\circ x - p\circ x\|_{L^2([a,b])}\leq \varepsilon,
\end{equation}
for some $\varepsilon>0$. Let $L_{[a,\beta]}$ denote the Lipschitz constant of $p$ on some time interval $[a,\beta]\subseteq[a,b]$ and $\hat{x}$ be the approximate solution in Eq.~\eqref{eq:approx_ODE}. Suppose that $[a,\beta]$ is sufficiently small such that $(\beta-a)L_{[a,\beta]}<1$, then 
\begin{equation}
\|x-\hat{x}\|_{L^\infty ([a, \beta])}\leq 
\frac{(\beta-a)^{1/2}}{1-(\beta-a) L_{[a,\beta]}}\,\varepsilon\:.
\end{equation}

\end{proposition}

\begin{proof}
By the definition of Lipschitz constant $L_{[a,\beta]}$, we have
\begin{equation}\label{eq:Lipschitz_bound}
\|p \circ x - p \circ \hat{x} \|^2_{L^2([a,\beta])}
\leq \int^\beta_a L_{[a,\beta]}^2 |x(s)-\hat{x}(s)|^2 ds
\leq L_{[a,\beta]}^2\|x-\hat{x}\|^2_{L^\infty([a,\beta])}(\beta-a).
\end{equation}

The claim is then proved via a similar analysis as in Proposition~\ref{prop:ODE_sol_bound}, i.e. by H\"{o}lder's inequality,
\begin{align}
\|x-\hat{x}\|_{L^\infty ([a, \beta])}
&\leq (\beta-a)^{1/2}\|f\circ x -p \circ \hat{x}\|_{L^2([a,\beta])}\nonumber \\
&\leq (\beta-a)^{1/2}\left(\|f\circ x - p \circ x \|_{L^2([a,\beta])} 
+ \|p \circ x - p \circ \hat{x} \|_{L^2([a,\beta])}\right) \label{eq:alphabeta} \\
&\leq (\beta-a)^{1/2}\, \varepsilon + (\beta-a) L_{[a,\beta]} \|x-\hat{x}\|_{L^\infty([a,\beta])}\:, \nonumber 
\end{align}
where the last inequality follows from Eq.~\eqref{eq:small_consistency_error}, $[a,\beta]\subseteq[a,b]$, and Eq.~\eqref{eq:Lipschitz_bound}.
\end{proof}

\section{weak-SINDy approximation to ODE systems and an application to POD discretization}
\label{sec:system of odes}
In this section, we extend the convergence theorems in Section~\ref{sec:scalar_ODE_analysis} for scalar ODEs to ODE systems. Our results readily generalize to the multi-variable setting as many of the proof techniques rely on Hilbert space methods (see Remark \ref{remark:single-multi-remark}). As an application, we show how weak-SINDy techniques can be applied to POD system inference as considered in \cite{qian2022reduced}. 

\subsection{Analysis of weak-SINDy approximation to ODE systems}
\label{sec:ode_system_analysis}

In this section, we will extend the error analysis in Section~\ref{sec:ODE case} to the case of systems of ODEs. 
Here $i$ is used as a component-wise index, e.g., $x_i$ denotes the $i$-th component of $\mathbf{x}\in\mathbb{R}^N$. Consider the ODE system
\begin{equation}\label{system}
\dot{\mathbf{x}}(t)= f(\mathbf{x}(t)),\qquad \mathbf{x}(a) = \boldsymbol{\eta}\in \mathbb{R}^N,\qquad \text{for } t\in[a,b]\:, 
\end{equation}
where $\mathbf{x}(t):=[x_1(t), \dots, x_N(t)]^T$ with $x_i\colon [a,b]\to X_i\subseteq\mathbb{R}$, $i=1,\dots,N$. The $i$-th component of $f\colon X^N \to \mathbb{R}^N$ is denoted as $f_i\colon X^N \to \mathbb{R}$, $i=1,\dots,N$, with $X^N:=\bigotimes_{i=1}^N X_i$ denoting the Cartesian product of $X_i$.
While the test functions $\psi_k\colon[a,b]\to\mathbb{R}$ remain identical to the one considered in Section~\ref{sec:ODE case}, here we consider the set of basis functions to be $\{\varphi_{\mathbf{j}}\}_{\mathbf{j}=0}^J$ that spans $\mathbb{P}_J(X^N)$, the space of polynomials up to degree $J$ on $X^N$, where $\mathbf{j}\in\mathbb{N}^N$ is a multi-index. Here $J$ refers to the max degree of a polynomial.

\begin{remark}
\label{remark:single-multi-remark}
In this section, the domain $X^N$ defined above is no longer the range of the trajectory as in Section~\ref{sec:ODE case}. 
This choice allows for definitions of the function spaces and the relevant polynomials $p_J, q_J$ and $p^*_{J,K}$ on $X^N$, which are used in the following analysis.
\end{remark}

The weak-SINDy method approximates the ODE system in a component-wise manner. Specifically, for each component $f_i$ of the vector-valued function $f$, the weak-SINDy approximation is given by
\begin{equation}\label{eq:p_iJK_def}
    p_{i,J,K}^*:= 
    \argmin_{p\in \mathbb{P}_J(X^N)}
    \|\projK(f_i\circ \mathbf{x} - p\circ \mathbf{x})\|_{L^2([a,b])}\:.
\end{equation}
As in the scalar ODE case, the minimization problem in Eq.~\eqref{eq:p_iJK_def} can be written in the least-squares form, i.e.,
\begin{equation}\label{eq:ODE_system_LS}
\textstyle
    p_{i,J,K}^* = \sum_{|\mathbf{j}|_\infty\leq J}\bm{w}_{\mathbf{j}}^{(i)} \cdot \varphi_{\mathbf{j}}\:,\quad\text{with}\quad 
    \bm{w}^{(i)} := \argmin_{\bm{w}}\|\bm{G}\bm{w} - \bm{b}^{(i)}\|^2_2\:,
\end{equation}
where
$
    \bm{G}_{k,\,\mathbf{j}} := \ip{\varphi_{\mathbf{j}}(\mathbf{x}(\cdot))}{\psi_k}
$ 
and 
$
    \bm{b}^{(i)}_k := \ip{\dot{x}_i}{\psi_k} = - \ip{x_i}{\dot{\psi}_k}.
$

To this end, we bound the difference between the approximate and original solutions following the same approach as in Proposition~\ref{prop:ODE_sol_bound}.
Here, a few lemmas that were stated/proved in the multi-variable setting in Section~\ref{sec:ODE case} can be directly applied.

\begin{proposition}
 Suppose that the system \eqref{system} has a solution $\mathbf{x}:[a,b]\rightarrow X^N \subseteq \mathbb{R}^N$. 
 Let $\hat{\mathbf{x}}:[a,b]\rightarrow X^N$ be a solution to the ODE system with dynamics $p\colon X^N\to\mathbb{R}^N$, i.e.,
 \begin{equation}\label{eq:approx_ODE_sys}
     \dot{\hat{\mathbf{x}}}(t) = p(\hat{\mathbf{x}}(t)),\quad \quad \hat{\mathbf{x}}(a) = \boldsymbol{\eta}\:.
 \end{equation} 
     Then, 
$         \|x_i - \hat{x}_i\|_{L^\infty([a,b])} \leq 
         (b-a)^{1/2}\|f_i\circ\hat{\mathbf{x}} - p_i\circ\mathbf{x}\|_{L^2([a,b])}
         $
     for $i = 1,\ldots, N$.

 \end{proposition}

Using the same argument as in Section~\ref{sec:scalar_ODE_analysis} leads to an estimate on the component-wise approximation error $\|f_i\circ \mathbf{x} - p_{i,J,K}^*\circ \mathbf{x}\|_{L^2([a,b])}$ as given in Theorem~\ref{multivar convergence theorem}. The generalization from single-variable to multi-variable functions is rather straightforward, given that the main tools for polynomial approximation error estimations, Lemmas~\ref{lem:polyapproxindeg} and \ref{lem:SobolevPolyest}, are already stated in the multi-variable form and the generalized reverse Sobolev inequality in the multi-variable setting is readily stated in Proposition~\ref{prop:mv-generalized_reverse_sobolev} and proved in Appendix~\ref{appendix:recurrence_relations}.

\begin{theorem}\label{multivar convergence theorem}
Suppose $f_i\in H^m(X^N)$, $\mathbf{x}\in C^s([a,b])$ and $s$ be such that $s<\frac{m}{2}$. Further, suppose that the composition operator $C_x$ is a bounded operator in both the $L^2$ and $H^s$ sense, then the approximation error \begin{equation}
\begin{alignedat}{2}
    \|f_i\circ \mathbf{x} - p_{i,J,K}^*\circ \mathbf{x}\|_{L^2([a,b])} &\leq  
    C K^{-{s}}\|f_i\circ \mathbf{x}\|_{{H^s([a,b])}}
    + C \, \|C_x\|_{\op}\, J^{-m} \,  \|f_i\|_{H^m(X^N)}\\
    &+ C K^{-s}(\|f_i\circ \mathbf{x}\|_{H^s([a,b])}+\tilde{C})\:.
\end{alignedat}
\end{equation}
\end{theorem}

With this bound on the component-wise approximation error, we proceed with the following proposition to bound the difference between the original and approximate solutions following the same strategy as in Proposition~\ref{prop:Lipschitz}.

\begin{proposition} 
Consider ODE system \eqref{system} on time interval $[a,b]$.
Suppose that the polynomial $p_i=p_{i,J,K}^*$ defined in Eq.~\eqref{eq:p_iJK_def} satisfies
\begin{equation}
    \|f_i\circ \mathbf{x} - p_i\circ \mathbf{x}\|_{L^2([a,b])}\leq \varepsilon,
\end{equation}
for some $\varepsilon>0$.
Let $L_{[a,\beta]}$ denote the Lipschitz constant of $p_i$ on some time interval $[a,\beta]\subseteq[a,b]$ and $\hat{\mathbf{x}}$ be the approximate solution in Eq.~\eqref{eq:approx_ODE_sys}. Suppose that $[a,\beta]$ is sufficiently small such that $(\beta-a)L_{[a,\beta]}<1$, then 
\begin{equation}
\|x_i-\hat{x}_i\|_{L^\infty ([a, \beta])}\leq 
\frac{(\beta-a)^{1/2}}{1-(\beta-a) L_{[a,\beta]}}\,\varepsilon\:.
\end{equation}

\end{proposition}

\subsection{POD Background} \label{subsec:POD_discretization}
 Suppose that we have the following PDE
\[\partial_t u = F(u(x,t))\] for some function $F$ with a unique solution $u\colon\Omega\times [0,T]\to\mathbb{R}$. Discretization of the domain of $u$ offers a way of approximating a PDE as an ODE system. However, depending on the necessary amount of points in the discretization these systems can be very high dimensional. To circumvent this high-dimensionality, one can choose to instead use a POD discretization. For a good overview of POD refer to \cite{luchtenburg2009introduction, holmes2012turbulence,sirovich1987turbulence}.  Recently, as seen in \cite{qian2022reduced}, a non-intrusive method was devised to identify a system of ODEs over POD space. In \cite{qian2022reduced}, it was assumed that $F$ has a representative system of ODEs that was at most quadratic in the temporal POD modes. Here we exemplify that the weak-SINDy technique can act as a non-intrusive method to create a surrogate system of ODEs over POD space. We start with a brief review of POD.

The goal of POD is to efficiently represent (project) a solution $u(x,t)$ to a PDE in a finite number of modes so that,
\[u(x,t) \approx \sum_{i=1}^N s_i(t)u_i(x).\]
 Here, we assume that $x\in \Omega\subset\joinrel\subset \mathbb{R}^n$ is compact and $t\in [0,T]$. In this section, we call $\{u_i\}_{i=1}^N$ the spatial modes and $\{s_i\}_{i=1}^N$ the temporal modes. To do this, we define an operator kernel $R(x,y)$ by 
\begin{equation}\label{eq:R-formula}
R(x,y):=\frac{1}{T}\int_0^T u(x,t)\otimes u(y,t)\,dt\:,
\end{equation}
which can be considered as the time-average of $u(x,t)\otimes u(y,t)$. If this function $R(x,y)$ is continuous,  thus square integrable over the compact space $\Omega\times \Omega$, we have that $R$ is the kernel of a Hilbert-Schmidt integral operator. We can then define an operator $\mathcal{R}$ as
\begin{equation}\label{eq:R-operator_formula}
[\mathcal{R}v](x)=\int_{\Omega} R(x,y) v(y) \,dy \:,
\end{equation}
which is self-adjoint since $R(x,y) = R(y,x)$. Therefore, the spectral theorem guarantees that there exists an orthonormal basis $\{u_i(x)\}_{i=1}^\infty$ of eigenfunctions of $\mathcal{R}$, which are considered as the spatial modes of $u(x,t)$. 
By orthogonality of $\{u_i(x)\}_{i=1}^\infty$, the temporal modes can be defined as
\begin{equation}\label{eq:temporal-modes}
s_i(t) = \int_\Omega u(x,t) u_i(x)\, dx = \langle u(x,t),u_i(x)\rangle_\Omega,
\end{equation}
where the inner product is taken over $\Omega$. In this paper, we refer to the temporal modes defined by Eq.~\eqref{eq:temporal-modes} as the \emph{exact} temporal modes.

In the next sections, we will exemplify the use of the weak-SINDy technique in finding a surrogate system of ODEs that govern the temporal POD modes. We split this into two sections, one in which weak-SINDy uses the exact temporal modes as input data, and another in which proxy temporal modes are defined and used as input data. The difference between the two modes are demonstrated in the numerical experiments in Section~\ref{sec:Numerical Experiments}.

\subsection{Generating a surrogate ODE system for the exact POD modes}
\label{v1}

With the POD modes defined in Section~\ref{subsec:POD_discretization}, the solution can be expanded as $u(x,t) = \sum_{i=1}^\infty s_i(t)u_i(x)$. From Eq.~\eqref{eq:temporal-modes}, we have 
\begin{equation}
    \dot{s_i}(t)=\ip{\partial_t u(\cdot, t)}{u_i}_{\Omega}=\ip{F(u(\cdot, t))}{u_i}_{\Omega},\quad i=1, 2, \dots\:.
\end{equation}
Assuming that the evolution of the first $N$ temporal modes is governed by an ODE system with some dynamics $f\colon\mathbb{R}^N\to\mathbb{R}^N$, i.e.,
\begin{equation}\label{eq:exact_POD_system}
\dot{\mathbf{s}}(t)= f(\mathbf{s}(t))\:
\quad\text{with}\quad
\mathbf{s}(t):=\begin{bmatrix}s_1(t)& \cdots & s_N(t)\end{bmatrix}^T\:,
\end{equation}
then, when $f$ is either unknown or expensive to evaluate, the weak-SINDy method can be applied to construct a surrogate model for $f$ from data $\mathbf{s}$ as discussed in Section~\ref{sec:ode_system_analysis}.
 
We note that $f$ is \emph{not} an approximation of the differential operator $F$ that governs the solution $u$. However, constructing an approximation of $f$ allows us to model the temporal modes $\mathbf{s}(t)$, while the spatial modes $u_i(x)$ remain constant.

Applying the weak-SINDy method to Eq.~\eqref{eq:exact_POD_system} leads to the surrogate model
\begin{equation}\label{eq:PODmodel}
\textstyle
\dot{s_i}^{\dagger} = \sum_{\mathbf{j}=0}^J\bm{w}_{\mathbf{J}}^{(i)} \cdot \varphi_{\mathbf{j}}(\mathbf{s}^{\dagger})\:,\quad i = 1,\dots, N,
\end{equation}
where the weights $\bm{w}_{\mathbf{J}}^{(i)}$ are computed as given in Eq.~\eqref{eq:ODE_system_LS}. 
In the remainder of this paper, we refer to the solutions ${s_i}^{\dagger}$ of Eq.~\eqref{eq:PODmodel} as the \emph{surrogate} temporal modes, which can be used to construct an approximate PDE solution $u(x,t)\approx u^{\dagger}(x,t):=\sum_{i=1}^N s_i^\dagger (t) u_i(x)$.

\subsection{Generating a surrogate ODE system for proxy POD modes}\label{v2}
When the exact temporal modes $\mathbf{s}(t)$ are unavailable, they may be approximated by proxy temporal modes $\mathbf{s}^*(t)$ governed by the following equation
\begin{equation}
\textstyle
\dot{s}_i^*
=\left \langle \partial_t\left(\sum_{\ell=1}^N s_\ell^*(t) u_\ell\right), u_i\right \rangle_\Omega  
= \left\langle F\left(\sum_{\ell=1}^N s_\ell^*(t) u_\ell\right), u_i\right \rangle_\Omega\:, \quad i=1, \dots, N\:,  
\end{equation}
which can be written as the ODE system
\begin{equation}\label{eq:proxy_POD_system}
\textstyle
    \dot{\mathbf{s}}^* = g(\mathbf{s}^*)\:
    \quad\text{with}\quad 
    g_i := \left\langle F\left(\sum_{\ell=1}^N s_\ell^*(t) u_\ell\right), u_i\right \rangle_\Omega\:. 
\end{equation}
Similar to the case of using exact POD modes, applying the weak-SINDy method to Eq.~\eqref{eq:proxy_POD_system} leads to a surrogate model for the proxy POD modes.  

\begin{remark}
    We note that the exact and proxy POD modes are equivalent at the limit when the expansion degree $N\to\infty$. 
\end{remark}

\section{Numerical Experiments}\label{sec:Numerical Experiments}

In this section, we demonstrate the accuracy of surrogate models generated by the weak-SINDy method for several ODEs and verify the theoretical results proved in Sections~\ref{sec:ODE case} and \ref{sec:system of odes}. 
In sections~\ref{sec:smooth_ODE} and \ref{sec:sobolev_ODE}, we apply weak-SINDy to scalar ODEs with various regularity. In section~\ref{sec:POD_PDE}, we consider the case when weak-SINDy is applied to a system of ODEs from POD approximations of a PDE, as discussed in Section~\ref{sec:system of odes}.
The ODEs considered in this section are solved using the {\tt{odeint}} solver from the {\tt{scipy}} package in Python with a very restrictive maximum time step size ($10^{-6}$) to minimize the effect of numerical discretization errors.
The numerical solutions, $x(t)$, are then used to compute the matrix $\mathbf{G}$ and vector $\mathbf{b}$ defined in Eq.~\eqref{weak-SINDy_form}, where the integration-by-parts technique in Eq.~\eqref{eq:integration_by_parts} is used to avoid evaluations of $\frac{dx}{dt}$, and the integrals are performed using the function {\tt{simpson}} in the {\tt{scipy}} package \cite{2020SciPy-NMeth}.
The linear system \eqref{weak-SINDy_form} is then solved using the least squares solver {\tt{lstsq}} provided in the {\tt{numpy}} Python package \cite{harris2020array}.
In these numerical tests, monomials are used as the basis functions $\{\varphi_j\}_{j=0}^J$.
As for the test functions, we consider both the Legendre polynomials and the Fourier basis to demonstrate that the convergence results in the test space hold for any orthogonal test functions.
The choice of test functions $\{\psi_k\}_{k=0}^K$ will be specified in each experiment. 

\subsection{Smooth scalar ODE}\label{sec:smooth_ODE}
We first consider a simple example of a scalar ODE
\begin{equation}\label{eq:scalarODE_smooth}
    \dot{x} = f(x) := e^{-2x},\qquad x(0) = 0\:,
\end{equation}
over the time interval $[0,1]$, which has a smooth analytic solution $x_a(t) = \frac{1}{2} \ln(2t+1)$. 
To be consistent to other numerical tests in this section, we used a numerical solution $x(t)$ to Eq.~\eqref{eq:scalarODE_smooth} on a uniform grid in $t$ with step size $10^{-4}$ as the data for constructing surrogate models for Eq.~\eqref{eq:scalarODE_smooth}.

\subsubsection{Approximation errors}\label{sec:approximation_error_Legendre}
In this section, we use the normalized Legendre polynomials as test functions and explore the effect of the degrees $J$ and $K$ of the basis and test functions to various approximation errors of the resulting surrogate model.
To demonstrate the convergence properties of approximation errors, we first recall the main triangle inequality \eqref{eq:triangle_ineq} used in the analysis in Section~\ref{sec:ODE case},
\begin{align*}
\underbrace{\|f\circ x - p_{J,K}^*\circ x\|_{L^2([a,b])}}_{\text{(L)}}
\leq& 
\underbrace{\|f\circ x - \projK(f\circ x)\|_{L^2([a,b])}}_{\text{(R1)}}
+ \underbrace{\|\projK(f\circ x-p^*_{J,K}\circ x)\|_{L^2([a,b])}}_{\text{(R2)}}\\
&+\underbrace{\|p^*_{J,K}\circ x - \projK(p^*_{J,K}\circ x)\|_{L^2([a,b])}}_{\text{(R3)}},
\end{align*}
and explore how the choices of $J$ and $K$ affect the magnitude of each term.

Figure~\ref{fig:JKchanges} reports the total approximation error (L) in solid blue lines and the upper bound (R1)$+$(R2)$+$(R3) in dashed orange lines for $K=5,\,10,\,20$ and $J=1,\dots,30$.
In each figure, a vertical black dashed line is plotted at $J=K$. 
Figure~\ref{fig:all_three} shows the magnitudes of (R1), (R2), and (R3) under the same setting in green, blue, and red lines, respectively.

\begin{figure}[h]
    \centering
    \includegraphics[scale = 0.38]{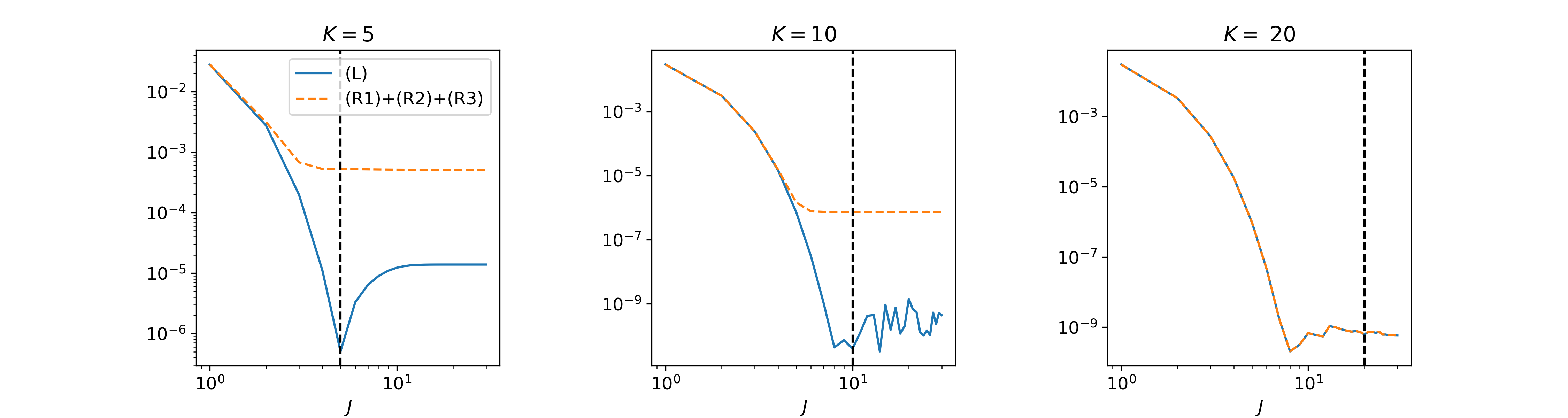}
    \caption{The total approximation error (L) [solid blue lines] and the upper bound (R1)$+$(R2)$+$(R3) [dashed orange lines] plotted versus the degree of basis functions $J=1,\dots,30$, for various degrees of test functions $K=$ 5, 10, and 20. A vertical black dashed line is plotted at $J=K$. }
    \label{fig:JKchanges}
\end{figure}
\begin{figure}[h]
    \centering
    \includegraphics[scale = 0.38]{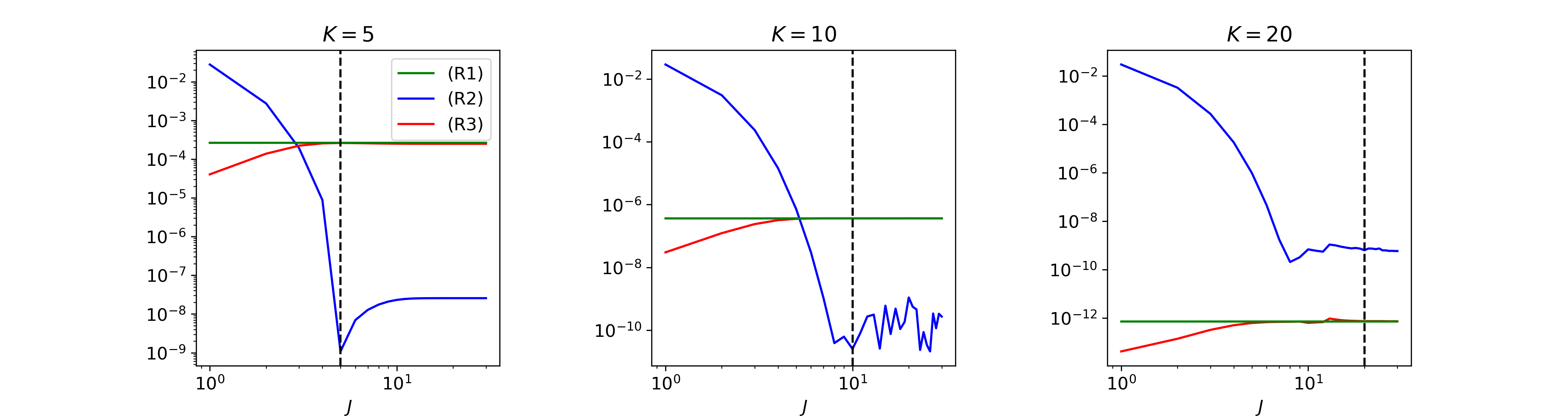}
    \caption{The error terms (R1), (R2), and (R3) plotted versus the degree of basis functions $J=1,\dots,30$ in green, blue, and red lines, respectively, for various degrees of test functions $K=$ 5, 10, and 20. A vertical black dashed line is plotted at $J=K$. }
    \label{fig:all_three}
\end{figure}

The results in Figure~\ref{fig:JKchanges} shows spectral convergence of the total approximation error (L) when $J\leq K$. We observe that the approximation error stops improving as $J$ increases when $J>K$, since the matrix $\mathbf{G}$ no longer satisfies the full column rank assumption considered in Section~\ref{subsec:assumption} and thus the linear system Eq.~\eqref{weak-SINDy_form} becomes an under-determined system.
From Figure~\ref{fig:all_three}, we notice that (L) closely reflects the approximation error in the test space (R2), which decays spectrally in this smooth example as stated in Proposition~\ref{prop:minproblem}. In the case when $K=20$, the value of (R2) stops decreasing at around $10^{-9}$ due to numerical integration errors. When the value of $K$ is low, the upper bound is dominated by (R1) and (R3), and the values of which decrease when $K$ grows as indicated in Lemma~\ref{lem:polyapproxindeg} and Proposition~\ref{prop:R3_bound}.

\subsubsection{Approximation error with Fourier test functions}\label{sec:approximation_error_Fourier}
In this section, we repeat the tests in Section~\ref{sec:approximation_error_Legendre} with the test functions chosen to be part of the standard Fourier basis given by $\{\psi_k(t)\} := \{1\}\cup\{ \sqrt{2}\cos(2\pi k t)\}_{k=1}^K\cup\{\sqrt{2}\sin(2\pi k t) \}_{k=1}^{K}$. 

Figure~\ref{fig:fourier} shows the values of the total approximation error (L), the individual terms (R1), (R2), (R3), and the upper bound (R1)$+$(R2)$+$(R3) for $J=1,\dots,30$ and $K=10$, which results in $2K+1=21$ Fourier test functions.

\begin{figure}[h]
    \centering
     \begin{subfigure}{.5\textwidth}
    \includegraphics[width=.9\linewidth]{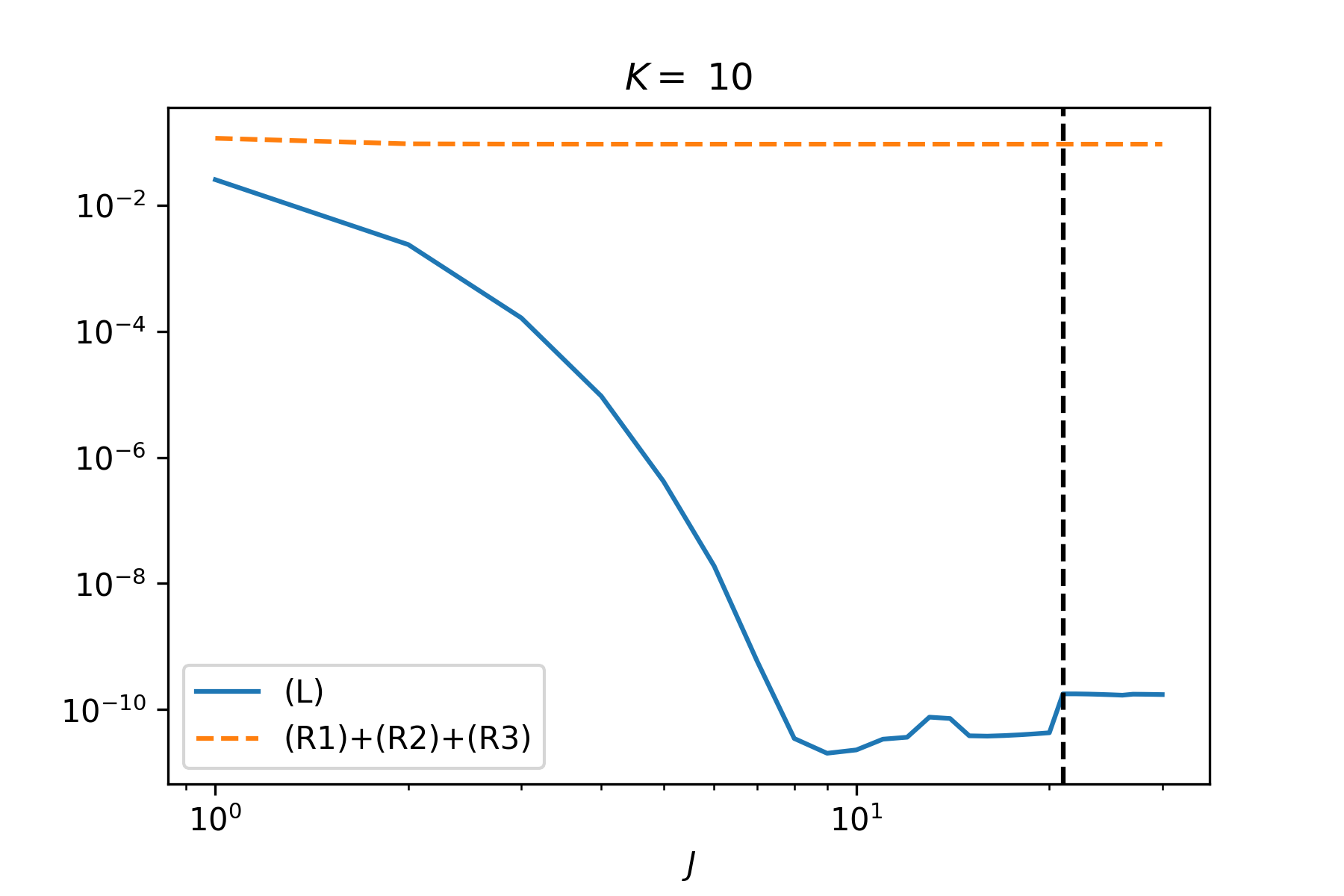}
    \caption{Total approximation error and upper bound}
    \end{subfigure}%
    \begin{subfigure}{.5\textwidth}
    \includegraphics[width=.9\linewidth]{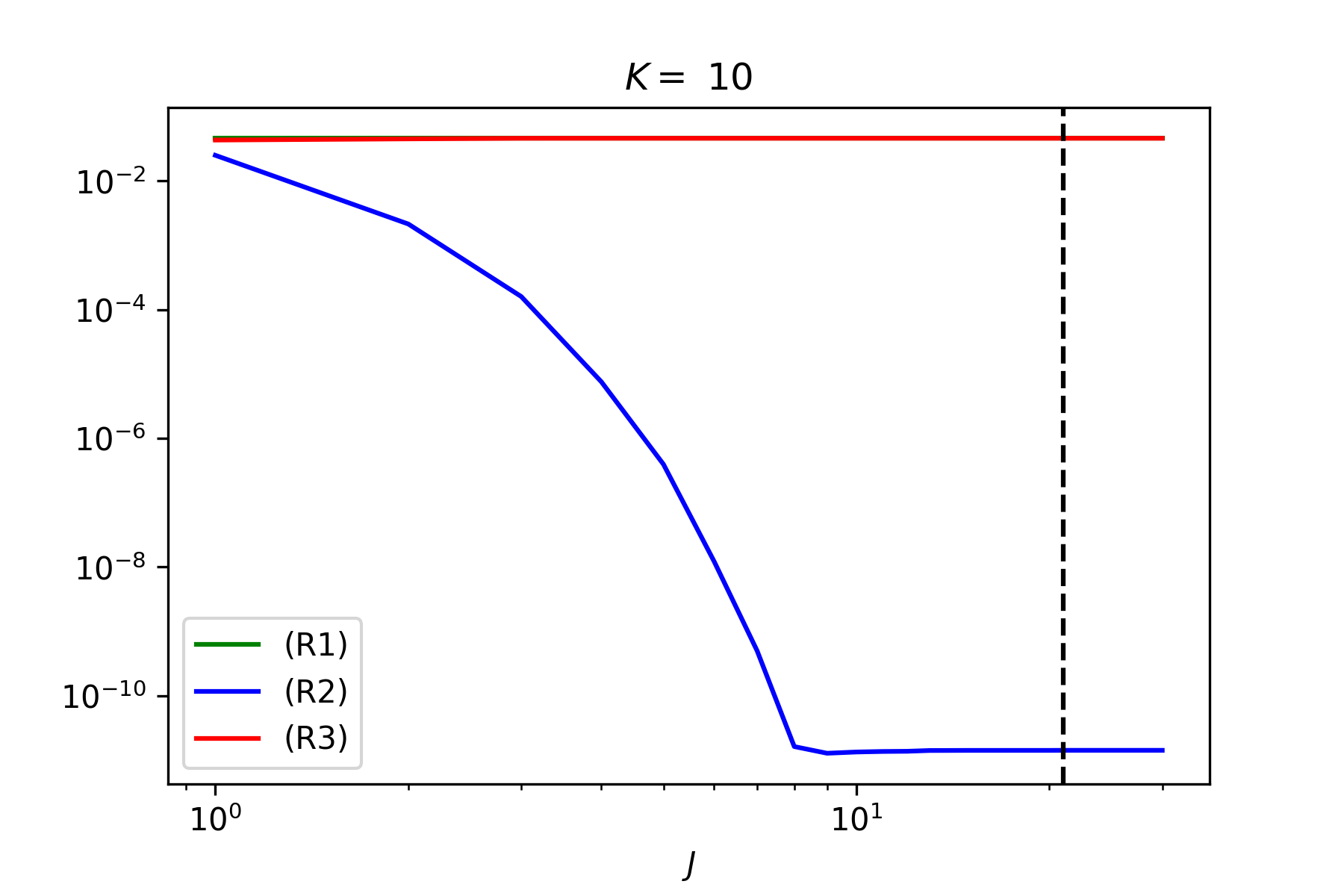}
    \caption{Individual error terms}
    \end{subfigure}
    \caption{Total approximation error (L), the individual terms (R1), (R2), (R3), and the upper bound (R1)$+$(R2)$+$(R3) plotted versus the degree of basis functions $J=1,\dots,30$ with the degree of Fourier basis $K=10$, which results in $2K+1=21$ Fourier test functions.}
    \label{fig:fourier}
\end{figure}

From the results in Figure~\ref{fig:fourier}, we observe that the total error (L) still closely resembles the behavior of the projection error (R2) in the test space. In the case of Fourier test functions, the value of (R2) still converges spectrally when $J\leq K$ as when using Legendre polynomials as test functions (cf. Figure~\ref{fig:all_three}), since the convergence rate of (R2) is independent to the choice of test functions, as shown in Proposition~\ref{prop:minproblem}.
We also note that in Figure~\ref{fig:fourier}, (R2) saturates at around $10^{-11}$ due to numerical integration errors.
Since the dynamics $f$ in Eq.~\eqref{eq:scalarODE_smooth} is not periodic, expansions with Fourier test functions converge much slower than the ones using Legendre polynomials, which significantly affects the values of (R1) and (R3). As shown in Figure~\ref{fig:fourier}, the error bound is still dominated by (R1) and (R3) even when $2K+1=21$ Fourier test functions were used. The terms (R1) and (R3) are only expected to decrease when $K$ increases and are not expected decrease when $J$ increases.

\subsubsection{Convergence of the solution}\label{experiment1c}
In this section, we verify the convergence property in Proposition~\ref{prop:Lipschitz}, which bounds the difference between $x$ and $\hat{x}$, the solutions to the ODE Eq.~\eqref{eq:scalarODE_smooth} and the associated surrogate model $\dot{\hat{x}}=p^*(\hat{x})$, respectively.
Specifically, we demonstrate that, in a time interval $[a,b]$ that satisfies $L(b-a)<1$ with $L$ the Lipschitz constant of $p^*$, the error $\|\hat{x} - x\|_{L^{\infty}([a,b])}$ is indeed bounded by $\sfrac{\varepsilon (b-a)^{1/2}}{(1-L(b-a))}$, where $\varepsilon = \|f\circ x - p\circ x\|_{L^2[a,b]}$ as given in Proposition~\ref{prop:Lipschitz}.
Here the test functions are chosen to be the Fourier basis as in Section~\ref{sec:approximation_error_Fourier}. We focus on the case that $K=20$ and $J=5$ with the approximate polynomial $p_{J=5, K=20}^*$ from weak-SINDy denoted as $p^*$ for simplicity. In this experiment, the solution to Eq.~\eqref{eq:scalarODE_smooth} was calculated over the time interval $[0,3]$ using a uniform grid with step-size $3\times 10^{-4}$.

To demonstrate the bound in Proposition~\ref{prop:Lipschitz}, we set a threshold value $s \leq 1$ and compute $\tau\in [0,3]$ such that $\tau\,L = s$, where the Lipschitz constant $L$ of $p^*$ is computed over the range of the analytic solution $x(t)= \frac{1}{2}\ln(2t+1)$ over time interval $[0,3]$. 
With $\varepsilon = \|f\circ x - p^*\circ x\|_{L^2([0,3])}$, the error bound over time frame $[0,\tau]$ is then given by $\sfrac{\varepsilon\cdot \tau^{1/2}}{(1-s)}$. 

Figure~\ref{fig:Linf} shows the computed solution error $|x(t)-\hat{x}(t)|$ and the upper bound provided by Proposition~\ref{prop:Lipschitz} over time interval $[0,\tau]$. 
Here the upper bound is calculated with $\varepsilon\approx 0.002$, $s=0.8$, and $\tau\approx 0.4$.
The result validates the upper bound while suggesting that the estimate may not be sharp.
 
\begin{figure}[h]
    \centering
    \includegraphics[scale = 0.45]{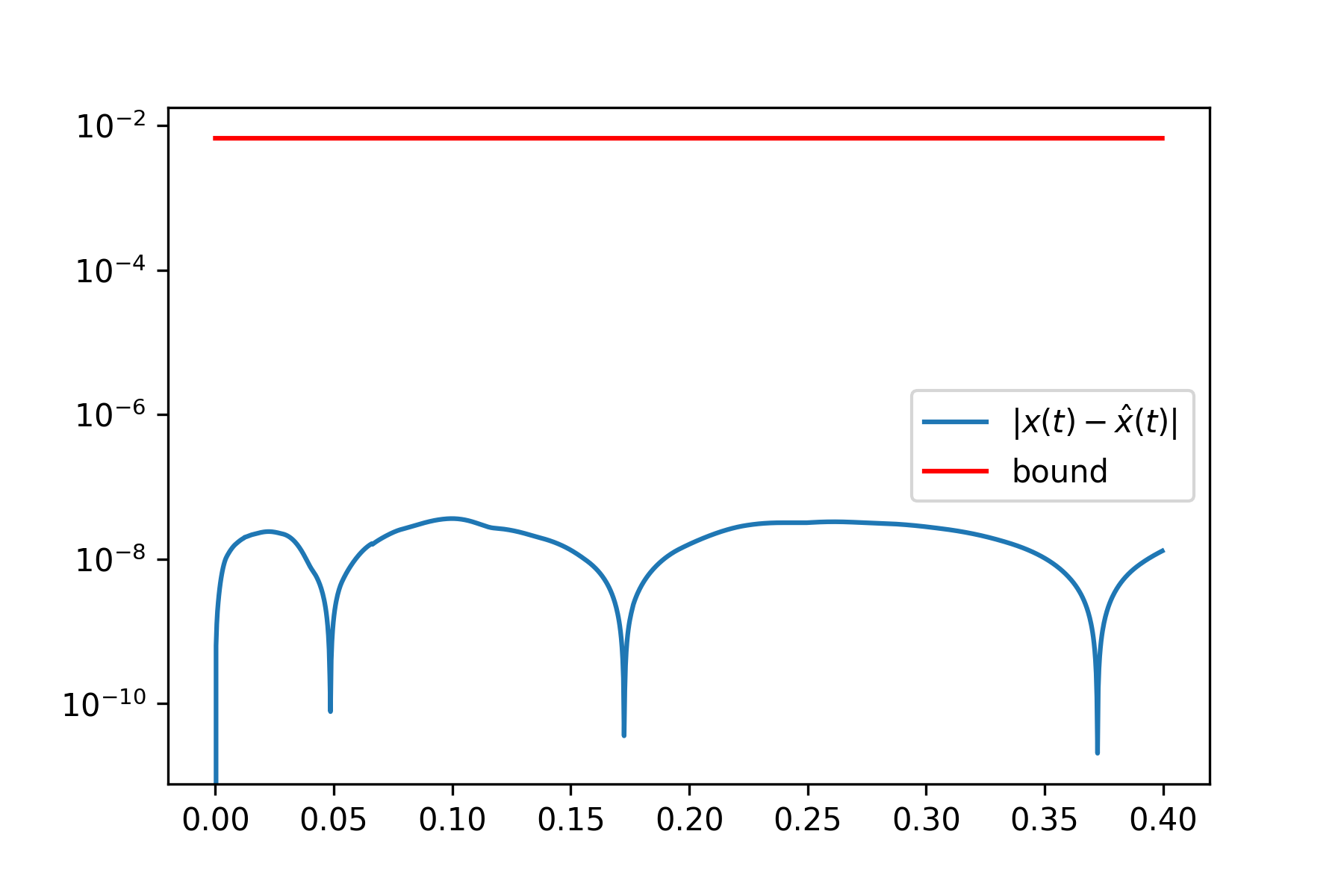}
    \caption{Solution error $|x(t)-\hat{x}(t)|$ (blue) and the estimate (red) given in Proposition~\ref{prop:Lipschitz} are plotted on time interval  $[0,0.4]$.}
    \label{fig:Linf}
\end{figure}

\subsection{Scalar ODE -- varying regularity}
\label{sec:sobolev_ODE} 
We next consider scalar ODEs of the form
\begin{equation}\label{eq:sobolev_ode}
\dot{x}(t) = g_\alpha(x(t)),\quad x(0) = 2, \quad\text{with}\quad 
g_\alpha(x) = \mathbf{1}_{[-\frac{\pi}{2},\frac{\pi}{2}]}\cos^\alpha(x)-\frac{1}{2}\:,
\end{equation}
where $\mathbf{1}$ denotes the indicator function and $\alpha\geq0$ is a tunable parameter for the regularity of $g_\alpha$. Specifically, it can be shown that $g_\alpha\in H^{\alpha+\sfrac{1}{2}}(X)$ for $\alpha\geq0$. 
To ensure that the solution $x$ exercises the non-smooth part of the dynamics $g_\alpha$, we solve Eq.~\eqref{eq:sobolev_ode} on time interval $[0,2]$ on a uniform grid in $t$ with step size $10^{-4}$.  

To highlight one of the significant changes the regularity of $g_\alpha$ has on the projection error of weak-SINDy, we calculate $\|\projK(g_\alpha\circ x-p_{J,K}^*\circ x )\|_{L^2([0,2])}$ ((R2) term in Section~\ref{sec:approximation_error_Legendre}) as $J$ increases. 
Here Legendre polynomials were used for both the test functions $\{\psi_k\}_{k=0}^{20}$ and the basis functions $\{\varphi_j\}_{j=0}^J$, $J\in \{1,2\ldots, 20\}$.
For each $\alpha \in \{1,2,3,4\}$, Figure~\ref{fig:function-plots} plots $g_\alpha$ over the domain $[-2, 2]$. Figure~\ref{fig:varying_alpha-L2 diff} plots the projection error (R2) at $K = 20$ and varying basis function degree $J= 1,2,\ldots, 20$ for each $\alpha$, with additional dashed lines indicating the expected convergence rate in $J$ from Proposition~\ref{prop:minproblem}.  

The observed results confirms that the convergence rate for the projection error given in Proposition~\ref{prop:minproblem}.  

\begin{figure}[h]
\centering
\begin{subfigure}{.5\textwidth}
  \centering
  \includegraphics[width=.9\linewidth]{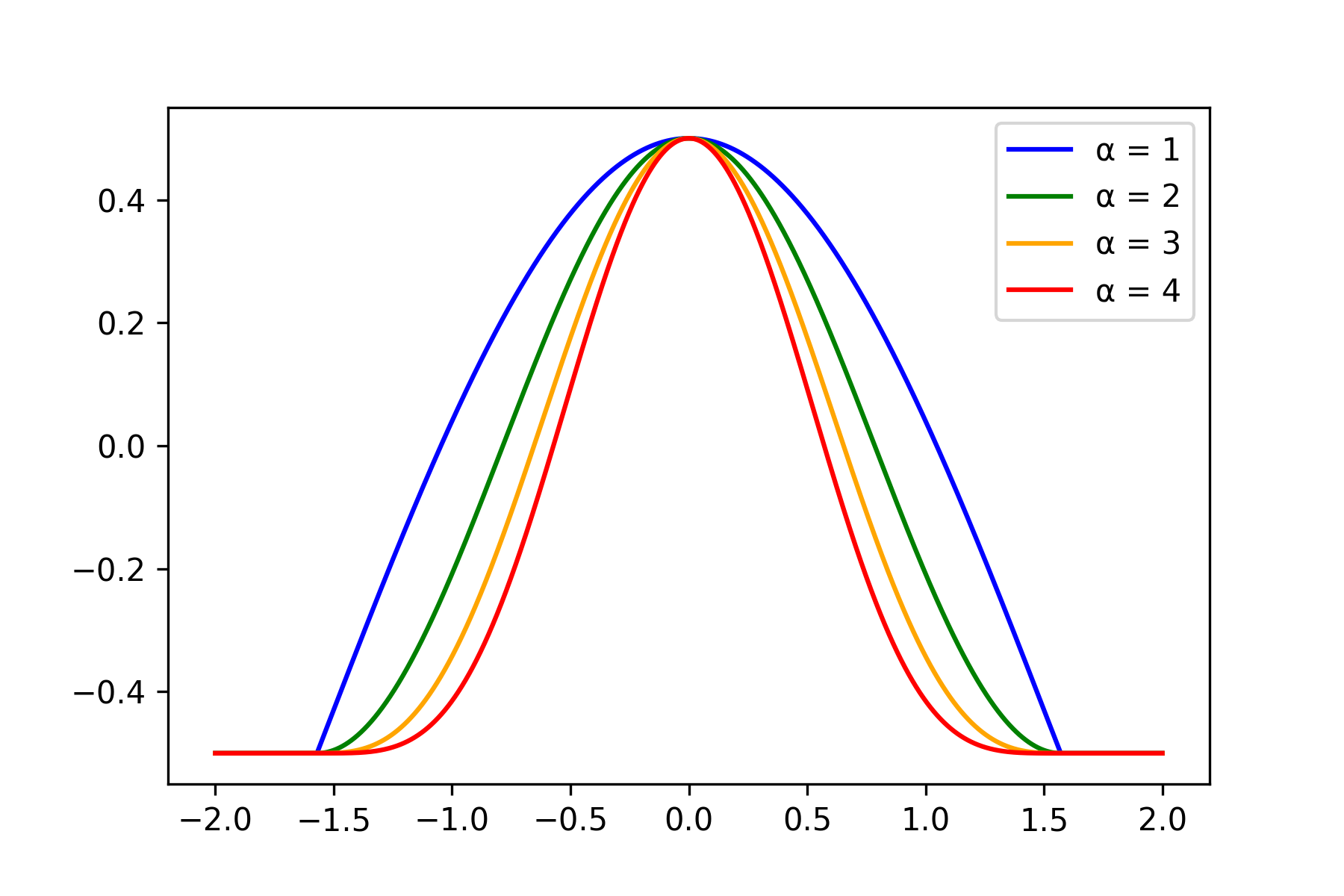}
  \caption{Dynamics $g_\alpha$}
  \label{fig:function-plots}
\end{subfigure}%
\begin{subfigure}{.5\textwidth}
  \centering
  \includegraphics[width=.9\linewidth]{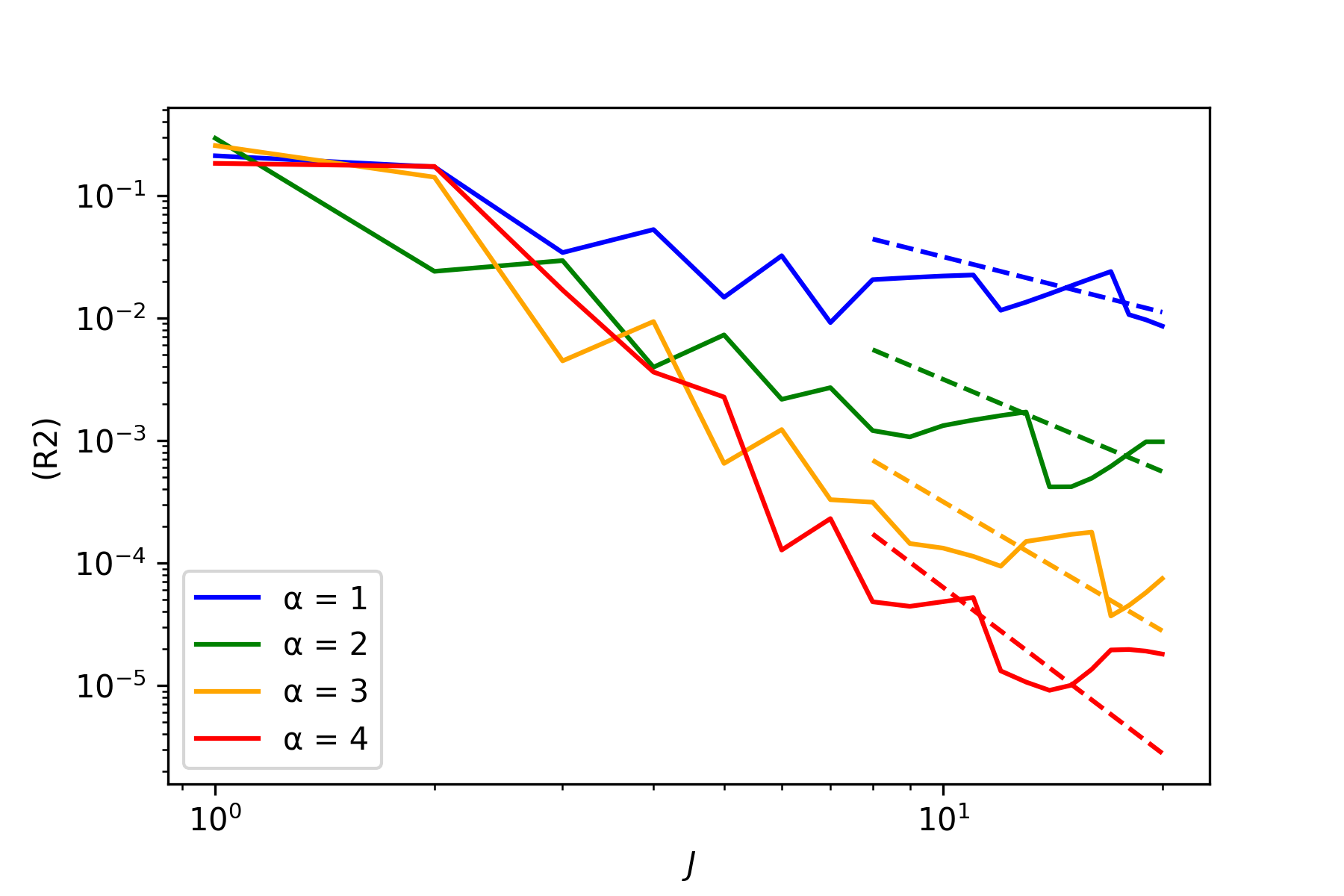}
  \caption{Projection error (R2) for each $g_\alpha$}
  \label{fig:varying_alpha-L2 diff}
\end{subfigure}
\caption{Figure~\ref{fig:function-plots} plots $g_\alpha$ over the domain $[-2, 2]$ for $\alpha \in \{1,2,3,4\}$. Figure~\ref{fig:varying_alpha-L2 diff} plots the projection error (R2) at $K = 20$ and varying degree $J= 1,2,\ldots, 20$ for each $\alpha$, with dashed lines indicating the expected convergence rate of (R2) in $J$ from Proposition~\ref{prop:minproblem}.  
}

\label{fig:fig2}
\end{figure}

\subsubsection{Approximation Errors -- varying regularity}\label{ex:approx_errors_var_reg}
In this section, we carry out the same computations as in Section~\ref{sec:approximation_error_Legendre} but on Eq.~\eqref{eq:sobolev_ode} with $\alpha = 1$ and $\alpha = 4$. Figures~\ref{fig:alpha1_errorbound} and \ref{fig:alpha4_errorbound} reports the total approximation error (L) in solid blue lines and the upper bound (R1)$+$(R2)$+$(R3) in dashed orange lines for $K=5,\,10,\,20$ and $J=1,\dots,30$. In each figure, a vertical black dashed line is plotted at $J=K$. Figures~\ref{fig:alpha1_all_three} and \ref{fig:alpha4_all_three} shows the magnitudes of (R1), (R2), and (R3) under the same setting in green, blue, and red lines, respectively. Figures~\ref{fig:alpha1_all_three} and \ref{fig:alpha4_all_three} also feature a purple dashed line to indicate the expected convergence rates. The discussion of Section \ref{sec:approximation_error_Legendre} carries over into this section as well. 

\begin{figure}[h]
    \centering
    \includegraphics[scale = 0.38]{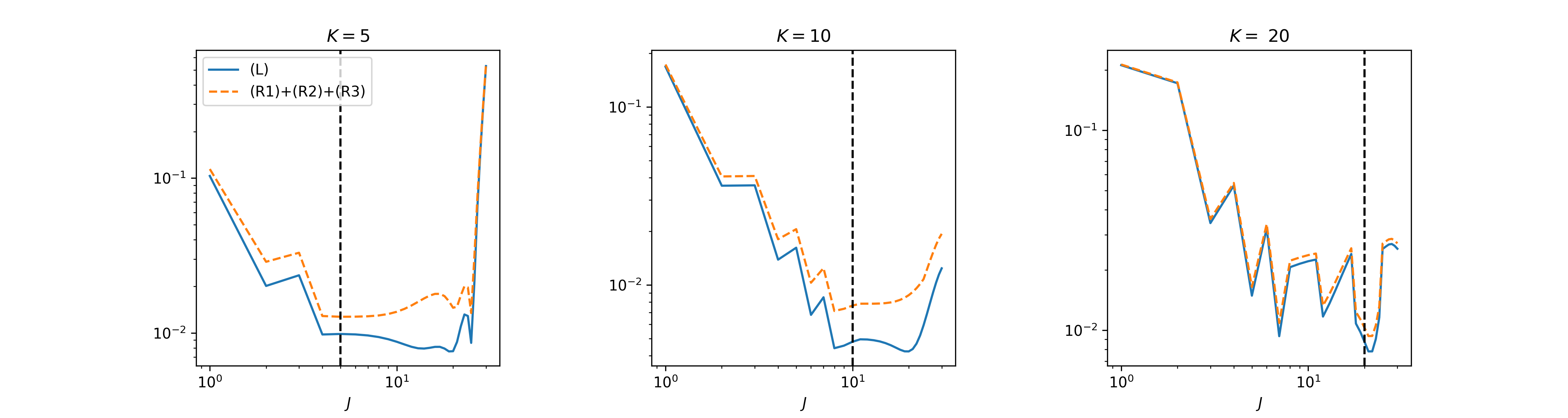}
    \caption{For $\alpha = 1$, the total approximation error (L) [solid blue lines] and the upper bound (R1)$+$(R2)$+$(R3) [dashed orange lines] are plotted versus the degree of basis functions $J=1,\dots,30$, for various degrees of test functions $K=$ 5, 10, and 20. A vertical black dashed line is plotted at $J=K$.}
    \label{fig:alpha1_errorbound}
\end{figure}

\begin{figure}[h]
    \centering
    \includegraphics[scale = 0.38]{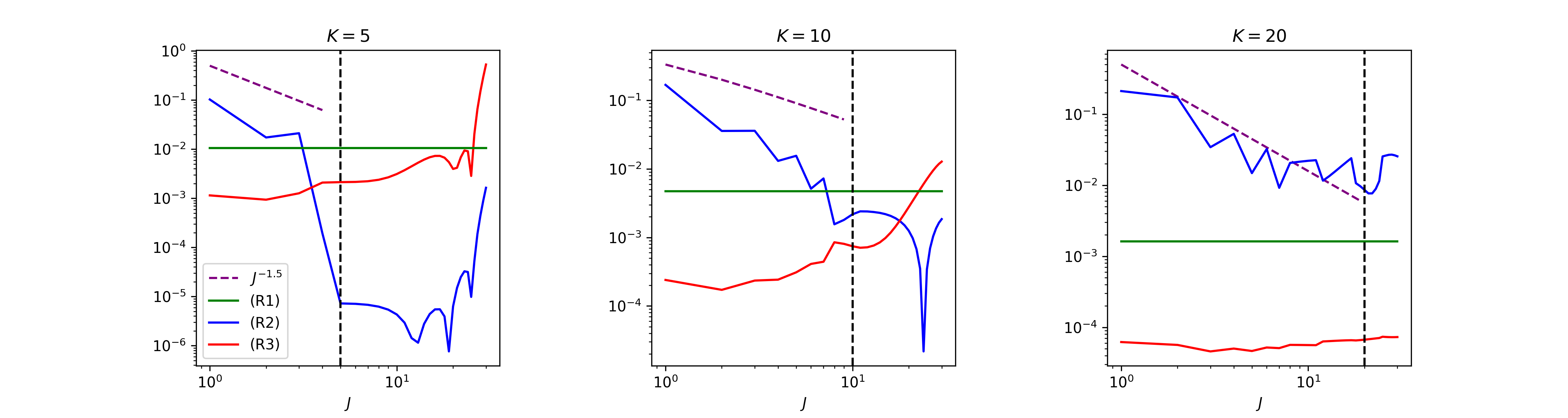}
    \caption{With $\alpha = 1$ the error terms (R1), (R2), and (R3) are plotted versus the degree of basis functions $J=1,\dots,30$ in green, blue, and red lines, respectively, for various degrees of test functions $K=$ 5, 10, and 20. A vertical black dashed line is plotted at $J=K$ and a purple dashed line indicates the expected convergence rate.}
    \label{fig:alpha1_all_three}
\end{figure}

\begin{figure}[h]
    \centering
    \includegraphics[scale = 0.38]{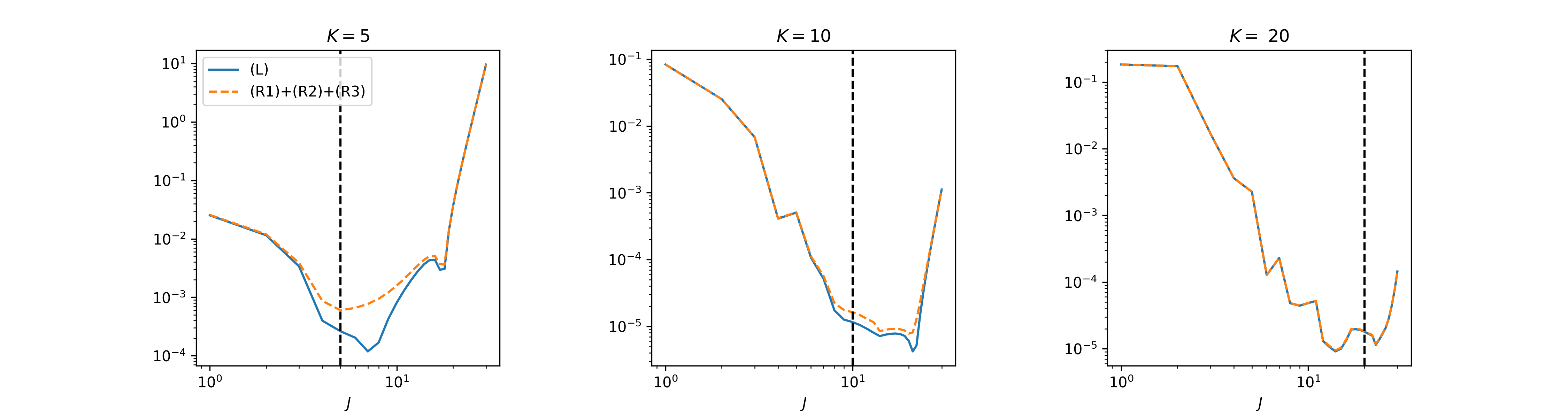}
    \caption{For $\alpha = 4$, the total approximation error (L) [solid blue lines] and the upper bound (R1)$+$(R2)$+$(R3) [dashed orange lines] are plotted versus the degree of basis functions $J=1,\dots,30$, for various degrees of test functions $K=$ 5, 10, and 20. A vertical black dashed line is plotted at $J=K$.}
    \label{fig:alpha4_errorbound}
\end{figure}

\begin{figure}[h]
    \centering
    \includegraphics[scale = 0.38]{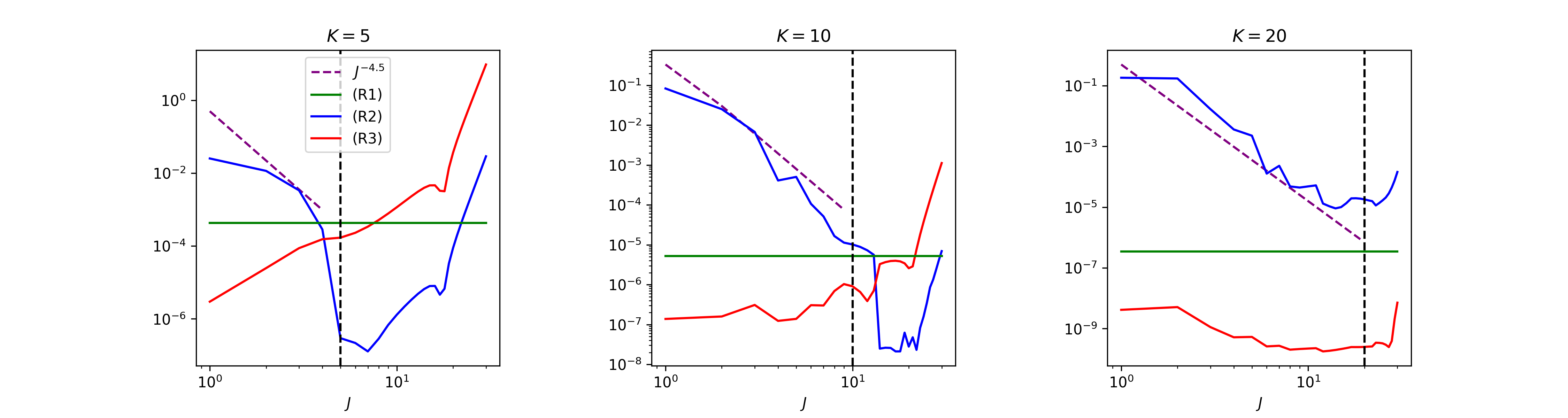}
    \caption{With $\alpha = 4$ the error terms (R1), (R2), and (R3) are plotted versus the degree of basis functions $J=1,\dots,30$ in green, blue, and red lines, respectively, for various degrees of test functions $K=$ 5, 10, and 20. A vertical black dashed line is plotted at $J=K$ and a purple dashed line indicates the expected convergence rate.}
    \label{fig:alpha4_all_three}
\end{figure}

\subsection{Weak-SINDy surrogate for POD discretization of a PDE}\label{sec:POD_PDE}
In this section we consider a one-dimensional diffusion equation: 
\begin{equation}\label{eq:1D_diffusion}
\partial_t u  = \beta(x) \partial_{xx}u 
\end{equation}
with initial and boundary conditions given by
\begin{equation}\label{eq:1D_diffusion_bdry_cond}
u(x, 0) = x + \sin(2\pi x) + 1, \qquad u(0, t) = 1, \qquad u(1,t) = 2, 
\end{equation}
over the domain $\Omega\times [0,T] = [0,1]\times [0,10]$. In each experiment, the solution $u(x,t)$ was numerically solved for using a FTCS scheme \cite{leveque2002finite}. The intervals $[0,1]$ and $[10,10]$ were divided into uniform grids with step size of $10^{-2}$ and $10^{-3}$ respectively. Step-sizes were chosen to maintain stability of the FTCS scheme with the chosen diffusivities. With the numerical solution $u(x,t)$, the spatial POD modes are calculated by solving a discretized version of Eq.~\eqref{eq:R-operator_formula} as a generalized eigenvalue problem using the \texttt{eigh} function in \texttt{scipy.linalg}.
The exact temporal modes are computed using Eq.~\eqref{eq:temporal-modes}, where the integrals are performed using {\tt{scipy.integrate.simpson}}. 
Again, the ODEs considered in this section are solved using the {\tt{scipy.integrate.odeint}} solver with a maximum time step size ($10^{-5}$) to minimize the effect of numerical discretization errors. For consistency between experiments, all POD approximations were done using two spatial modes.

We apply the weak-SINDy method to construct surrogate models for the exact and proxy temporal modes in Sections~\ref{sec:true_podmodes_constant_beta} and \ref{sec:proxy_podmodes_constant_beta}, respectively for Eq.~\eqref{eq:1D_diffusion} with constant diffusivity factor $\beta$.
In Section~\ref{sec:true_podmodes_varying_beta}, we construct a weak-SINDy surrogate model for the exact temporal modes of Eq.~\eqref{eq:1D_diffusion} with $\beta$ a discontinuous function in space.
In each experiment, a Fourier test basis  $\{\psi_k(t)\} := \{1\}\cup\{ \sqrt{2}\cos(2\pi k t)\}_{k=1}^K\cup\{\sqrt{2}\sin(2\pi k t) \}_{k=1}^{K}$ for $K=40$ and a monomial projection basis $\{\varphi_{\vec{j}}(x)\} := \{x_1^{j_1}\cdot \ldots \cdot x_N^{j_N}\}_{\vec{j}=(j_1,\ldots, j_N)}$ was used to generate the weak-SINDy surrogate models given by Eq.~\eqref{eq:PODmodel}.

\subsubsection{Exact POD modes of 1D diffusion equation with constant diffusivity}\label{sec:true_podmodes_constant_beta}
In this experiment, we consider Eq.~\eqref{eq:1D_diffusion} with a constant $\beta(x) := 5\times 10^{-3}$. Applying the weak-SINDy method to the exact temporal modes with max basis degree $J=1$ leads to the following coupled polynomial ODE model
\begin{equation}\label{eq:surrogate_model_pod_constant_beta}
\begin{alignedat}{2}
  \dot{s}_0 &=  7.56\times 10^{-4} -1.76\times 10^{-3}\,s_0 \, -1.09\times10^{-2}\, s_1\, -1.18\times 10^{-10}\, s_0s_1\\
  \dot{s}_1 &=   1.35\times 10^{-2} -3.14\times 10^{-2}\,s_0 \,  -1.96\times 10^{-1}\, s_1 -2.12\times 10^{-9}\,s_0s_1\:.
\end{alignedat}
\end{equation}
We denote the solution to Eq.~\eqref{eq:surrogate_model_pod_constant_beta} as $s_i^\dagger$, which is referred to as the surrogate temporal modes as introduced in Section \ref{v1}. 

In Figure~\ref{fig: diffusion solution and POD approxs}, we plot the original solution $u$, the POD reconstruction from spatial mode $u_i$ and exact POD mode $s_i$, and the surrogate solution $u^\dagger$ from spatial mode $u_i$ and surrogate temporal mode $s_i^\dagger$.

\begin{figure}[h]
\centering
\subfloat[Original solution $u$]{%
\centering\label{fig:original_soln}\includegraphics[width=0.3\textwidth]{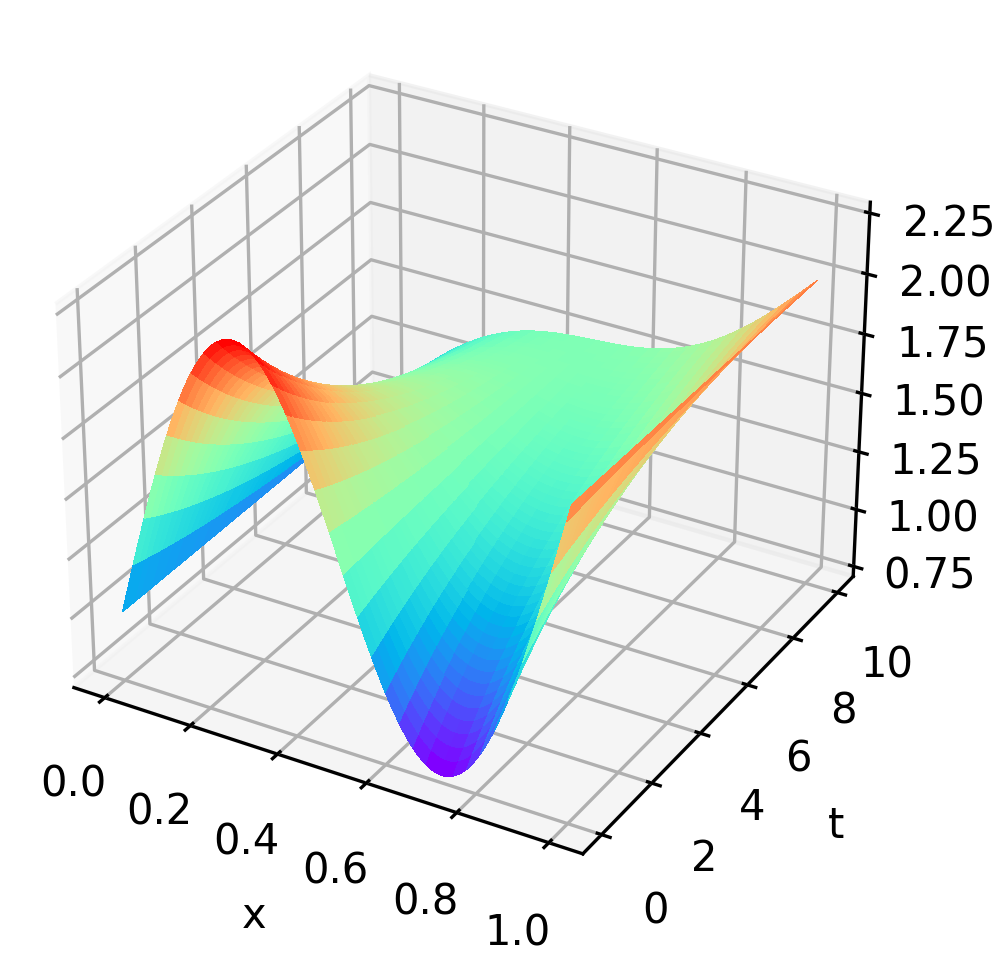}}
\quad
\subfloat[POD reconstruction of $u$]{%
\centering \label{fig:original_pod}\includegraphics[width=0.3\textwidth]{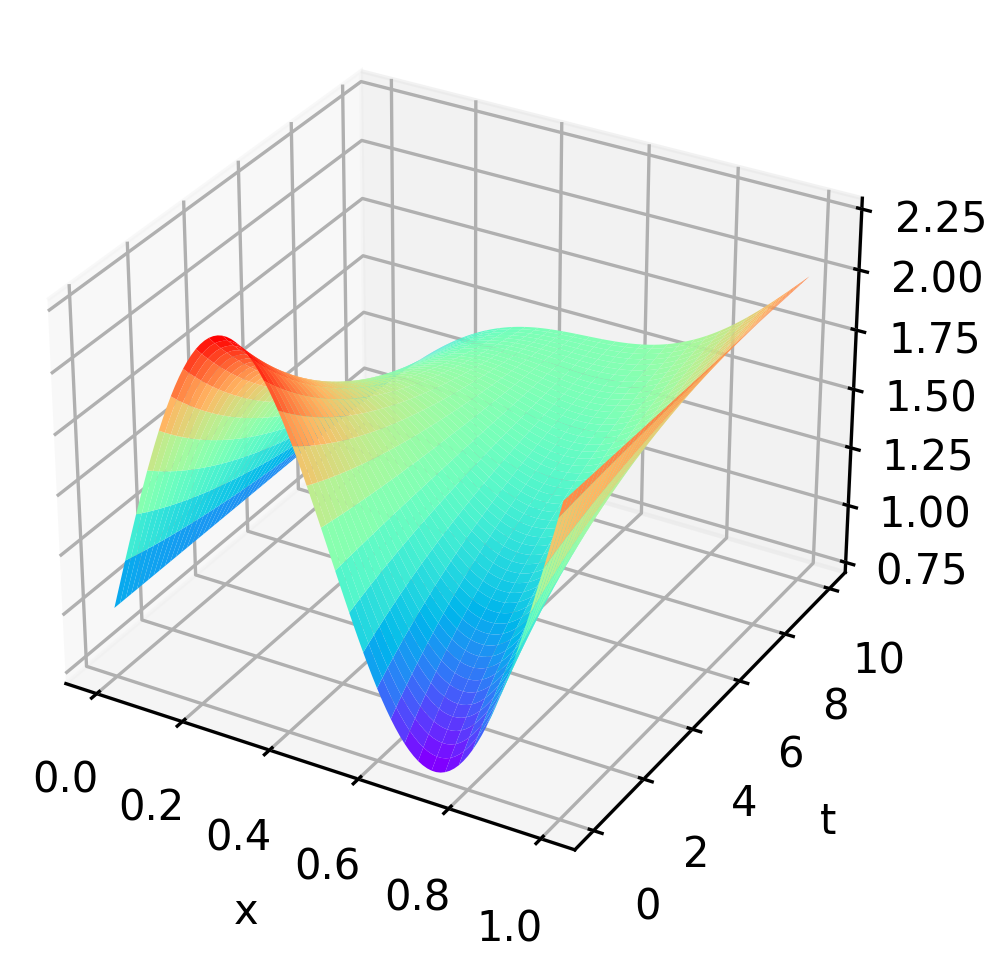}}
\quad
\subfloat[Surrogate POD approximation $u^\dagger$]{%
\centering\label{fig:pod_surrogate}\includegraphics[width=0.3\textwidth]{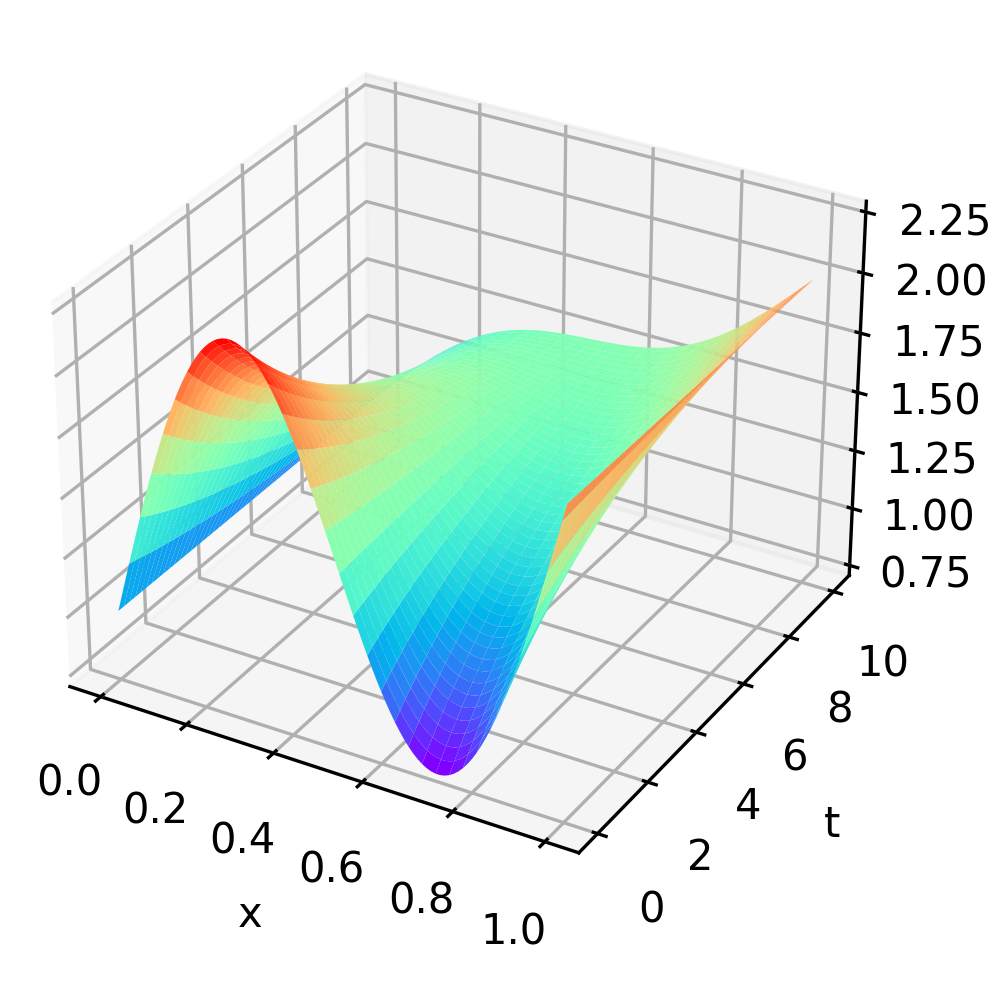}}
\caption{The original solution $u(x,t)$ to Eq.~\eqref{eq:1D_diffusion} with constant $\beta$, its POD reconstruction, and a POD approximation $u^\dagger$ from the surrogate temporal modes $s_i^\dagger$. }
\label{fig: diffusion solution and POD approxs}
\end{figure}

In Figure \ref{fig:Error_diff}, we plot the heat maps for the error between $u$ and its POD reconstruction and the error between the POD reconstruction and the surrogate solution $u^\dagger$ at each $x$ and $t$ in a logarithmic scale. 
For this tame example, we see that the POD reconstruction nearly replicates the original solution with a small number of modes, and the surrogate model provided by the weak-SINDy method gives accurate approximations to the exact temporal modes.

\begin{figure}[h]
\centering
\subfloat[Error between the POD reconstruction and $u(x,t)$]{%
\centering\label{fig:error - pod and original}\includegraphics[width=0.35\textwidth]{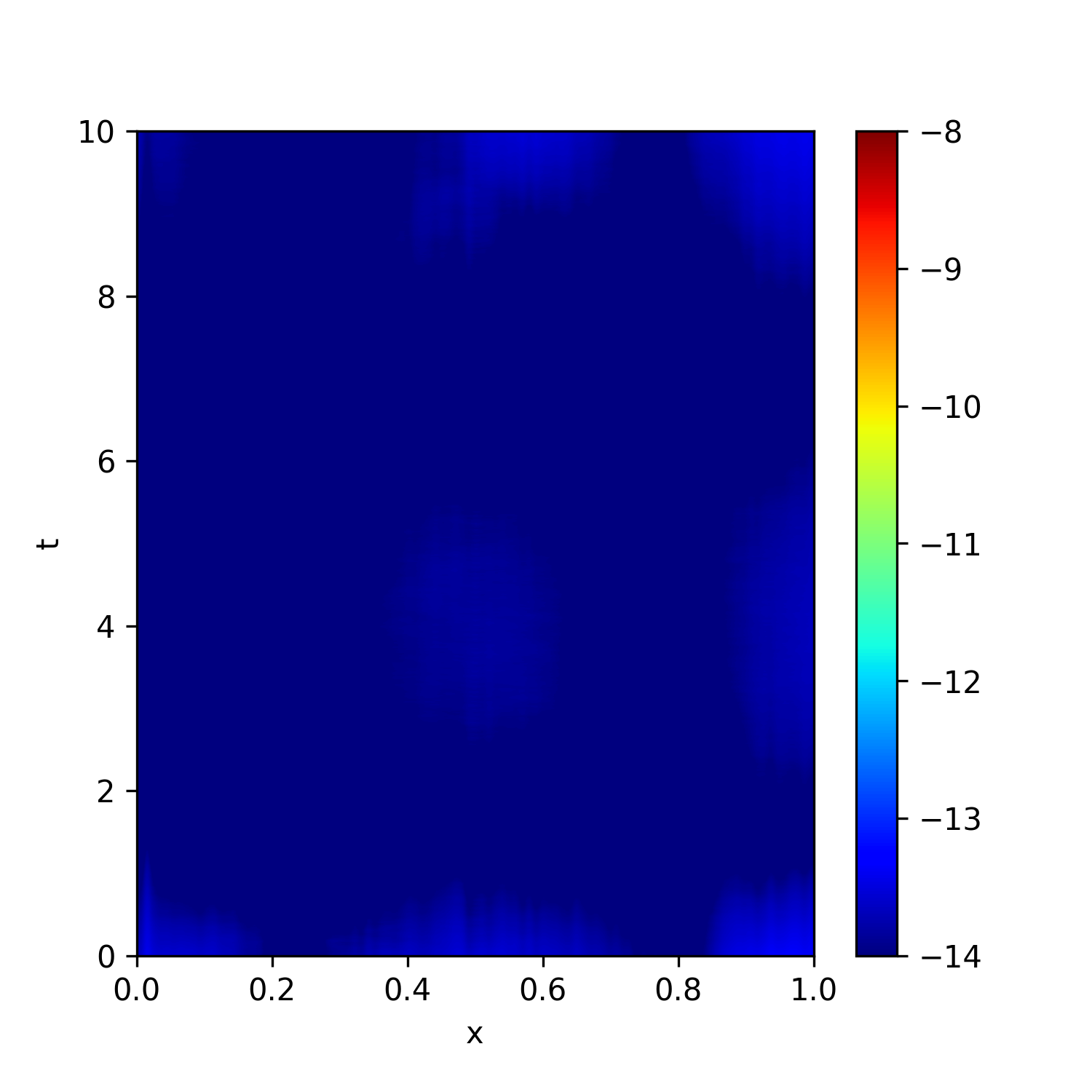}}
\qquad
\subfloat[Error between the surrogate POD approximation $u^\dagger$ and the POD reconstruction of $u$]{%
\centering\label{fig:error - pod and surrogate} \includegraphics[width=0.35\textwidth]{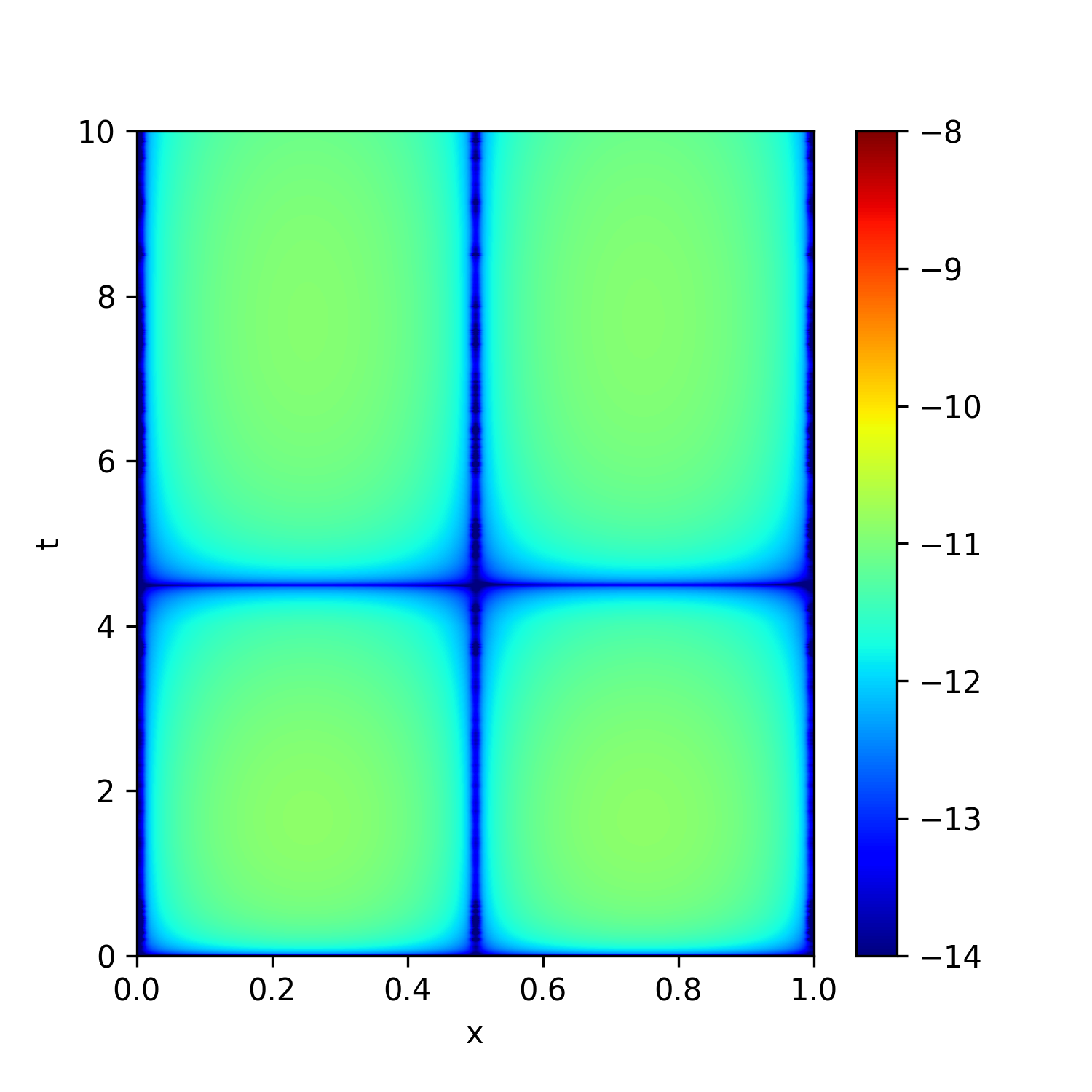}}
\caption{
The error between $u$ and its POD reconstruction (Figure~\ref{fig:error - pod and original}) and the error between the POD reconstruction and the surrogate solution $u^\dagger$ (Figure~\ref{fig:error - pod and surrogate}) at each $x$ and $t$ in a logarithmic scale. }
\label{fig:Error_diff}
\end{figure}

\subsubsection{Proxy POD modes of 1D diffusion equation with constant diffusivity}
\label{sec:proxy_podmodes_constant_beta}
In this section, we will perform the same experiment as in Section~\ref{sec:true_podmodes_constant_beta} but using the proxy temporal modes defined in Section~\ref{v2}. The proxy temporal modes are obtained by solving the ODE system
\begin{equation}
\label{proxypodeq}
\textstyle
\dot{s}_i^*= \beta \cdot \left \langle \sum_{\ell = 1}^N s_\ell^*(t)\partial_{xx}u_\ell(x), u_i\right\rangle_\Omega\:,\quad i=1,\dots,N\:,
\end{equation}
with the initial condition given in Section~\ref{sec:true_podmodes_constant_beta}. 
Here the derivatives are calculated via fourth order finite differences. 
Applying the weak-SINDy method to the proxy temporal modes with max basis degree $J=1$ leads to the following coupled polynomial ODE model 
 
\begin{equation}\label{eq:surrogate_model_proxy_pod_constant_beta}
\begin{alignedat}{2}
  \dot{s}_0 &=  7.65\times 10^{-4} -1.79\times 10^{-3}\,s_0\, -1.11\times 10^{-2}\, s_1\, -1.20\times 10^{-10}\, s_0s_1\\
  \dot{s}_1 &=  1.34\times 10^{-2}  -3.14 \times 10^{-2}\,s_0\, -1.96\times 10^{-1}\, s_1\, -2.08\times 10^{-9}\, s_0s_1 \:.
\end{alignedat}
\end{equation}
The solutions to Eq.~\eqref{eq:surrogate_model_proxy_pod_constant_beta} are again denoted by ${s}_i^\dagger$, which are the surrogate temporal modes from the proxy POD modes $s_i^*$.

In Figure~\ref{fig:proxy_error_comparison}, we plot the heat maps for the error between $u$ and the proxy POD reconstruction from proxy temporal modes $\mathbf{s}^*$ and the error between the proxy POD reconstruction and the surrogate solution $u^\dagger$ at each $x$ and $t$ in a logarithmic scale. 
As opposed to the case for exact temporal modes considered in Figure~\ref{fig:Error_diff}, here the proxy POD reconstruction error is observed to increase over time and is much greater than the surrogate error from weak-SINDy method, i.e., the surrogate POD approximation $u^\dagger$ is of similar quality to the proxy POD reconstruction from $\mathbf{s}^*$, which justifies the use of weak-SINDy method for constructing surrogate models for proxy POD modes.

\begin{figure}[h]
\centering
\subfloat[Error between the proxy POD reconstruction and $u(x,t)$]{%
\centering\label{fig:proxy_pod_U_error}
\includegraphics[width=0.35\textwidth]{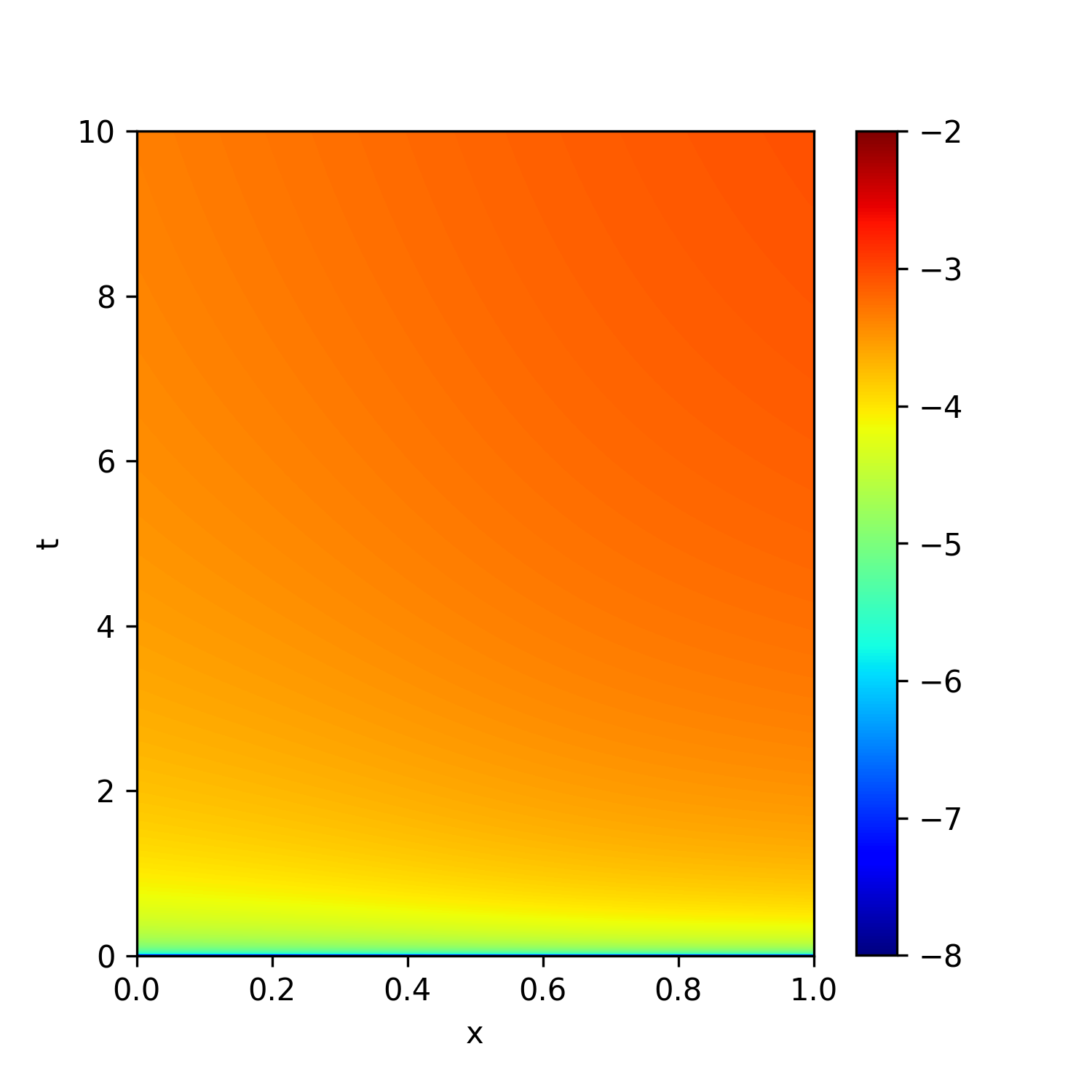}}
\qquad
\subfloat[Error between the surrogate POD approximation $u^\dagger$ and the POD reconstruction of $u$]{%
\centering\label{fig:proxy_to_proxyproxy_error} \includegraphics[width=0.35\textwidth]{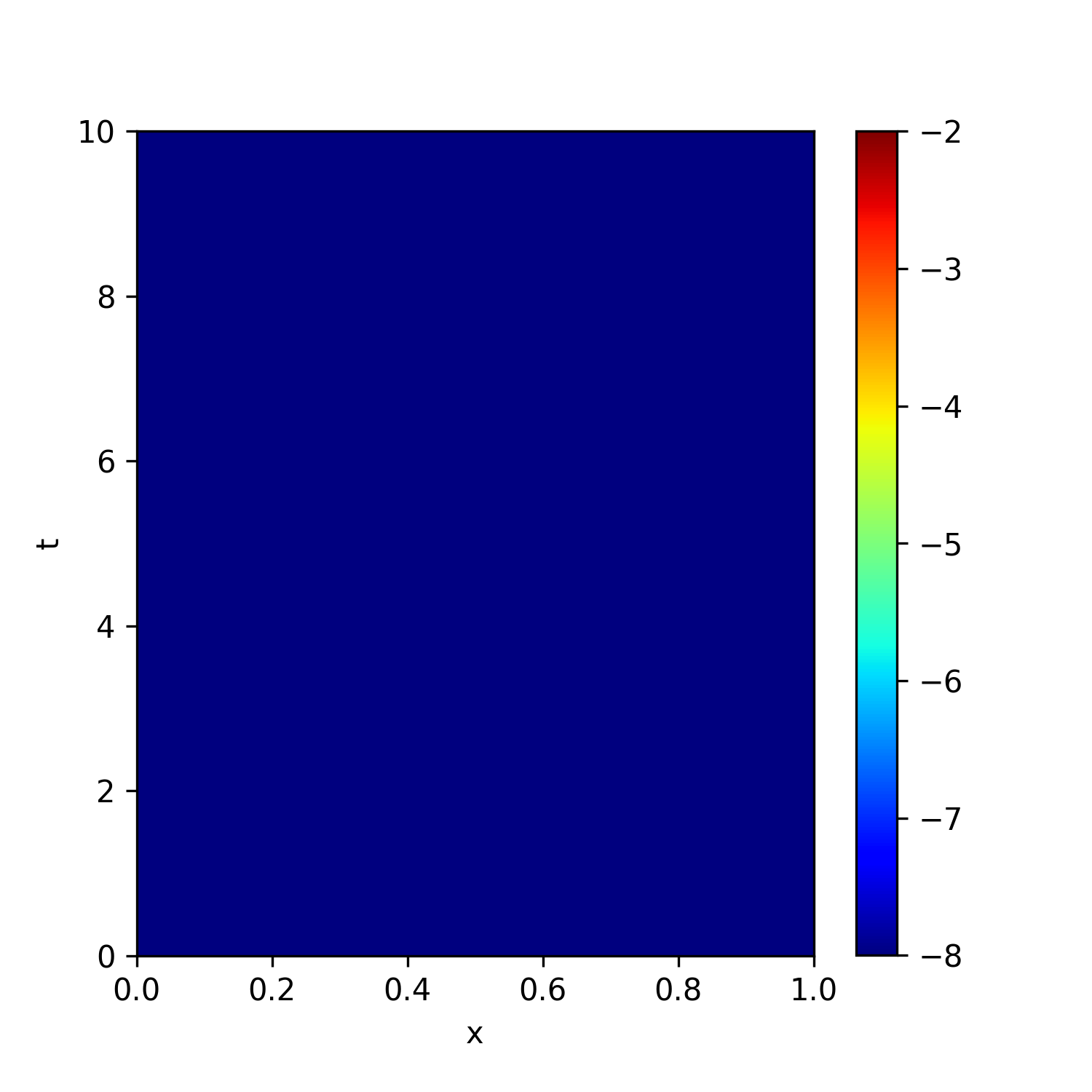}}
\caption{The error between $u$ and its proxy POD reconstruction (Figure~\ref{fig:proxy_pod_U_error}) and the error between the proxy POD reconstruction and the surrogate solution $u^\dagger$ (Figure~\ref{fig:proxy_to_proxyproxy_error}) at each $x$ and $t$ in a logarithmic scale. }
\label{fig:proxy_error_comparison}
\end{figure}

\subsubsection{Exact POD modes of 1D diffusion equation with discontinuous diffusivity}
\label{sec:true_podmodes_varying_beta}

In this section, we explore the case when the diffusivity factor $\beta$ is a discontinuous function in $x$, specifically we consider
\begin{equation}\label{eq:discont_beta}
\beta(x) = 5\times 10^{-3}\cdot H_{0.5}(x)\:
\quad\text{with}\quad
H_{0.5}(x)=
    \begin{cases}
        1 & \text{if } x \leq 0.5\\
        0 & \text{if } x >0.5
    \end{cases}\:.
\end{equation}
In Figure~\ref{fig:POD_modes_discont_beta}, we plot the first two spatial and temporal modes for the solution $u(x,t)$ to Eq.~\eqref{eq:1D_diffusion} with $\beta(x)$ given by Eq.~\eqref{eq:discont_beta}. 
As shown in the figure, the exact temporal modes are continuous while the spatial modes exhibit a discontinuity, which is due to the discontinuity of $\beta$ in space.

\begin{figure}[h]
\centering \includegraphics[width = .8\textwidth]{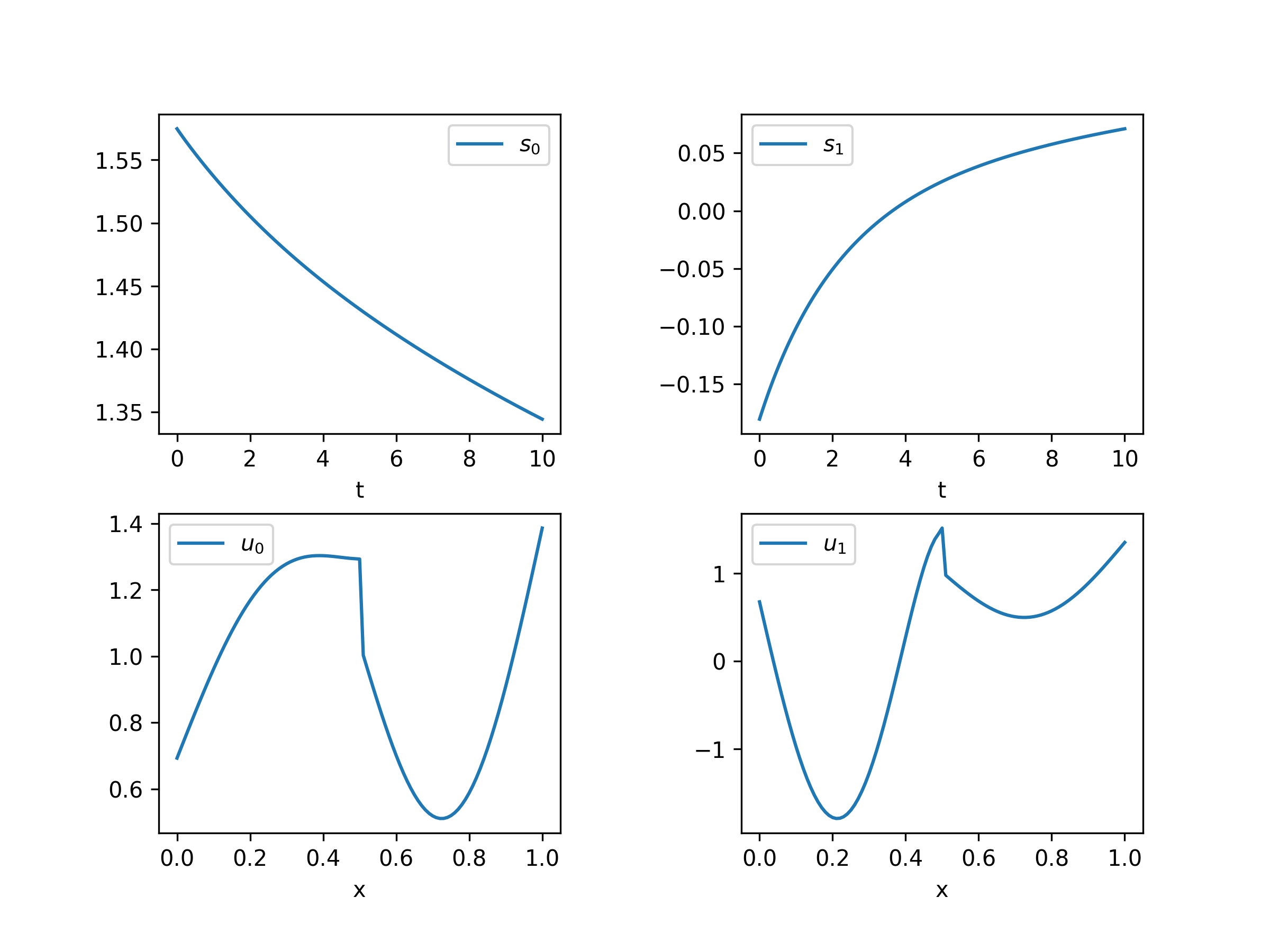}
\caption{The first two spatial and temporal modes of the solution $u(x,t)$ to Eq.~\eqref{eq:1D_diffusion} with $\beta(x)$ given by Eq.~\eqref{eq:discont_beta}. As expected, since $\beta(x)$ is spatially discontinuous, we see a discontinuity only in the spatial modes.}
\label{fig:POD_modes_discont_beta}
\end{figure}

We apply the weak-SINDy method and generate two surrogate models for the first two exact temporal modes (shown in Figure~\ref{fig:POD_modes_discont_beta}). 
The two surrogate models are of max degree one $(J=1)$ and three $(J=2)$, respectively. 
Figure~\ref{fig:discont_beta_originalsoln_pod_approx_surrogates} shows the original solution $u$ with discontinuous $\beta$ in Eq.~\eqref{eq:discont_beta}, the POD reconstruction of $u$, and the surrogate solutions $u^\dagger$ from surrogate models with $J=1$ and $J=2$.

\begin{figure}[h]
\centering
\subfloat[Original solution $u$]{%
\centering\label{fig:original_soln_discont_beta}\includegraphics[width=0.2\textwidth]{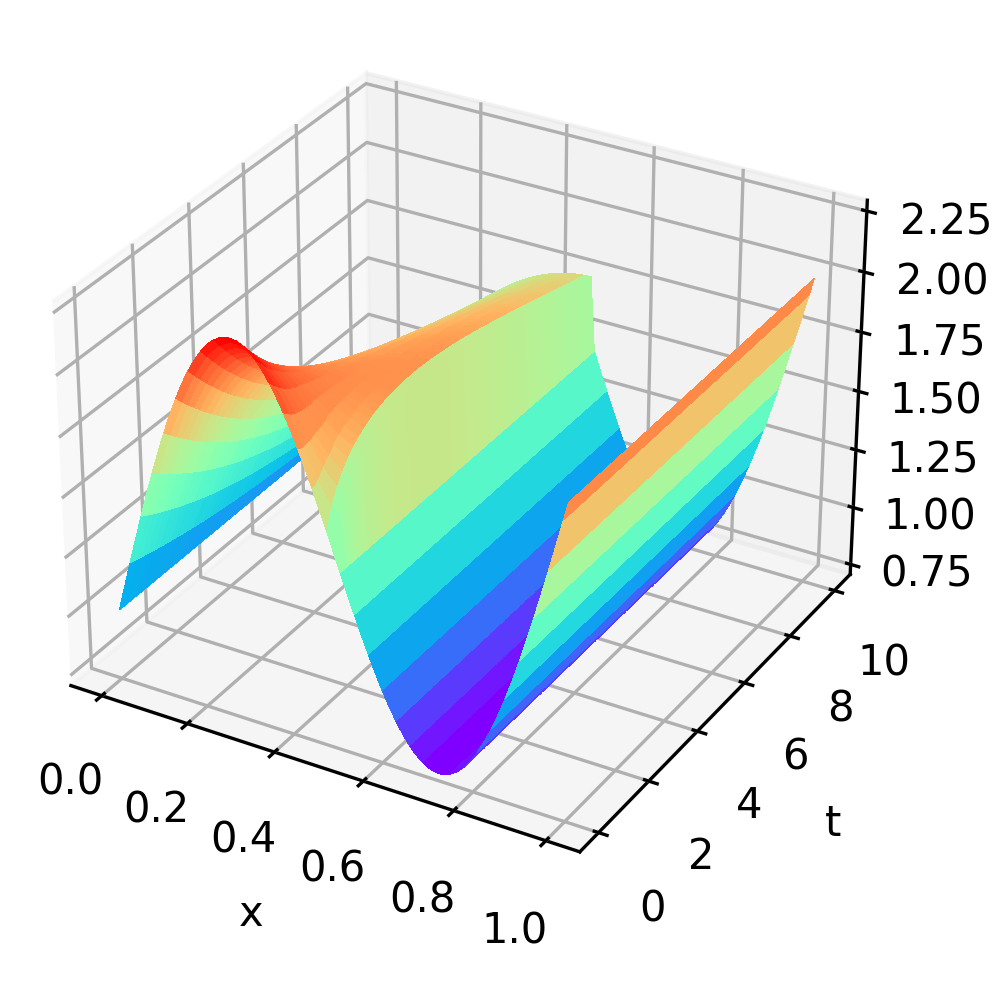}}
\qquad
\subfloat[POD reconstruction of $u$]{%
\centering\label{fig:PODapprox_discont_beta}\includegraphics[width=0.2\textwidth]{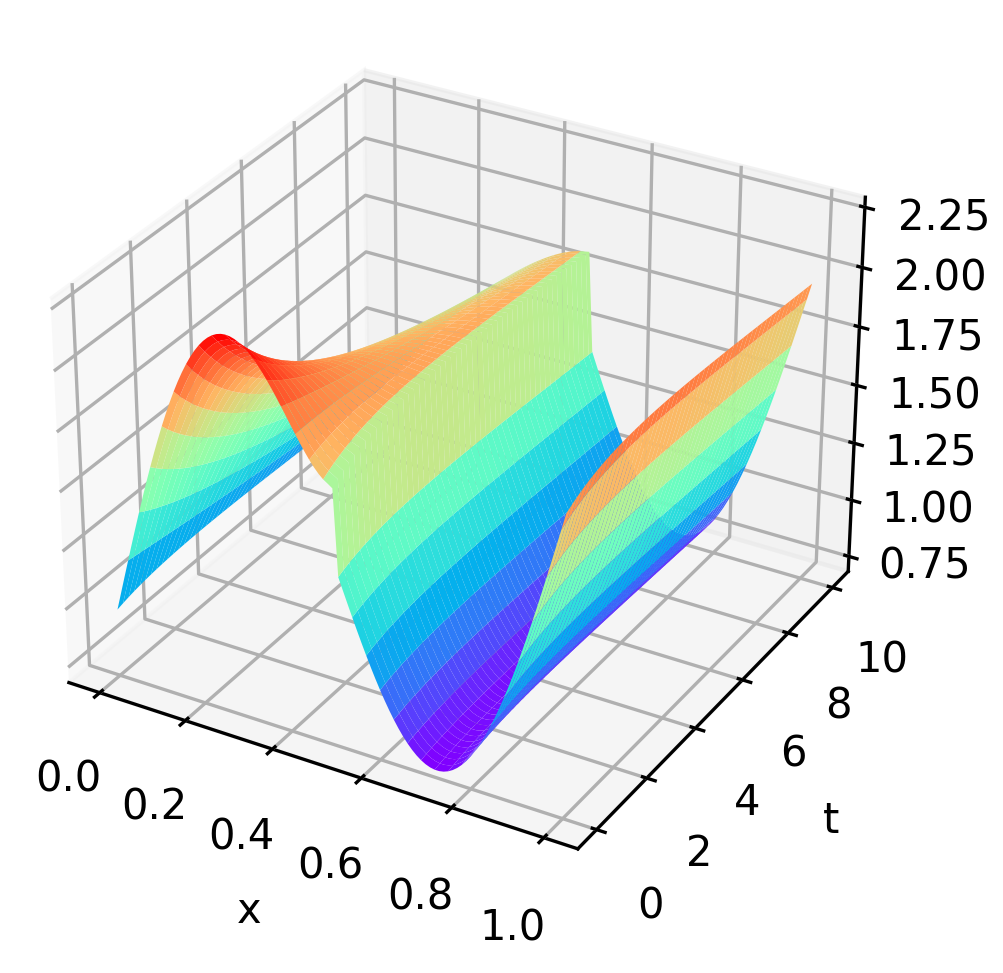}}
\qquad
\subfloat[Surrogate solution $u^\dagger$ with $J=1$]{%
\centering\label{fig:deg1_surrogate_discont_beta}\includegraphics[width=0.2\textwidth]{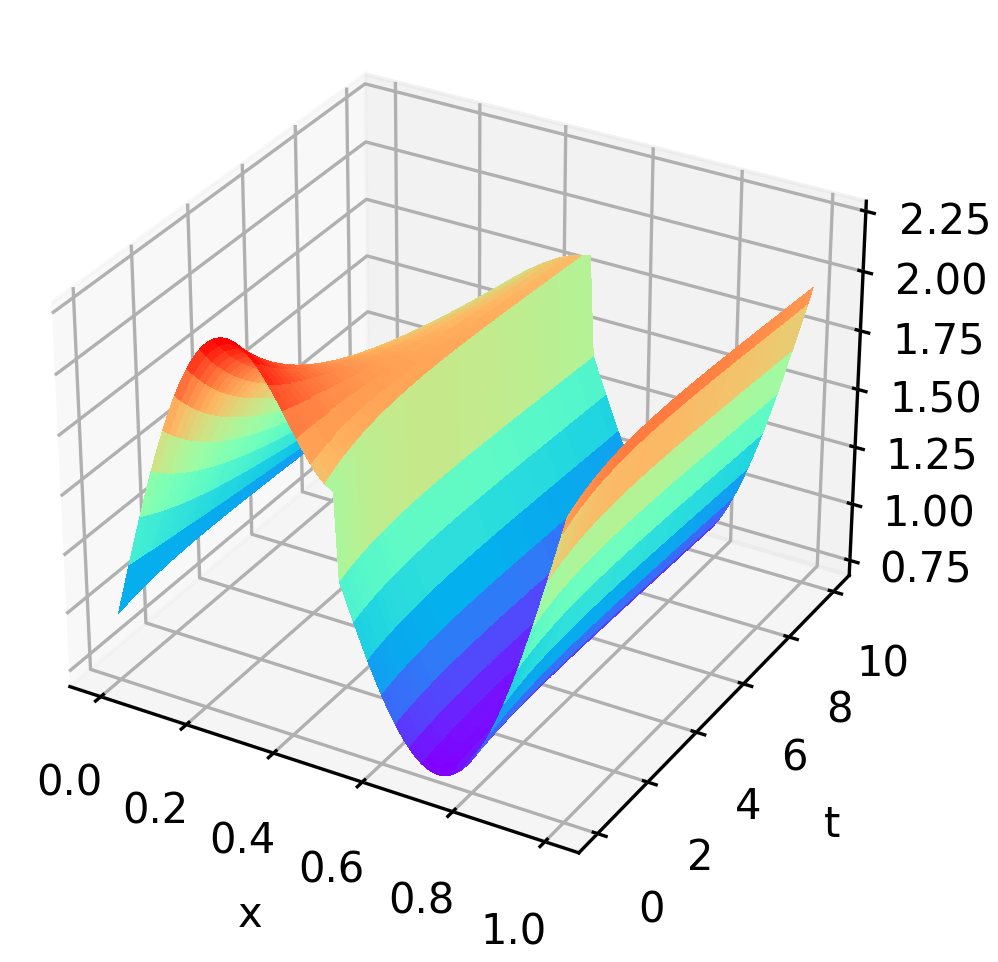}}
\qquad
\subfloat[Surrogate solution $u^\dagger$ with $J=2$]{%
\centering\label{fig:deg2_surrogate_discont_beta}\includegraphics[width=0.2\textwidth]{numerical_experiments/discont_beta_1deg_surrogate.png}}

\caption{The original solution $u(x,t)$ to Eq.~\eqref{eq:1D_diffusion} with discontinuous $\beta$ in Eq.~\eqref{eq:discont_beta}, the POD reconstruction of $u$, and POD approximations $u^\dagger$ from surrogate temporal modes $s_i^\dagger$ with max degree $J=1$ and $J=2$. }
\label{fig:discont_beta_originalsoln_pod_approx_surrogates}
\end{figure}

In Figure~\ref{fig:error_discont_beta_originalsoln_pod_approx_surrogates}, we plot the error between $u$ and its POD reconstruction (Figure~\ref{fig:error_PODapprox_discont_beta}), the error between the POD reconstruction of $u$ and $u^\dagger$ with $J=1$ (Figure~\ref{fig:error_deg1_surrogate_discont_beta}), and the error between the POD reconstruction of $u$ and $u^\dagger$ with $J=2$ (Figure~\ref{fig:error_deg2_surrogate_discont_beta}).
It can be observed from Figure~\ref{fig:error_PODapprox_discont_beta} that the POD reconstruction is not nearly accurate as in the case for constant $\beta$, due to the discontinuity of the solution in space.
However, the results in Figures~\ref{fig:error_deg1_surrogate_discont_beta} and \ref{fig:error_deg2_surrogate_discont_beta} indicate that the difference between the POD reconstruction and the surrogate solution $u^\dagger$ is still fairly small, especially when the max degree increases from $J=1$ to $J=2$, as expected from the error estimate in Theorem~\ref{multivar convergence theorem}.

\begin{figure}
\centering
\subfloat[Error between $u$ and its POD reconstruction]{%
\centering\label{fig:error_PODapprox_discont_beta}\includegraphics[width=0.27\textwidth]{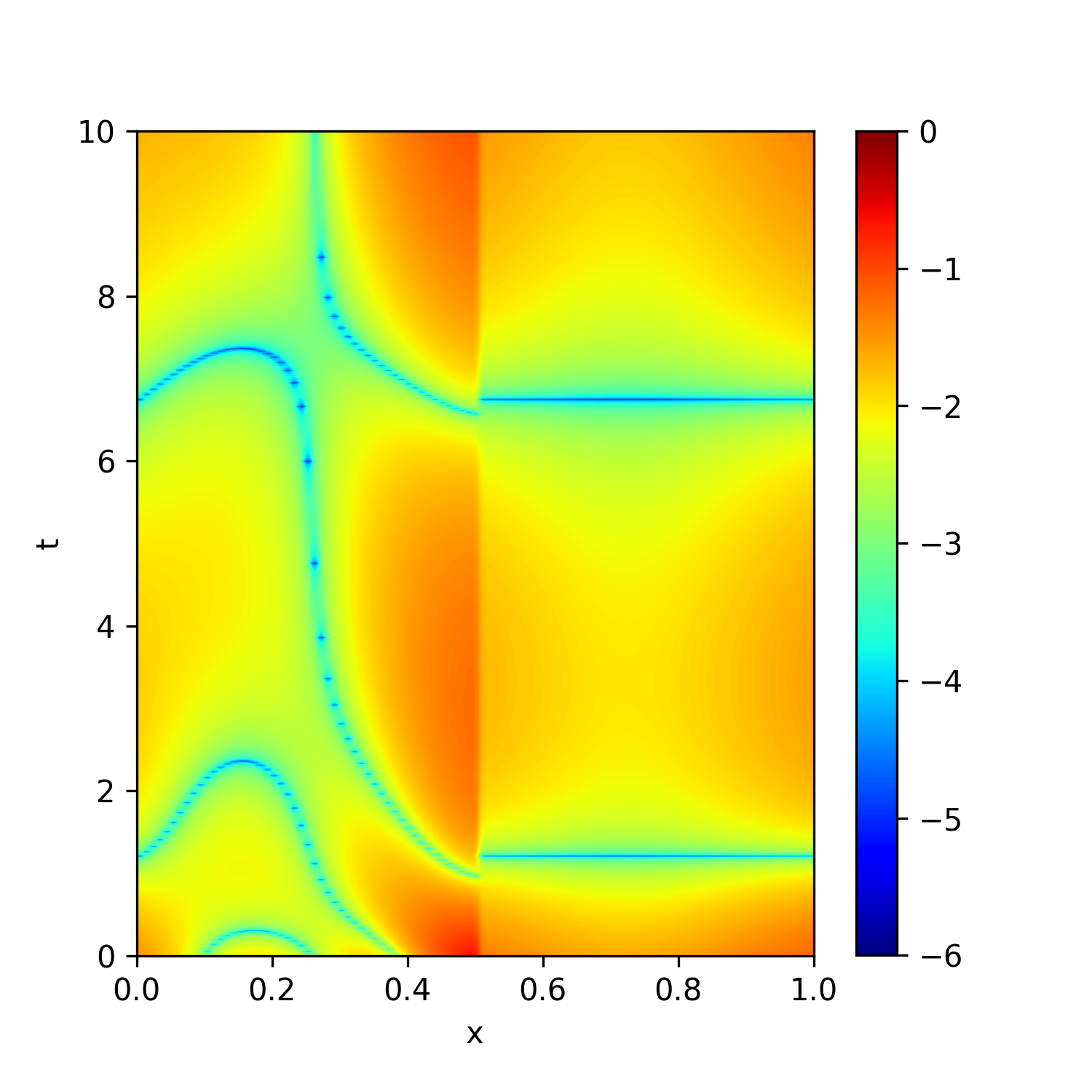}}
\qquad
\subfloat[Error between the POD reconstruction and $u^\dagger$ with $J=1$]{%
\centering\label{fig:error_deg1_surrogate_discont_beta}\includegraphics[width=0.27\textwidth]{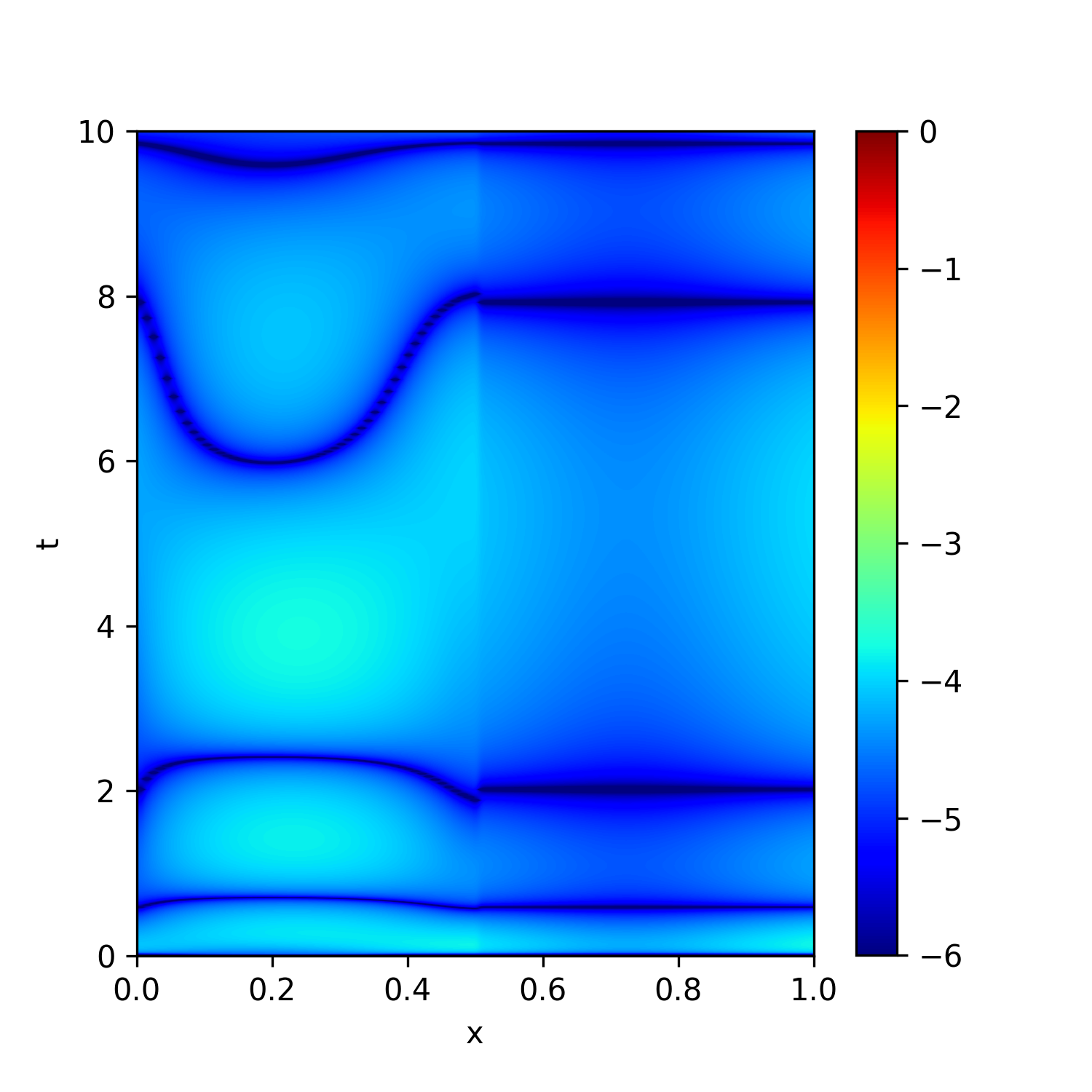}}
\qquad
\subfloat[Error between the POD reconstruction and $u^\dagger$ with $J=2$]{%
\centering\label{fig:error_deg2_surrogate_discont_beta}\includegraphics[width=0.27\textwidth]{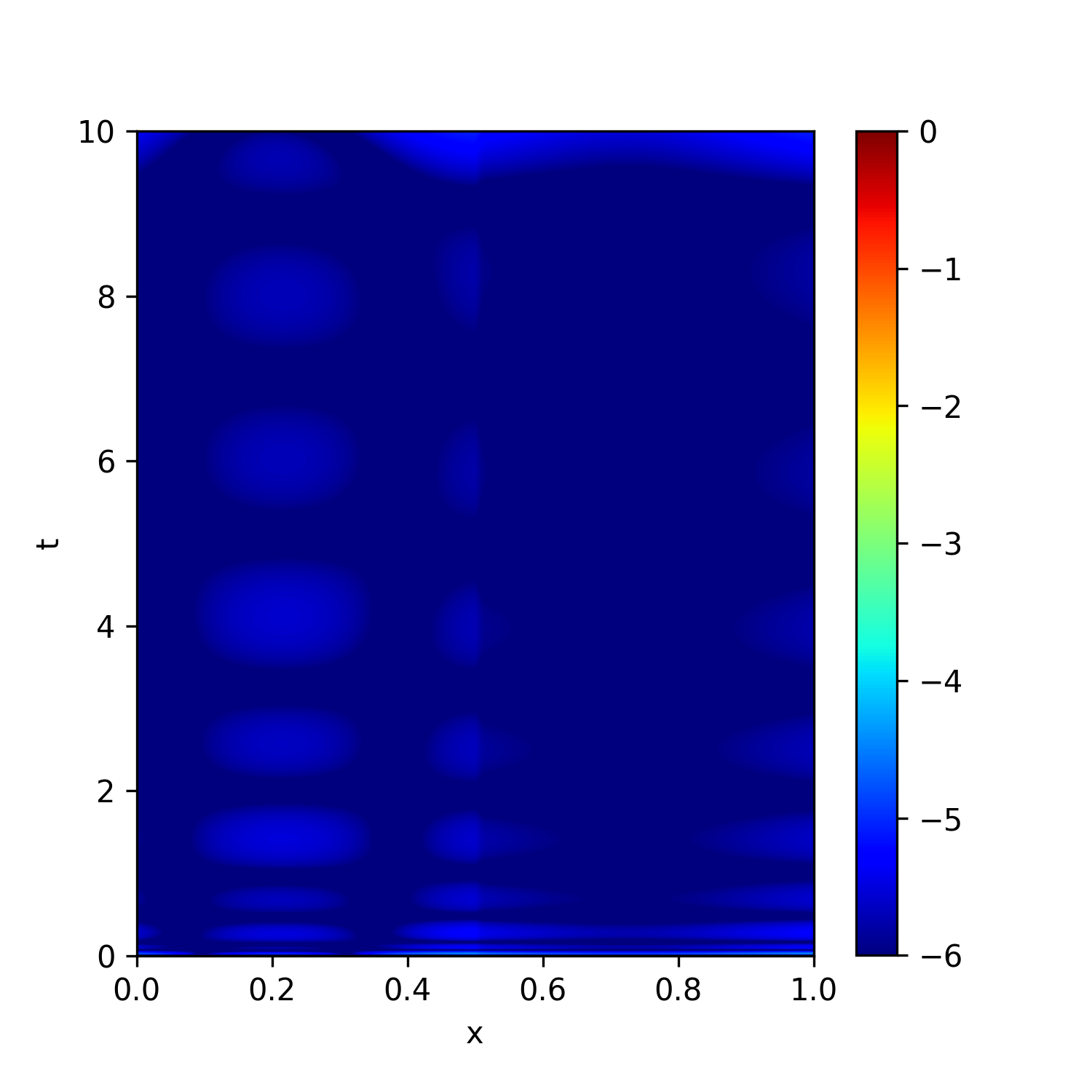}}

\caption{The error between $u$ and its exact POD reconstruction (Figure~\ref{fig:error_PODapprox_discont_beta}) for discontinuous $\beta$ and the error between the exact POD reconstruction and the approximate solution $u^\dagger$ from surrogate models with $J=1$ and $J=2$ (Figures~\ref{fig:error_deg1_surrogate_discont_beta} and \ref{fig:error_deg2_surrogate_discont_beta}) at each $x$ and $t$ in a logarithmic scale. }
\label{fig:error_discont_beta_originalsoln_pod_approx_surrogates}
\end{figure}

\section{Conclusions}
In this paper, error estimates for the surrogate models generated by the weak-SINDy technique with polynomial projection basis are provided. 
It was proved that (i) the dynamics of surrogate models generated by the weak-SINDy technique converges to the underlying dynamics as the basis degree increases and (ii) the surrogate solutions are reasonably close to the original solutions. 
The error estimates also provide a convergence rate of the surrogate dynamics that depends on the regularity of the dynamics and the resulting solutions. 
The analysis relies on the assumption that weak-SINDy procedure leads to a linear system with a unique solution, implying that the number of test functions $K$ needs to be greater than the number of projection basis $J$, which is verified in the numerical results.
Numerical experiments of applying the weak-SINDy technique to generate surrogate models for POD approximations to PDEs were performed.
The reported results confirm that, when the exact POD modes are unavailable (as in most practical scenarios), the error introduces by POD approximation often outweigh the surrogate modeling error, justifying the use of weak-SINDy surrogate models for POD approximation.

\section{Acknowledgements}
The authors would like to sincerely thank Konstantin Pieper for many illuminating conversations and the referees for a careful reading of the manuscript, which has greatly improved the quality. 

\appendix
\section{Proof of Proposition \ref{prop:pjk_converges_to_qj}}
\label{sec:appendix}
In this section we give the full details to show that the polynomial $p^*_{J,K}$ converges to the polynomial $q_J$. As the proof will rely on the interplay of the matrix and functional form of the minimization problems, we will define the following maps for convenience. 

\begin{definition} Let $x:[a,b]\rightarrow X$, $f:X\rightarrow \mathbb{R}$, $\{\psi_k\}_{k=0}^K$ be a finite subcollection of a polynomial orthonormal basis for $L^2([a,b])$ and $\projK$ be projection onto $\operatorname{span}\{\psi_k\}_{k=0}^K$.  Given a polynomial $p\in \mathbb{P}_J(X)$ we define the maps \[\chi:\mathbb{R}^{J+1}\rightarrow \mathbb{P}_J(X),\quad \quad [a_0, \ldots ,  a_J]^\top \in \mathbb{R}^{J+1}\mapsto \mathbb{P}_J(X)\ni a_0+a_1x + \ldots + a_{J}x^J,\]
\[\rho_{f,K}:\mathbb{P}_{J}(X)\rightarrow \mathbb{R}, \quad \quad \mathbb{P}_J(X)\ni p \mapsto \|\projK(f\circ x - p \circ x)\|^2_{L^2([a,b])},\]
\[\rho_{f,\infty}:\mathbb{P}_{J}(X)\rightarrow \mathbb{R}, \quad \quad \mathbb{P}_J(X)\ni p \mapsto \|f\circ x - p \circ x\|^2_{L^2([a,b])},\]
and 
\[M_{f,K}:\mathbb{R}^{J+1}\rightarrow \mathbb{R}, \quad \quad M_{f,K}:= \rho_{f,K} \circ \chi\]

\[M_{f,\infty}:\mathbb{R}^{J+1}\rightarrow \mathbb{R}, \quad \quad M_{f,\infty}:= \rho_{f,\infty} \circ \chi.\]

\end{definition}
Note that the coefficients of $p_{J,K}^*$, as a vector, is a minimizer of $M_{f,K}$ and the coefficients of $q_J$, as a vector, is a minimizer for $M_{f,\infty}$.

\begin{lemma}\label{lem:coefficientbound} If $p_{J,K}^*$ is unique for all $K$ and a fixed $J\in \mathbb{N}$, the sequence of vectors $\chi^{-1}(p^*_{J,K})$ is bounded in $\mathbb{R}^{J+1}$.
\end{lemma}
\begin{proof}
Let the $(k,j)$-th entry of $\bm{G}^{K}\in\mathbb{R}^{(K+1)\times (J+1)}$ be $\ip{\varphi_j\circ x}{\psi_k}$.
Let $\bm{g}^{K+1}$ be the $(K+1)$-th row of the matrix $\bm{G}^{K+1}$, i.e. 
 \[\bm{G}^{K+1} = \left[ \begin{array}{c}\bm{G}^K \\ \bm{g}^{K+1}\end{array} \right].\]
Assume for some $K_0$ that $\bm{G}^{K_0}$ is of full rank (i.e. a unique minimizer is assumed). 
Now, 
\begin{align*}
    \bm{w}^\top(\bm{G}^{K_0+1})^\top(\bm{G}^{K_0+1})\bm{w} &= \bm{w}^\top(\bm{G}^{K_0})^\top(\bm{G}^{K_0})\bm{w}+ \bm{w}^\top(\bm{g}^{K_0+1})^\top(\bm{g}^{K_0+1})\bm{w} \\
    &\geq\bm{w}^\top(\bm{G}^{K_0})^\top(\bm{G}^{K_0})\bm{w}>0
\end{align*}
by the full rank assumption. Therefore, $(\bm{G}^{K})^\top(\bm{G}^{K})$ is positive definite for all $K\geq K_0$ and the eigenvalues are bounded away from zero for all $K\geq K_0$. For all $K\geq K_0$, let $\bm{w}^K\in \mathbb{R}^{J+1}$ be such that $(\bm{G}^K)^\top(\bm{G}^K)\bm{w}^K = (\bm{G}^K)^\top \bm{b}^K$ where $\bm{b}^K=[\ip{f\circ x}{\psi_0}, \ldots, \ip{f\circ x}{\psi_K}]^\top.$ Note, $\bm{w}^K = \chi^{-1}(p_{J,K}^*)$. 
\[
    \|(\bm{G}^K)^\top\bm{G}^K \bm{w}^K\|_2 = \|(\bm{G}^K)^\top \bm{b}^k\|_2
\]
Therefore, 
\[\|\bm{w}^K\|_2 \leq \frac{\|(\bm{G}^{K})^\top\|_2 \|\bm{b}^K\|_2}{\lambda_{\text{min}}((\bm{G}^K)^\top\bm{G}^K)}. \]
Here $\|\cdot \|_2$ refers to the vector $2$-norm and the induced matrix norm and $\lambda_{\text{min}}(\cdot)$ refers to the minimum eigenvalue. From the definitions of $\bm{G}^{K}$ and $\bm{b}^{K}$, it is straightforward to see that $\|(\bm{G}^{K})^\top\|_2 \|\bm{b}^K\|_2$ is bounded for all $K$ by 
$ \sqrt{\sum_{j=1}^J \|\varphi_j \circ x\|^2_{L^2([a,b])}} \cdot \|f\circ x\|_{L^2([a,b])}$ by orthogonality of the test functions $\{\psi_k\}$, which proves the claim.
\end{proof}

\begin{lemma}\label{lem:uniformconvgence_K} If $f \in H^{m}(X)$, $x\in C^s([a,b])$, and $C_x:H^s(X)\rightarrow H^s([a,b])$ is bounded, then $M_{f,K}$ converges to $M_{f,\infty}$ uniformly as $K\rightarrow\infty$ over compact subsets of $\mathbb{R}^{J+1}$. 
\end{lemma}
\begin{proof}
By Lemma~\ref{lem:polyapproxindeg}
\begin{align*}
|M_{f,K}(\vec{p})-M_{f,\infty}(\vec{p})| &= \left|\|\projK(f\circ x - p)\|^2_{L^2([a,b])}- \|(f\circ x - p)\|^2_{L^2([a,b])}\right|\\
& = \|(f\circ x - p\circ x) - \projK(f\circ x - p\circ x) \|^2_{L^2([a,b])}\\
&\leq \frac{\|f\circ x - p\circ x\|^2_{H^{{s}}([a,b])}\cdot C}{K^{2s}} 
\end{align*}
for some constant $C$  independent of $f\circ x-p\circ x$ and $K$. In the above, $p=\chi(\vec{p})$. Note that the regularity of $f\circ x - p\circ x$ is the same for all polynomials our assumption. As our functions act on a compact subset, this assures that $\|f\circ x - p\circ x\|^2_{H^{s}([a,b])}\cdot C$ is bounded. Therefore, $|M_{f,K}(\vec{p})-M_{f,\infty}(\vec{p})|$ can be made arbitrarily small for any $\vec{p}$ by choosing $K$ large enough. 
\end{proof}
\begin{proposition}
Under the assumption that the maps $M_{f,K}$ and $M_{f,\infty}$ have unique minimizers, the maps $M_{f,K}$ and $M_{f,\infty}$ are strongly convex. 
\end{proposition}
\begin{proof}
To prove that $M_{f,K}$ is a strongly convex map, we note that 
\begin{equation}\label{eq:matrixform}
    M_{f,K}(\vec{p})= \|\projK(f\circ x-p\circ x) \|^2_{L^2} = \|\bm{G}^{K}\vec{p}-\bm{b}^K\|_2^2\, \quad p= \chi^{-1}(\vec{p}).
\end{equation}
Here it is clear that $2(\bm{G}^{K})^\top(\bm{G}^K)$ is the Hessian for $M_{f,K}$. Assuming the uniqueness of minimizers, the proof of Lemma~\ref{lem:coefficientbound} shows that $m_k=\lambda_{\min}((\bm{G}^{K})^\top(\bm{G}^K))$ is an increasing sequence bounded away from zero. This is enough to show strong convexity for all $M_{f,K}$ by using the matrix form of $M_{f,K}$ shown in Equation \eqref{eq:matrixform}. With strong convexity of $M_{f,K}$ established for all $K$, the proof of strong convexity for $M_{f,\infty}$ is done by a limit argument.
\end{proof}
\begin{lemma}\label{lem:strongconvexitylemma}
If $g:B\subset\mathbb{R}^{J+1}\rightarrow \mathbb{R}$ is strongly convex then there exists a constant $C$ such that if $x^*=\argmin_{x\in B}(g)$ and $\varepsilon>0$ then $|g(y)-g(x^*)|\leq \varepsilon$ implies $\|y-x^*\|^2_2\leq C\varepsilon$.
\end{lemma}
\begin{proof}
If $g$ is a strongly convex function this is equivalent to 
\[g(y) \geq g(x) + \ip{s_x}{y-x} +\frac{\mu}{2} \|y-x\|^2\]
where $s_x$ is a sub-differential at $x$. As $g$ is strongly convex, it has a unique minimizer, call it $x^*$. At $x = x^*$
\[g(y) - g(x^*) \geq \ip{s_{x^*}}{y-x^*} + \frac{\mu}{2}\|y-x^*\|^2.\]
As $x^*$ is the minimum we have that $0$ is a sub differential. Hence, 
\[g(y) -g(x^*)  \geq \frac{\mu}{2}\|y-x^*\|^2\]

\end{proof}

\begin{proposition}\label{prop:coefficientconvergence} Under the assumption of unique minimizers of $M_{f,K}$ and $M_{f\,\infty}$, the coefficients of $p_{J,K}^*$ converge to the coefficients of $q_J$, i.e. \[\lim_{K\rightarrow \infty}\chi^{-1}(p^*_{J,K}) = \chi^{-1}(q_J).\] 
\end{proposition}
\begin{proof}
Let $B\subset \mathbb{R}^{J+1}$ be a compact ball that contains $\vec{p}_{J,K}=\chi^{-1}(p^*_{J,K})$ for all $K$ and $\vec{q}_J  = \chi^{-1}(q_J)$. Since $M_{f,K}$ converges uniformly to $M_{f,\infty}$ over compact sets, given an $\varepsilon>0$ there exists a $K_\varepsilon$ such that 
\[|M_{f,K}(\vec{p})-M_{f,\infty}(\vec{p})|<\varepsilon\]
for all $K\geq K_\varepsilon$ and $\vec{p}\in B$.
By minimality, for all $K\geq K_\varepsilon$ we have that 
\[M_{f,K}(\vec{p}_{J,K})\leq M_{f,K}(\vec{q}_J)\leq M_{f, \infty}(\vec{q}_J)+\varepsilon\]
and
\[M_{f,\infty}(\vec{q}_{J})\leq M_{f,\infty}(\vec{p}_{J,K})\leq M_{f,K}(\vec{p}_{J,K})+\varepsilon.\]
By combining the above inequalities, 
\begin{equation}\label{eq:triangleside1}
|M_{f,\infty}(\vec{q}_{J}) - M_{f,K}(\vec{p}_{J,K})|\leq \varepsilon.    
\end{equation}

By uniform convergence over $B$, since $\vec{p}_{J,K}\in B$ we have
\begin{equation}\label{eq:triangleside2}
  |M_{f,K}(\vec{p}_{J,K}) - M_{f,\infty}(\vec{p}_{J,K})|\leq \varepsilon  
\end{equation}
for $K\geq K_\varepsilon$.
Combining Equations Eq.~\eqref{eq:triangleside1} and Eq.~\eqref{eq:triangleside2} gives us,
\begin{equation}\label{eq:triangleside3}
|M_{f,\infty}(\vec{p}_{J,K}) - M_{f,\infty}(\vec{q}_J)|\leq 2\varepsilon.  
\end{equation}
An application of Lemma~\ref{lem:strongconvexitylemma} gives us the result.
\end{proof}

\section{A generalized reverse Sobolev inequality}
\label{appendix:recurrence_relations}
In this appendix, we give a proof of a generalized reverse Sobolev inequality (Proposition \ref{prop:mv-generalized_reverse_sobolev}), valid for polynomials (either multi-variable or single variable) under composition with a mapping $\mathbf{x}(t)\in C^s([a,b])$. The proof uses the approach taken in \cite{canuto1982approximation, canuto2007spectral} towards proving the standard reverse Sobolev inequality (see Proposition \ref{lem:polyinverseineq} for reference). 

The statement we seek to prove is the following:
\begin{proposition}[Generalized reverse Sobolev inequality]
\label{prop:mv-generalized_reverse_sobolev}
Suppose $\mathbf{x}\in C^s([a,b])$ and the rank assumption (assumption~(iv) in Section~\ref{subsec:assumption}) holds. If $u$ is a $J$-th degree polynomial, then there exists a constant $\tilde{C}$ such that
\begin{equation}
\|u\circ \mathbf{x}\|_{H^{s}([a,b])}\leq \tilde{C} J^{2s} \|u\circ \mathbf{x}\|_{L^2([a,b])}.
\end{equation}
\end{proposition}
This is a generalized version from the standard reverse Sobolev inequality for polynomials: 
\begin{proposition}[Polynomial reverse Sobolev inequality]\label{lem:polyinverseineq}
Suppose $u\in\mathbb{P}_J(\Omega)$, then there exists a constant $C$ such that, for $0\leq \nu \leq \mu$,
\begin{equation}
\|u\|_{H^\mu(\Omega)}\leq C J^{2(\mu-\nu)}\|u\|_{H^\nu(\Omega)}\:.
\end{equation}
\end{proposition}
\begin{proof}
The proof can be found in \cite[Lemma 2.4]{canuto1982approximation}. 
\end{proof}
In \cite{canuto1982approximation}, the standard reverse Sobolev inequality is proven by using some derived recurrence relations on the coefficients of a polynomial when expressed in a normalized Legendre polynomial basis. Using this specific polynomial basis, we get an inequality in terms of the coefficients which in turn is an inequality on the norms. Hence our proof rewrites Proposition \ref{lem:polyinverseineq} back in terms of the coefficients (Lemma \ref{lem:coeff-inequality}). We then get bounds on the derivatives of a polynomial $u$ composed with $\mathbf{x}$ using spectral bounds on an auxiliary matrix $\bm{M}_{\varphi, J, \mathbf{x}}$ related to $\bm{G}$ and Lemma \ref{lem:coeff-inequality}. The full inequality is then proven by expanding via multi-variable Fa\`{a} di Bruno formula (see \cite{fraenkel_1978}) and applying the assumption that $\mathbf{x}\in C^s([a,b])$ which provides bounds on the derivatives. 

To begin, we adopt the notation of \cite{canuto1982approximation} and state the relation between the expansion coefficients of a normalized Legendre series and the ones of its derivative.

\begin{notation}The symbol $\sum^{j_b}_{j=j_a}\,^\prime$ (with $j_b$ possibly infinite) will denote the summation over all integers $j$ such that $j_a\leq j \leq j_b$ and $j-j_a$ is even. 
\end{notation}

\begin{lemma}
    Let $\varphi_j(y) = \lambda_jL_j$ the $j$-th normalized Legendre polynomial defined on $[-1,1]$. Let $u(y) = \sum_{j=0}^\infty \hat{u}_{j}\varphi_j(y)$ and $\partial_y u(y) = \sum_{j=0}^\infty \hat{u}^{(1)}_{j}\varphi_j(y)$, then 
    \begin{equation}
    \label{eq:coeff_relation}
    \hat{u}_j^{(1)} = 2 \lambda_j \left(\sum^\infty_{p= j + 1}\,^\prime \lambda_p\hat{u}_p\right)\:.
    \end{equation}
\end{lemma}
\begin{proof}
The result follows from the equations contained Section 2.3.2 in \cite{canuto2007spectral} for the standard Legendre polynomials, which is then extended to the case of normalized Legendre polynomials in Equation (2.21) found in \cite{canuto1982approximation}.
\end{proof}

Using Equation~\eqref{eq:coeff_relation} in the multi-variable context allows for proving the following lemma. For the following lemma, we restate Proposition \ref{lem:polyinverseineq} in different terms. While Proposition \ref{lem:polyinverseineq} is written as an inequality of the polynomial norms, by orthonormality of the normalized Legendre polynomials, the proofs presented in \cite{canuto1982approximation, canuto2007spectral} can be interpreted as an inequality of the coefficients of a finite Legendre series (a polynomial) and the ones of its derivative. The lemma below is stated for multi-variable finite Legendre series with a multi-index indicated by a boldface letter. In the following, the notation $\hat{u}^{(q,i)}_{\bm{j}}$ refers to the coefficients of the $q$-th partial derivative of $u$ in the $i$-th coordinate and $|\bm{j}|_{\infty}$ refers to the max degree of a multi-index $\bm{j}$.
\begin{lemma}
\label{lem:coeff-inequality}
    Let $\varphi_{\bm{j}}(\bm{y})$ be a tensor product of $\lambda_{j_i} L_{j_i}(y)$ with $\lambda_{j_i} = \sqrt{j_i + \frac{1}{2}}$. Suppose that $u(\bm{y}) = \sum_{|\bm{j}|_{\infty}<J}\hat{u}_{\bm{j}} \varphi_{\bm{j}}(\bm{y})$ then 
    \begin{equation}
        \|D_i u\|^2_{L^2}\leq  CJ^4 \|u\|_{L^2}^2,
    \end{equation}
    i.e.,
    \begin{equation}
    \label{eq:first_derv_coeff}
    \sum_{|\bm{j}|_\infty <J} \hat{u}^{(1,i)}_{\bm{j}}\hat{u}^{(1,i)}_{\bm{j}}\leq C J^4 \sum_{|\bm{j}|_\infty <J} \hat{u}_{\bm{j}}\hat{u}_{\bm{j}}.
    \end{equation}
Further, 
 \begin{equation}
 \label{eq:q_derv_coeff}
 \sum_{|\bm{j}|_\infty <J} \hat{u}^{(q,i)}_{\bm{j}}\hat{u}^{(q,i)}_{\bm{j}}\leq C J^{4q} \sum_{|\bm{j}|_\infty <J} \hat{u}_{\bm{j}}\hat{u}_{\bm{j}}.
 \end{equation}
\end{lemma}
\begin{proof} As stated in \cite{canuto1982approximation} the result relies on the coefficient relations in Equation~\eqref{eq:coeff_relation} and follows the same structure as Lemma 2.1 in \cite{canuto1982approximation}. The general inequality \eqref{eq:q_derv_coeff} is a repeated application of \eqref{eq:first_derv_coeff}. 
\end{proof}
To extend the result to the case where the polynomial $u$ is composed with $\mathbf{x}$, we need the following definition. 
\begin{definition}
  Let $ \{\varphi_{\bm{j}}(\cdot)\}_{|\bm{j}|_{\infty}<J}$ be a basis of multi-variable normalized Legendre polynomials. Let $\bm{M}_{\varphi, J, \mathbf{x}}$ be the Gram matrix with entries given by 
\begin{equation}
    \bm{M}_{\varphi, J, \mathbf{x}}(\bm{j}, \bm{\ell}) = \ip{\varphi_{\bm{j}}(\mathbf{x}(t))}{\varphi_{\bm{\ell}}(\mathbf{x}(t))}_{L^2([a,b])},
\end{equation}     
    where some ordering has been placed on the multi-indices. 
\end{definition}
By definition of the Gram matrix $\bm{M}_{\varphi, J, \mathbf{x}}$, we have that if $u = \sum_{|\bm{j}|_{\infty}\leq J} \hat{u}_{j} \varphi_{\bm{j}}(\cdot)$ and $v =\sum_{|\bm{j}|_{\infty}\leq J} \hat{v}_{j} \varphi_{\bm{j}}(\cdot)$, then 
\begin{equation}
\ip{u\circ \mathbf{x}}{v\circ \mathbf{x}}_{L^2([a,b])} = \bm{u}^\top \bm{M}_{\varphi, J, \mathbf{x}} \bm{v},
\end{equation}
where $\bm{u}$ is a vector of coefficients $\hat{u}_{\bm{j}}$ and likewise for $\bm{v}$. In the following, we use the notation $C_{\mathbf{x}}\partial^q_i u:= (\partial^q_i u)\circ \mathbf{x}$ to avoid confusion with $\partial^q_i (u\circ \mathbf{x})$. 

\begin{lemma} 
\label{lem:matrix-derivative-bound}
Suppose that $\bm{M}_{\varphi, J, \mathbf{x}}$ is a positive definite matrix. If
$u(\bm{y}) = \sum_{|\bm{j}|_{\infty}<J}\hat{u}_{\bm{j}} \varphi_{\bm{j}}(\bm{y})$  and $\partial^q_i u(\bm{y}) = \sum_{|\bm{j}|_\infty <J} \hat{u}^{(q,i)}_{\bm{j}}\varphi_{\bm{j}}(\bm{y})$, then 
\begin{equation}
\|C_{\mathbf{x}} \partial^q_i u\|_{L^2([a,b])} \leq \sqrt{\frac{\lambda_{\text{max}}}{\lambda_{\text{min}}}} CJ^{2q}\|C_{\mathbf{x}} u\|_{L^2([a,b])},
\end{equation}
where $\lambda_{\text{max}}$ and $\lambda_{\text{min}}$ are the maximum and minimum eigenvalues of $\bm{M}_{\varphi, J, \mathbf{x}}$. 
\end{lemma}
\begin{proof} Let $\bm{u}$ and $\bm{u}_{(q,i)}$ be vectors with entries given by $\hat{u}_{\bm{j}}$ and $\hat{u}^{(q,i)}_{\bm{j}}$ respectively. By an application of Lemma \ref{lem:coeff-inequality} 
\begin{equation}
\|C_{\mathbf{x}} \partial^q_i u\|^2_{L^2([a,b])} = \bm{u}^\top_{(q,i)} \bm{M}_{\varphi, J, \mathbf{x}} \bm{u}_{(q,i)} \leq \lambda_{\text{max}} \bm{u}_{(q,i)}^\top \bm{u}_{(q,i)} \leq \lambda_{\text{max}} CJ^{4q}\bm{u}^\top\bm{u}.
\end{equation}
At the same time, 
\begin{equation}
\|C_{\mathbf{x}} u\|^2_{L^2([a,b])} = \bm{u}^\top \bm{M}_{\varphi, J, \mathbf{x}} \bm{u}
\geq \lambda_{\text{min}} \bm{u}^\top \bm{u}.
\end{equation}
Combining the above inequalities gives the intended result. 
\end{proof}

We now show that under the rank assumption (Section \ref{subsec:assumption}) we have that $\bm{M}_{\varphi, J, \mathbf{x}}$ is automatically positive definite. 
\begin{proposition}
\label{prop:M_posdef}
Under the rank assumption, $\bm{M}_{\varphi, J, \mathbf{x}}$ is positive definite. 
\end{proposition}
\begin{proof} Suppose $\bm{M}_{\varphi, J, \mathbf{x}}$ is not positive definite. There exists a nonzero vector $\bm{v}$ such that $\bm{v}^\top \bm{M}_{\varphi, J, \mathbf{x}}\bm{v} = 0$. Hence, $\|\nu\circ \mathbf{x}\|_{L^2([a,b])} = 0$ for $\nu(\bm{y}) = \sum_{\bm{j}} \hat{v}_{\bm{j}} \varphi_{\bm{j}}(\bm{y})$, where $\hat{v}_{\bm{j}}$ are the entries of $v$ and thus $\nu$ is a nonzero polynomial. Therefore, the existence of $\nu$ shows that the null space of $\bm{G}$ is non-empty and thus $\bm{G}$ cannot be full column rank which contradicts to the rank assumption. 
\end{proof} 

With the above in place, we can now prove the main result of this section. We first make a small remark on the strategy of the remaining part of our proof. 

\begin{remark} To get an idea of our proof strategy, consider the single variable version. We rewrite the derivative in terms of the composition operator. 
\begin{equation}
\frac{d}{dt}(u\circ x(t)) = C_{x}\left( \frac{d}{dy} u(y) \right) \dot{x}(t).\end{equation}
If we further differentiate we have
\begin{equation}
\frac{d^2}{dt^2}(u\circ x(t)) =  C_{x}\left( \frac{d}{dy} u(y) \right)  \ddot{x}(t) + C_{x}\left(\frac{d^2}{dy^2} u(y)\right)  (\dot{x}(t))^2.
\end{equation}
The generalization of the above is given by the multi-variable Fa\`{a} di Bruno formula. In the proof of Proposition~\ref{prop:mv-generalized_reverse_sobolev}, we bound the terms of the form $C_{\mathbf{x}} \partial^q_i u$ by applying Lemma \ref{lem:matrix-derivative-bound}, and the derivatives of $\mathbf{x}(t)$ are readily bounded from the regularity assumption of $\mathbf{x}(t)$.
\end{remark}

\begin{proof}[Proof of Proposition \ref{prop:mv-generalized_reverse_sobolev}] By applying the multi-variable Fa\`{a} di Bruno formula, we express $\frac{d^q}{dt^q}u(\mathbf{x})$ in terms of the sums and products of $C_{\mathbf{x}} \partial^q_i u$, $q\leq s$, and the time derivatives of $\mathbf{x}(t)$. Since each $x_i\in C^s([a,b])$, the time derivatives of $x_i$ are bounded over the compact interval. We apply both the bound in Lemma \ref{lem:matrix-derivative-bound} and the $L^\infty$ bound on the time derivatives of $x_i(t)$ for each $i$ to achieve the desired result when $s$ is an integer. This result can then be extended to the non-integer derivative case via standard interpolation arguments. 
\end{proof}

\bibliographystyle{plain}
\bibliography{refs}

\begin{thebibliography}{10}

\bibitem{brunton2016discovering}
Steven~L Brunton, Joshua~L Proctor, and J~Nathan Kutz.
\newblock Discovering governing equations from data by sparse identification of
  nonlinear dynamical systems.
\newblock {\em Proceedings of the National Academy of Sciences},
  113(15):3932--3937, 2016.

\bibitem{canuto2007spectral}
Claudio Canuto, M~Yousuff Hussaini, Alfio Quarteroni, and Thomas~A Zang.
\newblock {\em Spectral methods: fundamentals in single domains}.
\newblock Springer Science \& Business Media, 2007.

\bibitem{canuto1982approximation}
Claudio Canuto and Alfio Quarteroni.
\newblock Approximation results for orthogonal polynomials in sobolev spaces.
\newblock {\em Mathematics of Computation}, 38(157):67--86, 1982.

\bibitem{filters}
Alain {de Cheveigné} and Israel Nelken.
\newblock Filters: When, why, and how (not) to use them.
\newblock {\em Neuron}, 102(2):280--293, 2019.

\bibitem{dunford1988linear}
Nelson Dunford and Jacob~T Schwartz.
\newblock {\em Linear operators, part 1: general theory}, volume~10.
\newblock John Wiley \& Sons, 1988.

\bibitem{fraenkel_1978}
L.~E. Fraenkel.
\newblock Formulae for high derivatives of composite functions.
\newblock {\em Mathematical Proceedings of the Cambridge Philosophical
  Society}, 83(2):159–165, 1978.

\bibitem{harris2020array}
Charles~R. Harris, K.~Jarrod Millman, St{\'{e}}fan~J. van~der Walt, Ralf
  Gommers, Pauli Virtanen, David Cournapeau, Eric Wieser, Julian Taylor,
  Sebastian Berg, Nathaniel~J. Smith, Robert Kern, Matti Picus, Stephan Hoyer,
  Marten~H. van Kerkwijk, Matthew Brett, Allan Haldane, Jaime~Fern{\'{a}}ndez
  del R{\'{i}}o, Mark Wiebe, Pearu Peterson, Pierre G{\'{e}}rard-Marchant,
  Kevin Sheppard, Tyler Reddy, Warren Weckesser, Hameer Abbasi, Christoph
  Gohlke, and Travis~E. Oliphant.
\newblock Array programming with {NumPy}.
\newblock {\em Nature}, 585(7825):357--362, September 2020.

\bibitem{hastie2009elements}
Trevor Hastie, Robert Tibshirani, Jerome~H Friedman, and Jerome~H Friedman.
\newblock {\em The elements of statistical learning: data mining, inference,
  and prediction}, volume~2.
\newblock Springer, 2009.

\bibitem{holmes2012turbulence}
Philip Holmes, John~L Lumley, Gahl Berkooz, and Clarence~W Rowley.
\newblock {\em Turbulence, coherent structures, dynamical systems and
  symmetry}.
\newblock Cambridge university press, 2012.

\bibitem{kutz2016dynamic}
J~Nathan Kutz, Steven~L Brunton, Bingni~W Brunton, and Joshua~L Proctor.
\newblock {\em Dynamic mode decomposition: data-driven modeling of complex
  systems}.
\newblock SIAM, 2016.

\bibitem{leveque2002finite}
Randall~J LeVeque et~al.
\newblock {\em Finite volume methods for hyperbolic problems}, volume~31.
\newblock Cambridge university press, 2002.

\bibitem{ljung1998system}
Lennart Ljung.
\newblock System identification.
\newblock In {\em Signal analysis and prediction}, pages 163--173. Springer,
  1998.

\bibitem{luchtenburg2009introduction}
DM~Luchtenburg, BR~Noack, and M~Schlegel.
\newblock An introduction to the pod galerkin method for fluid flows with
  analytical examples and matlab source codes.
\newblock {\em Berlin Institute of Technology MB1, Muller-Breslau-Strabe}, 11,
  2009.

\bibitem{messenger2021weakpde}
Daniel~A Messenger and David~M Bortz.
\newblock Weak sindy for partial differential equations.
\newblock {\em Journal of Computational Physics}, 443:110525, 2021.

\bibitem{messenger2021weak}
Daniel~A Messenger and David~M Bortz.
\newblock Weak sindy: Galerkin-based data-driven model selection.
\newblock {\em Multiscale Modeling \& Simulation}, 19(3):1474--1497, 2021.

\bibitem{EEG}
Suresh Muthukumaraswamy.
\newblock High-frequency brain activity and muscle artifacts in meg/eeg: A
  review and recommendations.
\newblock {\em Frontiers in Human Neuroscience}, 7, 2013.

\bibitem{paulsen2016introduction}
Vern~I Paulsen and Mrinal Raghupathi.
\newblock {\em An introduction to the theory of reproducing kernel Hilbert
  spaces}, volume 152.
\newblock Cambridge University Press, 2016.

\bibitem{qian2022reduced}
Elizabeth Qian, Ionut-Gabriel Farcas, and Karen Willcox.
\newblock Reduced operator inference for nonlinear partial differential
  equations.
\newblock {\em SIAM Journal on Scientific Computing}, 44(4):A1934--A1959, 2022.

\bibitem{rosenfeld2022dynamic}
Joel~A Rosenfeld, Rushikesh Kamalapurkar, L~Gruss, and Taylor~T Johnson.
\newblock Dynamic mode decomposition for continuous time systems with the
  liouville operator.
\newblock {\em Journal of Nonlinear Science}, 32(1):1--30, 2022.

\bibitem{SCC.Rosenfeld.Kamalapurkar.ea2019a}
Joel~A. Rosenfeld, Rushikesh Kamalapurkar, Benjamin Russo, and Taylor~T.
  Johnson.
\newblock Occupation kernels and densely defined {L}iouville operators for
  system identification.
\newblock In {\em Proc. IEEE Conf. Decis. Control}, pages 6455--6460, December
  2019.

\bibitem{rosenfeld2019occupation}
Joel~A Rosenfeld, Benjamin Russo, Rushikesh Kamalapurkar, and Taylor~T Johnson.
\newblock The occupation kernel method for nonlinear system identification.
\newblock {\em arXiv preprint arXiv:1909.11792}, 2019.

\bibitem{russo2022liouville}
Benjamin~P Russo and Joel~A Rosenfeld.
\newblock Liouville operators over the hardy space.
\newblock {\em Journal of Mathematical Analysis and Applications},
  508(2):125854, 2022.

\bibitem{sirovich1987turbulence}
Lawrence Sirovich.
\newblock Turbulence and the dynamics of coherent structures. i. coherent
  structures.
\newblock {\em Quarterly of applied mathematics}, 45(3):561--571, 1987.

\bibitem{smagorinsky1963general}
Joseph Smagorinsky.
\newblock General circulation experiments with the primitive equations: I. the
  basic experiment.
\newblock {\em Monthly weather review}, 91(3):99--164, 1963.

\bibitem{tibshirani1996regression}
Robert Tibshirani.
\newblock Regression shrinkage and selection via the lasso.
\newblock {\em Journal of the Royal Statistical Society: Series B
  (Methodological)}, 58(1):267--288, 1996.

\bibitem{2020SciPy-NMeth}
Pauli Virtanen, Ralf Gommers, Travis~E. Oliphant, Matt Haberland, Tyler Reddy,
  David Cournapeau, Evgeni Burovski, Pearu Peterson, Warren Weckesser, Jonathan
  Bright, St{\'e}fan~J. {van der Walt}, Matthew Brett, Joshua Wilson, K.~Jarrod
  Millman, Nikolay Mayorov, Andrew R.~J. Nelson, Eric Jones, Robert Kern, Eric
  Larson, C~J Carey, {\.I}lhan Polat, Yu~Feng, Eric~W. Moore, Jake
  {VanderPlas}, Denis Laxalde, Josef Perktold, Robert Cimrman, Ian Henriksen,
  E.~A. Quintero, Charles~R. Harris, Anne~M. Archibald, Ant{\^o}nio~H. Ribeiro,
  Fabian Pedregosa, Paul {van Mulbregt}, and {SciPy 1.0 Contributors}.
\newblock {{SciPy} 1.0: Fundamental Algorithms for Scientific Computing in
  Python}.
\newblock {\em Nature Methods}, 17:261--272, 2020.

\bibitem{WEI1990177}
Musheng Wei.
\newblock {P}erturbation of the least squares problem.
\newblock {\em Linear Algebra and its Applications}, 141:177--182, 1990.

\bibitem{convergence_of_SINDy}
Linan Zhang and Hayden Schaeffer.
\newblock On the convergence of the sindy algorithm.
\newblock {\em Multiscale Modeling \& Simulation}, 17(3):948--972, 2019.

\end{thebibliography}

\end{document}